\documentclass[9pt]{amsart}
\usepackage{graphicx}
\usepackage{latexsym}
\usepackage{amssymb, latexsym}
\usepackage{amsmath}
\usepackage{mathrsfs}
\usepackage{adjustbox}
\usepackage{amsxtra}
\usepackage{amsthm}
\usepackage{url}
\usepackage{verbatim}
\usepackage{tikz-cd}

\usepackage{tikz}
\usetikzlibrary{shapes.misc, positioning}

\usepackage{float}

\usepackage[urlcolor=blue]{hyperref}

\newtheorem{theorem}{Theorem}[section]
\newtheorem{corollary}[theorem]{Corollary}

\newtheorem{lemma}[theorem]{Lemma}

\newtheorem{proposition}[theorem]{Proposition}

\newtheorem*{theorem*}{Theorem}

\theoremstyle{definition}
\newtheorem{definition}[theorem]{Definition}

\newtheorem{question}[theorem]{Question}

\newtheorem{remark}[theorem]{Remark}

\newcommand{\Coll}{\mathop{\rm Coll}}

\newcommand{\image}{\mathbin{\hbox{\tt\char'42}}}

\newcommand{\smallleq}{\mathrel{\mathchoice{\raise2pt\hbox{$\scriptstyle\leq$}}{\raise1pt\hbox{$\scriptstyle\leq$}}{\raise1pt\hbox{$\scriptscriptstyle\leq$}}{\scriptscriptstyle\leq}}}
\newcommand{\smalllt}{\mathrel{\mathchoice{\raise2pt\hbox{$\scriptstyle<$}}{\raise1pt\hbox{$\scriptstyle<$}}{\raise0pt\hbox{$\scriptscriptstyle<$}}{\scriptscriptstyle<}}}

\newcommand{\GCH}{{\rm GCH}}

\newcommand{\ZFC}{{\rm ZFC}}

\newcommand{\p}{\mathbb{P}}

\newcommand{\la}{\langle}
\newcommand{\ra}{\rangle}

\newcommand{\forces}{\Vdash}

\newcommand{\Los}{\L o\'s}

\newcommand{\restrict}{\upharpoonright}

\newcommand{\crit}{\text{crit}}

\newcommand{\dom}{\text{dom}}

\newcommand{\Ord}{{\rm Ord}}

\title{Cardinals of the $P_\kappa(\lambda)$-filter games}
\author{Tom Benhamou}

\thanks{This research  was supported by the National Science Foundation under Grant
No. DMS-2346680}
\address[Benhamou]{Department of Mathematics, Rutgers University, New Brunswick, NJ 08854-
8019, USA}
\email{tom.benhamou@rutgers.edu}
\author{Victoria Gitman}
\address[Gitman]{ Department of Mathematics, CUNY Graduate Center, New York, NY, USA}
\email{vgitman@gmail.com}
\subjclass[2010]{vgitman@gmail.com}

\begin{document}
\begin{abstract}
   We investigate forms of filter extension properties in the two-cardinal setting involving filters on $P_\kappa(\lambda)$. We generalize the filter games introduced by Holy and Schlicht in \cite{HolySchlicht:HierarchyRamseyLikeCardinals} to filters on $P_\kappa(\lambda)$ and show that the existence of a winning strategy for Player II in a game of a certain length can be used to characterize several large cardinal notions such as: $\lambda$-super/strongly compact cardinals, completely $\lambda$-ineffable cardinals, nearly $\lambda$-super/strongly compact cardinals, and various notions of generic super and strong compactness. We generalize a result of Nielson from \cite{NielsenWelch:games_and_Ramsey-like_cardinals} connecting the existence of a winning strategy for Player II in a game of finite length and two-cardinal indescribability. We generalize the result of \cite{MagForZem} to construct a fine $\kappa$-complete precipitous ideal on $P_\kappa(\lambda)$ from a winning strategy for Player II in a game of length $\omega$. Finally, we improve Theorems 1.2 and 1.4 from \cite{MagForZem} and partially answer questions Q.1 and Q.2 from \cite{MagForZem}.
\end{abstract}
\maketitle
\section{Introduction}
In an attempt to generalize reflection and compactness properties of first-order logic, Tarski \cite{Tarski:SomeProblems} discovered cardinal notions whose existence require axiomatic frameworks which are strictly stronger than $\ZFC$. The fundamental discovery of Tarski provided a connection between compactness for certain strong logics and the ability to extend filters to special ultrafilters. Tarski's result initiated a fruitful line of research, isolating new large cardinal notions:
\begin{itemize}
    \item (weak compactness) Suppose that $\kappa$ is inaccessible. Every $\kappa$-complete filter on a $\kappa$-algebra $\mathcal{A}$ ($\kappa$-complete sub-algebra of $P(\kappa)$ of size $\kappa$) can be extended to a $\kappa$-complete ultrafilter on $\mathcal{A}$  $\Longleftrightarrow$ $L_{\kappa,\kappa}$ ($L_{\kappa,\omega}$) satisfies the compactness theorem for theories of size $\kappa$.
    \item (strong compactness) Suppose that $\kappa$ is regular. Every $\kappa$-complete filter (on any set) can be extended to a $\kappa$-complete ultrafilter $\Longleftrightarrow$ $L_{\kappa,\kappa}$ ($L_{\kappa,\omega}$) satisfies the full compactness theorem.
\end{itemize}
The gap between weak compactness and strong compactness is quite large, especially considering Magidor's identity crises \cite{magidor:identityCrisis} that the first strongly compact cardinal can consistently be supercompact. Recently, many intermediate filter-extension-like principles have been studied. For example, Hayut \cite{YairSquare} analyzed level-by-level strong compactness, square principles, and subcompact cardinals which are tightly connected to the results of this paper. 

Weakly compact cardinals can alternatively be characterized by the existence of certain ultrafilters on powersets of $\kappa$ of $\kappa$-sized models of set theory. A \emph{$\kappa$-model} $M$ is an $\in$-model of size $\kappa$ satisfying $\ZFC^-$\footnote{The theory $\ZFC^-$ consists of all the $\ZFC$ axioms minus powerset, with the collection scheme in place of the replacement scheme and with the well-ordering principle in place of the axiom of choice.} that is closed under sequences of length less than $\kappa$. Canonical examples of $\kappa$-models are elementary substructures of $H_{\kappa^+}$ of size $\kappa$ that are closed under ${<}\kappa$-length sequences. Given a $\kappa$-model $M$, a \emph{weak $M$-ultrafilter} is an ultrafilter on $P(\kappa)^M$ that is $\kappa$-complete from the point of view of the model, namely, it is closed for sequences that are elements of $M$. A weak $M$-ultrafilter is an \emph{$M$-ultrafilter}, if it is additionaly normal from the point of view of the model: closed under diagonal intersections for sequences from $M$. It is a folklore result that an inaccessible cardinal $\kappa$ is weakly compact if and only if every $\kappa$-model $M$ has a weak $M$-ultrafilter if and only if it has an $M$-ultrafilter. This characterization is quite different from the extension-like properties considered above, as it only mentioned the existence of certain ultrafilters. A natural question is: can the two be combined? For instance, given a $\kappa$-model $N$ extending a $\kappa$-model $M$ with a weak $M$-ultrafilter $U$, can we find a weak $N$-ultrafilter $W$ extending $U$? Keisler and Tarski \cite{KeislerTarski} (see also \cite[Prop. 2.9]{HolySchlicht:HierarchyRamseyLikeCardinals}) showed that this filter extension property again characterizes weak compactness. However, if we remove the ``weakness" condition from the ultrafilters, then surprisingly the property becomes inconsistent, as shown by the second author \cite[Prop. 2.13]{HolySchlicht:HierarchyRamseyLikeCardinals}. This suggests a more delicate approach, a \textit{strategic} extension of an $M$-ultrafilter to an $N$-ultrafilter. Game variations of this filter extension property were considered first by Holy and Schlicht \cite{HolySchlicht:HierarchyRamseyLikeCardinals}. 
\begin{definition} Suppose that $\kappa$ is an inaccessible cardinal and $\delta\leq\kappa^+$ is an ordinal.\footnote{Here we use game names that correspond to our later notation as opposed to the names originally used by Holy and Schlicht.}
\begin{enumerate}
\item Let $wG_\delta(\kappa)$ be the following two-player game of perfect information played by the challenger and the judge. The challenger starts the game and plays a $\kappa$-model $M_0$ and the judge responds by playing an $M_0$-ultrafilter $U_0$. At stage $\gamma>0$, the challenger plays a $\kappa$-model $M_\gamma$, with $\{\la M_\xi,U_\xi\ra\mid \xi<\gamma\}\in M_\gamma$, elementarily extending his previous moves and the judge plays an $M_\gamma$-ultrafilter $U_\gamma$ extending $\bigcup_{\xi<\gamma}U_\xi$. The judge wins if she can continue playing for $\delta$-many steps. 
\item Let $G_\delta(\kappa)$ be an analogously defined game where we additionally require that the ultrapower of $M=\bigcup_{\xi<\delta}M_\xi$ by $U=\bigcup_{\xi<\delta}U_\xi$ is well-founded.
\end{enumerate}
\end{definition}
\begin{center}
\begin{tikzpicture}
\coordinate (a) at (-4,0);
\coordinate (b) at (4,0);
\coordinate (c) at (-3.5,0.5);
\coordinate (d) at (-3.5,-0.5);
\coordinate (e) at (-3.75,0);
\coordinate (f) at (-3.25,0);
\coordinate (g) at (-2.75, 0);
\coordinate (h) at (-2.25,0);
\coordinate (i) at (-1.75,0);
\coordinate (j) at (-1.25,0);
\coordinate (k) at (-0.75,0);
\coordinate (k0) at (0.25,0);
\coordinate (k1) at (0.75,0);
\coordinate (l) at (1.75,0);
\coordinate (m) at (2.25,0);
\coordinate (n) at (2.75,0);
\coordinate (o) at (3.25,0);

\draw (a)--(b);
\draw (c)--(d);

\node[below] at (e) {j};
\node[above] at (e) {c};
\node[above] at (f) {$M_0$};
\node[below] at (g) {$U_0$};
\node[above] at (h) {$M_1$};
\node[below] at (i) {$U_1$};
\node[above] at (j) {$M_2$};
\node[below] at (k) {$U_2$};
\node[above] at (l) {$M_\xi$};
\node[below] at (m) {$U_\xi$};
\node[above] at (k0){$\cdots$};
\node[below] at (k1){$\cdots$};
\node[above] at (n){$\cdots$};
\node[below] at (o){$\cdots$};
\end{tikzpicture}
\end{center}
In the original definition of the games, there is an additional parameter, a regular cardinal $\theta$, which restricts the challenger to play $\kappa$-models elementary in $H_\theta$. The parameter $\theta$ is important only for $\delta$ of cofinality $\omega$ because otherwise either player has a winning strategy in a game $G_\delta(\kappa)$ if and only if the same player has a winning strategy in the game with the $\kappa$-models elementary in $H_\theta$ \cite{HolySchlicht:HierarchyRamseyLikeCardinals}. 
The weak game $wG_\delta(\kappa)$ and the game $G_\delta(\kappa)$ are easily seen to be equivalent for games of length $\delta$ with $\text{cof}(\delta)\neq\omega$. We can weaken the game $(w)G_\delta(\kappa)$ by requiring the judge to play only weak $M$-ultrafilters and call the resulting game $(w)G^*_\delta(\kappa)$. The game $wG^*_\delta(\kappa)$ was considered in \cite{MagForZem} under the name \textit{Welch game}\footnote{Although in the Welch game, the challenger plays a $\kappa$-algebra rather than a $\kappa$-model, for all practical matters, the games are equivalent as the powerset of a $\kappa$-model is a $\kappa$-algebra and any $\kappa$-algebra can be absorbed into the powerset of a $\kappa$-model.}. 

Ultimately, several large cardinals in the interval between weakly compact cardinals and measurable cardinals have been characterized by the existence of a winning strategy for the judge in one of the filter games. The following list is a partial account of characterizations of this kind. Suppose that $\kappa$ is inaccessible:
\begin{itemize}   
    \item $\kappa$ is weakly compact if and only if the judge has a winning strategy in the game $G^{(*)}_1(\kappa)$ (see the previous paragraph).
    \item (Keisler-Tarski \cite{KeislerTarski}) $\kappa$ is weakly compact if and only if the judge has a winning strategy in the game $wG^*_\omega(\kappa)$.
    \item (Nielsen \cite[Theorem 3.4]{NielsenWelch:games_and_Ramsey-like_cardinals}) If the judge has a winning strategy in the game $G_n(\kappa)$ for some $1\leq n<\omega$, then $\kappa$ is $\Pi^1_{2n}$-describable and $\Pi^1_{2n-1}$-indescribable.
    \item (follows from \cite[Theorem 3.12]{NielsenWelch:games_and_Ramsey-like_cardinals}) The judge has a winning strategy in the game $wG_\omega(\kappa)$ if and only if $\kappa$ is completely ineffable.
    \item (\cite[Obs. 3.5]{HolySchlicht:HierarchyRamseyLikeCardinals}) Suppose $2^\kappa=\kappa^+$. Then $\kappa$ is measurable if and only if the judge has a winning strategy in the game $G^{(*)}_{\kappa^+}(\kappa)$.
\end{itemize}

In a recent work, Foreman, Magidor and Zeman \cite{MagForZem} continued this investigation and proved the following elegant result:
\begin{theorem}\label{thm:Precipintous measurable}The following are equiconsistent:
\begin{enumerate}
    \item There is an inaccessible cardinal $\kappa$ such that the judge has a winning strategy in the game $G_{\omega+1}^*(\kappa)$.
    \item There is a measurable cardinal.
\end{enumerate}
\end{theorem}
Their result is more delicate and involves the construction of ideals containing a dense closed tree.
\begin{theorem}\label{thm: dense tree measurable intro}
Assume that $\kappa$ is inaccessible, $2^\kappa=\kappa^+$, and that $\kappa$ does not carry a $\kappa$-complete $\kappa^+$-saturated ideal. Let $\omega<\delta<\kappa^+$ be a regular cardinal. If the judge has a winning strategy in
the game $G^*_\delta(\kappa)$, then there is a uniform normal ideal $\mathcal{I}$ on $\kappa$ and a set
$D \subseteq \mathcal{I}^+$ such that:
\begin{enumerate}
    \item $(D,\subseteq_\mathcal{I})$ is a downward growing tree of height $\delta$.
    \item $D$ is $\delta$-closed.
    \item  $D$ is dense in $\mathcal{I}^+$.
\end{enumerate}
In fact, it is possible to construct such a dense set $D$ where (1) and (2) above hold with the almost containment $\subseteq^*$
in place of $\subseteq_I$.    
\end{theorem}
They also proved a partial converse:
\begin{theorem}\label{thm: Partial converse}
    Let $\delta\leq\kappa$ be uncountable regular cardinals and $J$ be a $\kappa$-complete ideal on $\kappa$ which is $(\kappa^+,\infty)$-distributive\footnote{That is, the forcing $P(\kappa)/\mathcal{J}$ is $\kappa^+$-distributive.} and has a dense $\delta$-closed subset.
Then the judge has a winning strategy in the game $G^*_\delta(\kappa)$,  which is constructed in a natural way from the ideal $\mathcal{J}$.

\end{theorem} 
In their paper, the authors asked whether the filter games can be generalized to the two-cardinal setting \cite[Q.5]{MagForZem}.
In this paper we provide analogues of the filter games for filters on $P_\kappa(\lambda)$,  $(w)G_\delta(\kappa,\lambda)$ and $(w)G^*_\delta(\kappa,\lambda)$ (see Definition \ref{Definition:Filter Games}). We also consider the strong game $sG_\delta(\kappa,\lambda)$ (introduced in the one-cardinal context in \cite{MagForZem}), where we require that the ultrafilter resulting from unioning up the judge's moves is countably complete. We add in the parameter $\theta$ when we need the challenger's moves to be elementary in some $H_\theta$, which as in the original filter games, affects only games of length $\delta$ with $\text{cof}(\delta)=\omega$. Often, the one-cardinal $\kappa$-theory of ultrafilters, turns out to be a particular case of the two-cardinal $(\kappa,\lambda)$-theory when considering the case $\kappa=\lambda$. In particular, uniform filters on $\kappa$ can be identified with fine filters on $P_\kappa(\kappa)$, and the notions of normality and completeness coincide. Here we will also generalize the notion of a $\kappa$-model (see Definition \ref{def: basicModel}). In that sense, the two-cardinal games we introduce generalize the one-cardinal games of Holy and Schlicht. We then prove that major parts of the theory of the one-cardinal filter games (and in particular all the results above) generalize to the two-cardinal settings.

When passing from ultrafilters on $\kappa$ to ultrafilters on $P_\kappa(\lambda)$, a distinction appears between the existence of  $\kappa$-complete ultrafilters and normal ultrafilters. Thus, even for games of length 1, it is expected that there will be a difference between the assumption that there exists a winning strategy for the judge in the game $G^*_1(\kappa,\lambda)$ and the game $G_1(\kappa,\lambda)$. We start with a simple observation that $\lambda$-supercompact/strongly compact cardinals play the role of measurable cardinals.  
\begin{theorem*} Assume $2^\lambda=\lambda^+$ and $\lambda^{<\kappa}=\lambda$. \begin{enumerate}
 \item The judge has a winning strategy in the game $G^*_{\lambda^+}(\kappa,\lambda)$ if and only if  $\kappa$ is $\lambda$-strongly compact.
    \item The judge has a winning strategy in the game $G_{\lambda^+}(\kappa,\lambda)$ if and only if  $\kappa$ is $\lambda$-supercompact.
\end{enumerate} \end{theorem*}

\subsection*{Finite levels of the game}
The role of weakly compact cardinals is filled by nearly $\lambda$-supercompact cardinals and nearly $\lambda$-strongly compact cardinals of Schankar and White respectively \cite{schanker:nearSupercompactness,white:nearlyStronglyCompact}:
\begin{theorem*} 
Assume $\lambda^{<\kappa}=\lambda$. 
\begin{enumerate}
  \item The judge has a winning strategy in the game $G^*_{1}(\kappa,\lambda)$ if and only if  $\kappa$ is nearly $\lambda$-strongly compact.
    \item The judge has a winning strategy in the game $G_{1}(\kappa,\lambda)$ if and only if  $\kappa$ is nearly $\lambda$-supercompact.
  
\end{enumerate}\end{theorem*}
By results of Hayut and Magidor \cite{HayutMagidor:subcompact}, these cardinals are tightly connected to $\lambda$-$\Pi^1_1$-subcompact cardinals.

Moving to longer games, more differences arise between the games $G_\delta(\kappa,\lambda)$ and $G^*_\delta(\kappa,\lambda)$ (as they do correspondingly in the one-cardinal games). For the games with weak $M$-ultrafilters, we can strengthen (1) of the above theorem to:
\begin{theorem*} The judge has a winning strategy in the game $wG^*_{\omega}(\kappa,\lambda)$ if and only if $\kappa$ is nearly $\lambda$-strongly compact.\end{theorem*}
In contrast, the existence of a winning strategy for the judge in the games $G_\delta(\kappa,\lambda)$ for $1<\delta<\omega$ gives a proper consistency strength hierarchy. We generalize Theorem 3.4 of \cite{NielsenWelch:games_and_Ramsey-like_cardinals} using Baumgartner's $\lambda$-$\Pi^1_n$-indescribable cardinals:

\begin{theorem*}
Assume $\lambda^{<\kappa}=\lambda$. 
   \begin{enumerate}
       \item  A winning strategy for the judge in the game $G_n(\kappa,\lambda)$ is expressible by a  $\Pi^1_{2n}$-formula.

   \item  If the judge has a winning strategy in the game $G_n(\kappa,\lambda)$, then  $\kappa$ is $\Pi^1_{2n-1}$-$\lambda$-indescribable.
   
   \end{enumerate}
\end{theorem*}
For the game $wG_\omega(\kappa,\lambda)$ we have a simple equivalence:
\begin{theorem*}
    The judge has a winning strategy in the game $wG_\omega(\kappa,\lambda)$ if and only if $\kappa$ is completely $\lambda$-ineffable.
\end{theorem*}

\subsection*{Generic supercompactness}
Generalizing \cite[Theorem 2.1.2]{AbramsonHarringtonKleinbergZwicker:FlippingProperties} on completely ineffable cardinals, we show that completely $\lambda$-ineffable cardinals can be characterized by a form of generic supercompactness, and thus, by the previous theorem, so does the existence of a winning strategy for the judge in the game $wG_\omega(\kappa,\lambda)$. A strong relation between a winning strategy for the judge and various forms of generic supercompactness persists for the stronger games $G_\omega(\kappa,\lambda,\theta)$ and $sG_{\omega}(\kappa,\lambda,\theta)$, where the union ultrafilter is required to produce a well-founded ultrapower. For the various notions of generic supercompactness we use here, see Section \ref{sec:genericSupercompactness}.
 The results below are inspired by analogous results of Nielsen and Welch \cite{NielsenWelch:games_and_Ramsey-like_cardinals}. 

Given a model $M$, we say that an ultrafilter $U$ on $P_\kappa(\lambda)^M$ is \emph{weakly amenable} if the restriction of $U$ to any set in $M$ of size at most $\lambda$ in $M$ is an element of $M$, that is, $M$ contains all sufficiently ``small" pieces of $U$.
\begin{theorem}$\,$
\begin{enumerate}
    \item If the judge has a winning strategy in the game $G_\omega(\kappa,\lambda,\theta)$, then in some set-forcing extension, there is a weakly amenable $H_{\theta}$-ultrafilter with a well-founded ultrapower. Thus, if the judge has a winning strategy in the game $G_\omega(\kappa,\lambda,\theta)$ for every regular $\theta\geq\lambda^+$, then $\kappa$ is generically $\lambda$-supercompact for sets with weak amenability.\item 
If in a set-forcing extension, there is an elementary embedding $j:H_\theta\to M$ with $\crit(j)=\kappa$, $j(\kappa)>\lambda$, $j\image\lambda\in M$, and $M\subseteq V$, then the judge has a winning strategy in the game $G_\omega(\kappa,\lambda,\theta)$.
\end{enumerate}
\end{theorem}
Although $(1)$ and $(2)$ above are almost converses of each other, it is unclear to us how to get an exact equivalence.

\begin{theorem*}
If the judge has a winning strategy in the game $sG_\omega(\kappa,\lambda)$, then $\kappa$ is generically $\lambda$-supercompact with weak amenability and $\omega_1$-iterability.
\end{theorem*}
Above $\omega$, we have the following generic supecompactness equivalence:
\begin{theorem*}
The following are equivalent for a cardinal $\kappa$ and an uncountable regular cardinal $\delta\leq\lambda$.

\begin{enumerate}
\item $\kappa$ is generically $\lambda$-supercompact with weak amenability by $\delta$-closed forcing.
\item The judge has a winning strategy in the game $G_{\delta}(\kappa,\lambda)$.
\end{enumerate}
\end{theorem*}

\subsection*{Precipitous ideal and closed dense subtrees}

Finally, we prove a similar result to Foreman, Magidor and Zeman's Theorem~\ref{thm:Precipintous measurable}:
\begin{theorem*}
    If the judge has a winning strategy in the game $sG^*_\omega(\kappa,\lambda)$, then there is a precipitous ideal on $P_\kappa(\lambda)$. If the judge has a winning strategy in the game $sG_\omega(\kappa,\lambda)$, then we have, moreover, that the ideal is normal. 
\end{theorem*}
A major difference in our approach is that we do not pass through the game where we choose sets determining $M$-ultrafilters instead of $M$-ultrafilters. Instead, we construct a tree of $M$-ultrafilters and prove that this suffices to obtain a precipitous ideal. This approach can be used to slightly simplify the proof of Theorem \ref{thm:Precipintous measurable}. Of course, the observation of \cite{MagForZem} that one can move to a game where the judge plays sets determining ultrafilters is highly interesting on its own merit.

In the last section, we switch back to the one-cardinal setting and provide some additional information related to the results of \cite{MagForZem}. First, we show how to derive one more crucial property of the ideal constructed in Theorem \ref{thm: dense tree measurable intro}, which we call \textit{$\kappa$-measuring}. It is not difficult to see that a $\kappa$-complete ideal is $\kappa$-measuring if and only if forcing with $P(\kappa)/\mathcal{I}$ adds a weakly amenable ultrafilter on $P(\kappa)^V$. This is used to improve Theorem \ref{thm: Partial converse} to a full converse:
\begin{theorem*}
    Suppose $\delta\leq\kappa$ are uncountable regular cardinals and $\mathcal{I}$ is a $\kappa$-complete $\kappa$-measuring ideal on $\kappa$ with a $\delta$-closed dense subset $D$. Then the judge has a winning strategy $\sigma^D$ in the game $G^*_\delta(\kappa)$. 
\end{theorem*}
The strategy is constructed in the same way as in \ref{thm: Partial converse}

Finally, we address two other problems, Q.1 and Q.2 from \cite{MagForZem}. Q.1 asks whether any the assumptions of Theorem \ref{thm: dense tree measurable intro} can be dropped (see Question \ref{Question1} in \S~\ref{sec:precipitousideals}). We prove that the assumption about $\kappa$ carrying no $\kappa$-complete $\kappa^+$-saturated ideal of Theorem \ref{thm: dense tree measurable intro} is  necessary using the following theorem:
\begin{theorem*}
Suppose that $\mathcal{I}$ is a  $\kappa$-complete $\kappa$-measuring ideal on $\kappa$ and there is a tree $D\subseteq \mathcal{I}^+$ such that:
\begin{enumerate}
    \item $D$ is dense in $\mathcal{I}^+$.
    \item $(D,\subseteq_\mathcal{I})$ is a downward growing tree of height $\delta$.
    \item $D$ is $\delta$-closed.
\end{enumerate}
Then there is a winning strategy $\sigma$ for the game $wG^*_\delta(\kappa)$ such that for every partial run $R$ of the game played according to $\sigma$, the associated hopeless ideal $\mathcal{I}(R,\sigma)$ is not $\kappa^+$-saturated.
\end{theorem*}
Q.$2$ asks whether there is a correspondence between the ideal constructed in Theorem \ref{thm: dense tree measurable intro} and the strategy constructed in Theorem \ref{thm: Partial converse}. We show how one can slightly refine the construction of the dense tree $D$ in the previous theorem, and  build a special $D^*\subseteq D$ which gives such a correspondence:
\begin{theorem*}
    Let $D$ be a dense subtree of $\mathcal{I}$ satisfying $(1)$-$(3)$. Then the hopeless ideal associated to the strategy $\sigma^D$, $\mathcal{I}(\sigma^D)=\mathcal{I}$ if and only if $D=D^*$.
\end{theorem*}
This paper is organized as follows:
\begin{itemize}
    \item In Section \S~\ref{sec:modelsFilters}, we provide  two-cardinal analogues to $\kappa$-models $M$ and $M$-ultrafilters.
    \item In Section \S~\ref{sec:genericSupercompactness}, we define large cardinal variants of generic supercompactness arising from the games, and prove some relations between them.
    \item In Section \S~\ref{sec:games}, we define the two-cardinal filter games.
    \item In Section \S~\ref{sec: winning strategies}, we prove the nearly $\lambda$-super/strongly compact and completely $\lambda$-ineffable characterization, and also the results regarding generic supercompactness.
    \item In Section \S~\ref{sec:finitelength}, we prove the results regarding finite stages of the game and indescribability.
    \item In Section \S~\ref{sec:precipitousideals}, we prove the results about the precipitous ideals on $P_\kappa(\lambda)$.
    \item In Section \S~\ref{sec:moreon themeasurable}, we prove the further results related to the Foreman, Magidor, Zeman construction of \cite{MagForZem}.
    \item In Section \S~\ref{sec:Problems}, we collect some open problems related to this paper.
\end{itemize}
\subsection*{Notation}
For cardinals $\kappa$ and a set $X$, $P(X)$ denotes the powerset of $X$ and $$P_\kappa(X)=\{Y\subseteq X\mid |Y|<\kappa\}.$$ $\text{TrCl}(X)$ denotes the transitive closure of $X$. The set $$H_\lambda=\{X\mid |\text{TrCl}(X)|<\lambda\}.$$ Given an algebra of sets $\mathcal{A}\subseteq P(X)$ (i.e. $\emptyset\in\mathcal{A}$ and $\mathcal{A}$ is closed under finite unions, finite intersections, and complement), an ultrafilter $U$ on $\mathcal{A}$ is a subset of $\mathcal{A}$ which is closed under intersections, upwards closed with respect to inclusion, $\emptyset\notin U$ and for every $Y\in\mathcal{A}$, either $Y\in U$ of $Y^c=X\setminus Y\in U$. For two $\in$-models $N,M$ and for any class of first-order formulas $\Gamma$, we say that $N\prec_{\Gamma}M$ if $N$ is a submodel of $M$ and for any $\phi(x_1,\dots,x_n)\in \Gamma$, and any $a_1,\dots,a_n\in N$, $N\models \phi(a_1,\dots,a_n)$ if and only if $M\models \phi(a_1,\dots,a_n)$.
\section{Models and filters}\label{sec:modelsFilters}
Large cardinals $\kappa$ below a measurable cardinal can often be characterized by the existence of some type of ultrafilters on the powerset of $\kappa$ of $\kappa$-sized $\in$-models of set theory (see, for instance, \cite{gitman:ramsey}, \cite{HolyLuke:smallmodels}). We saw one such characterization for weakly compact cardinals in the introduction. In this section, we will introduce analogues for the two-cardinal context of the types of models and filters used in the theory of smaller large cardinals. 

In the following definition, we would like to capture what a possibly non-transitive, possibly class-size $\in$-model needs to satisfy in order to have enough information about its own version of $P_\kappa(\lambda)$ that we can consider having external ultrafilters on its $P_\kappa(\lambda)$ with which we can form ultrapowers.
\begin{definition}
Suppose $\kappa\leq\lambda$ are ordinals. A set or class $\in$-model $M$ is \emph{$(\kappa,\lambda)$-acceptable} if 
\begin{enumerate}
\item $M\models\ZFC^-$,
\item $M\models``\kappa,\lambda  \text{ are cardinals}\text{ and }P_\kappa(\lambda)\text{ exists}"$,
\item $M\prec_{\Sigma_0} V$, 
\item $\lambda+1\subseteq M$ and $P_\kappa(\lambda)^M\subseteq M$.
\end{enumerate}
\end{definition}
\noindent Conditions (3)-(4) hold automatically if $M$ is transitive, but provide a measure of correctness to $M$ in the case that it is not transitive. Using $\Sigma_0$-elementarity and condition (4), we can verify externally, for instance, that $P_\kappa(\lambda)^M$ consists precisely of those sets $A\in M$ such that $A\subseteq\lambda$ and $M$ has a bijection $f:\alpha\to A$ for some $\alpha<\kappa$.

Given a $(\kappa,\lambda)$-acceptable model $M$, let:
\begin{enumerate}
\item $P(P_\kappa(\lambda))^M=P(P_\kappa(\lambda)^M)\cap M$,
\item $P(\lambda)^M=P(\lambda)\cap M$,
\item $(\lambda^+)^M=\text{sup}\{\alpha\in M\mid M\models``\alpha\in\Ord\text{ and }\alpha=|\lambda|"\}$,
\item $H_{\lambda^+}^M=\{X\in M\mid M\models|\text{TrCl}(X)|\leq\lambda\}$.
\end{enumerate} 
If the sets on the left-hand side of the above definition were elements of $M$, then the definition would be standard. We only add it here because these sets may not be elements of $M$. Observe that even if $M$ is not transitive, $H_{\lambda^+}^M$ is transitive because $\lambda+1\subseteq M$. Given $\alpha<\lambda$, let $$X_\alpha=\{x\in P_\kappa(\lambda)^M\mid \alpha\in x\}.$$ Observe that each $X_\alpha\in M$ by separation. Once more, the transitivity of $M$ can be replaced by $\Sigma_0$-elementarity to see that $M$ and $V$ agree on the definition of $X_\alpha$. 

The next definition generalizes the notion of an $M$-ultrafilter to an external ultrafilter on $P(P_\kappa(\lambda))^M$.
\begin{definition}
Suppose that $M$ is a $(\kappa,\lambda)$-acceptable model and $U\subseteq P(P_\kappa(\lambda))^M$. We say that $U$ is a \emph{weak $M$-ultrafilter} if
\begin{enumerate}
\item $U$ is an ultrafilter on $P(P_\kappa(\lambda))^M$,
\item $U$ is fine, i.e. $X_\alpha\in U$ for every $\alpha<\lambda$,
\item $U$ is $M$-$\kappa$-complete, i.e. for every sequence $\{ A_\xi\mid\xi<\alpha\}\in M$, with $\alpha<\kappa$ and each $A_\xi\in U$, $\bigcap_{\xi<\alpha}A_\xi\in U$.
\end{enumerate}
We say that $U$ is an \emph{$M$-ultrafilter} if it is, additionally, $M$-normal, i.e. for every sequence $\{A_\xi\mid\xi<\lambda\}\in M$, with each \hbox{$A_\xi\in U$}, the diagonal intersection $\Delta_{\xi<\lambda}A_\xi\in U$, where $$\Delta_{\xi<\lambda}A_\xi=\{x\in P_\kappa(\lambda)^M\mid x\in \bigcap_{\xi\in x}A_\xi\}.$$
\end{definition}
If $M$ is $(\kappa,\lambda)$-acceptable and $U$ is an $M$-ultrafilter, then standard arguments show that any $f:P_\kappa(\lambda)^M\to\lambda$ in $M$ which is regressive on a set in $U$, namely, $$\{x\in P_\kappa(\lambda)^M\mid f(x)\in x\}\in U,$$ must be constant on a set in $U$. 

Given a $(\kappa,\lambda)$-acceptable model and a weak $M$-ultrafilter $U$, we may form the ultrapower of $M$ by $U$, denoted 
by $(\text{Ult}(M,U),E_U)$, to consists of all $=_U$-equivalence classes  $[f]_U$, where $ f\in M$ is a function with $\dom(f)=P_\kappa(\lambda)^M$. As is standard,  $[f]_UE_U[g]_U$ if and only if $\{x\in P_\kappa(\lambda)^M\mid f(x)\in g(x)\}\in U$. It is easy to see that the \Los-Theorem holds for the ultrapower embedding. Even though $E_U$ may not be well-founded, we will always assume without loss of generality that the well-founded part of the ultrapower $(\text{Ult}(M,U),E_U)$ has been collapsed, so that $E=\in$ on the well-founded part of $(\text{Ult}(M,U),E_U)$. 

\begin{definition} Suppose that $M$ is a $(\kappa,\lambda)$-acceptable model and $U$ is a weak $M$-ultrafilter. We say that:
\begin{enumerate} 
\item $U$ is \emph{good} if the ultrapower $(\text{Ult}(M,U),E_U)$ is well-founded.
\item $U$ has the \emph{countable intersection property} if for every $\{A_n\mid n<\omega\}$, with each $A_n\in U$, $\bigcap_{n<\omega}A_n\neq\emptyset$.
\item $U$ is \emph{weakly amenable} if for every sequence $\{A_\xi\mid\xi<\lambda\}\in M$, $$\{\xi<\lambda\mid A_\xi\in U\}\in M.$$
\end{enumerate}
\end{definition}
\noindent Observe that if $U$ has the countable intersection property, then $U$ is good. Also, if $M$ is, say, closed under $\omega$-sequences, then every weak $M$-ultrafilter has the countable intersection property because it is $M$-$\kappa$-complete.

\begin{proposition}\label{prop:weakultrapower}
Suppose that $M$ is a $(\kappa,\lambda)$-acceptable model, $U$ is a weak $M$-ultrafilter, and $j:M\to N$ is the ultrapower embedding of $M$ by $U$. Then:
\begin{enumerate}
\item $\kappa$ is an initial segment of the ordinals of $N$.
\item $\crit(j)=\kappa$,
\item For every $\alpha<\lambda$, $N\models j(\alpha)E_U[id]_U$ and $N\models j(\kappa)>|[id]_U|$.
\end{enumerate}
\end{proposition}
\begin{proof}
Since $U$ is $M$-$\kappa$-complete, it follows that $j(\alpha)=\alpha$ for all $\alpha<\kappa$ (recall that we have assumed that the well-founded part of $N$ has been collapsed).  Given $\alpha<\lambda$, we have that $j(\alpha)E_U[id]_U$ in $N$ by \Los\, since every $X_\alpha=\{x\in P_\kappa(\lambda)\mid \alpha\in x\}\in U$. Also, since for every $x\in P_\kappa(\lambda)^M$, $\kappa>|x|$, again by \Los, $N\models j(\kappa)>|[id]_U|$. Thus, in particular, $\crit(j)=\kappa$.  
\end{proof}
Note that since we don't have normality, $\kappa$ need not be an element of $N$ and, indeed, the initial segment of order-type $\kappa$ need not have a least upper bound because the ordinals could already be ill-founded there.

Conversely, it is easy to see that if $M$ is $(\kappa,\lambda)$-acceptable and $j:M\to N$ is an embedding with $\crit(j)=\kappa$ and there is $s\subseteq j(\lambda)$ in $N$ such that $$N\models j\image\lambda\subseteq s\text{ and }|s|<j(\kappa),$$ then $$U=\{A\subseteq P_\kappa(\lambda)^M\mid s\in j(A)\}$$ is a weak $M$-ultrafilter.

We can strengthen Proposition~\ref{prop:weakultrapower} significantly by assuming that $U$ is an $M$-ultrafilter, not just a weak $M$-ultrafilter. 
\begin{proposition}\label{prop:ultrapower}
Suppose that $M$ is a $(\kappa,\lambda)$-acceptable model, $U$ is an $M$-ultrafilter, and $j:M\to N$ is the ultrapower embedding of $M$ by $N$. Then:
\begin{enumerate}
\item $(\lambda^+)^M$ is an initial segment of the ordinals of $N$,
\item $j\image\lambda=[id]_U$ in $N$ and $N\models j(\kappa)>\lambda$,
\item $P(\lambda)^M\subseteq N$.
\end{enumerate}
\end{proposition}
\begin{proof}
We already showed that for every $\alpha<\lambda$, $j(\alpha)E_U [id]_U$. So suppose that $[f]_U E_U [id]_U$. Then by \Los, $\{x\in P_\kappa(\lambda)^M\mid f(x)\in x\}\in U$, and hence $f$ is constant on a set in $U$. Thus, there is $\alpha<\lambda$ such that $f(x)=\alpha$ on a set in $U$, and hence $[f]_U=j(\alpha)$. Therefore, $a E_U [id]_U$ if and only if $a=j(\alpha)$ for some $\alpha<\lambda$. Even though $N$ might not be well-founded, it can collapse $[id]_U$ to obtain (an order isomorphic to) $\lambda$. Thus, by our assumption that we collapsed the well-founded part of $N$, $\lambda+1\subseteq N$. It follows that $j(\kappa)>\lambda$ in $N$. Fix $A\subseteq\lambda$ in $M$. From $j\image\lambda$ and $\lambda$ in $N$, we get that $j\restrict\lambda\in N$. Observe that $j\image A=j(A)\cap j\image\lambda$ must be in $N$. Next, observe that $j\restrict \lambda$ and $j\image A$ can be used to recover $A$. Thus, in particular, $(\lambda^+)^M$ is an initial segment of the ordinals of $N$.
\end{proof}
Note that $(\lambda^+)^N$ could well be beyond $(\lambda^+)^M$ and, indeed, $(\lambda^+)^M$ may not be an element of $N$ because the ordinals could be ill-founded there. 

Conversely, it is easy to see that if $M$ is $(\kappa,\lambda)$-acceptable and $j:M\to N$ is an embedding with $\crit(j)=\kappa$, $j\image\lambda\in N$, and $j(\kappa)>\lambda$, then $$U=\{A\subseteq P_\kappa(\lambda)^M\mid j\image\lambda\in j(A)\}$$ is an $M$-ultrafilter.
\begin{proposition}\label{prop:wa}
Suppose that $M$ is a $(\kappa,\lambda)$-acceptable model, $U$ is a weakly amenable $M$-ultrafilter, and $j:M\to N$ is the ultrapower embedding of $M$ by $U$. Then $$P(\lambda)^M=P(\lambda)^N.$$ It follows that $(\lambda^+)^M=(\lambda^+)^N$, and, hence, also $H_{\lambda^+}^M=H_{\lambda^+}^N$.
\end{proposition}
\begin{proof}
By Proposition~\ref{prop:ultrapower}~(3), it suffices to show that if $A\subseteq\lambda$ in $N$, then $A\in M$. So suppose that $A\subseteq \lambda$ is in $N$. Since $j\restrict\lambda\in N$, we have $j\image A\in N$. Let $[f]_U=j\image A$. Then, for every $\xi<\lambda$, we have $$j(\xi)=[c_\xi]_U E_U [f]_U=j\image A$$ if and only if $$B_\xi=\{x\in P_\kappa(\lambda)^M\mid \xi\in f(x)\}\in U$$ and the sequence $\{ B_\xi\mid\xi<\lambda\}\in M$ because it is definable from $f$. By weak amenability, $A=\{\xi<\lambda\mid B_\xi\in U\}\in M$. 
\end{proof}

For the rest of this section, we suppose that $\kappa\leq\lambda$ are regular cardinals and $\lambda^{{<}\kappa}=\lambda$ (for instance, if $\GCH$ holds). The next definition generalizes the notions of (basic weak) $\kappa$-models which were used for the characterization of smaller large cardinals (see \cite{GitmanSchlicht:babyMeasurableCardinals}). 
\begin{definition}\label{def: basicModel}
A set $\in$-model $M$ of size $\lambda$ is a \emph{basic weak $(\kappa,\lambda)$-model} if 
\begin{enumerate}
\item $M\models\ZFC^-$, 
\item $M\prec_{\Sigma_0} V$, 
\item $\lambda+1\subseteq M$,
\item some $f:\lambda\overset{\text{onto}}{\underset{1-1}{\to}} P_\kappa(\lambda)\in M$.
\end{enumerate}
If we assume additionally that $M$ is transitive, then we remove the adjective ``basic". If we assume additionally that $M$ is closed under sequences of length ${<}\kappa$, $M^{{<}\kappa}\subseteq M$, then we remove the adjective ``weak".
\end{definition}
If $M$ is a basic weak $(\kappa,\lambda)$-model, then $P_\kappa(\lambda)\subseteq M$ since $M$ has a bijection between $\lambda$ and $P_\kappa(\lambda)$ and $\lambda\subseteq M$. Since $M\prec_{\Sigma_0}V$, $P_\kappa(\lambda)^M=P_\kappa(\lambda)$. Thus, $M$ is an acceptable $(\kappa,\lambda)$-model of size $\lambda$ that has the real $P_\kappa(\lambda)$. Let's argue that if $\kappa$ is inaccessible, then $V_\kappa$ is a subset and an element of every basic weak $(\kappa,\lambda)$-model $M$. Since $P_\kappa(\kappa)\subseteq P_\kappa(\lambda)$, $P_\kappa(\kappa)$ is in $M$ by separation. Since every element of $V_\kappa$ can be coded by a subset of an ordinal less than $\kappa$, $V_\kappa\subseteq M$. Thus, every $V_\alpha\in M$ for $\alpha<\kappa$, and so $V_\kappa\in M$ by replacement.
Note that the notion of a $\kappa$-model from the introduction coincides with the notion of a basic $(\kappa,\kappa)$-model.

The characterization of weakly compact cardinals in terms of the existence of $M$-ultrafilters for $\kappa$-models $M$ motivated the definitions of the following large cardinal notions that are weakenings of strong compactness and supercompactness respectively.
\begin{definition} $\,$
\begin{enumerate}
\item (White \cite{white:nearlyStronglyCompact}) A cardinal $\kappa$ is \emph{nearly} $\lambda$-\emph{strongly compact} if every $A\subseteq\lambda$ is an element of a $(\kappa,\lambda)$-model $M$ for which there is a weak $M$-ultrafilter.
\item (Schanker \cite{schanker:nearSupercompactness}) A cardinal $\kappa$ is \emph{nearly} $\lambda$-\emph{supercompact} if every $A\subseteq\lambda$ is an element of a $(\kappa,\lambda)$-model $M$ for which there is an $M$-ultrafilter.
\end{enumerate}
\end{definition}
Note that if $\kappa$ is nearly $\lambda$-supercompact (strongly compact), then $\kappa$ is $\theta$-supercompact (strongly compact) for every $\theta$ such that $2^{\theta^{<\kappa}}\leq\lambda$. 
\begin{proposition}\label{prop:nearlyLambdaStrongCompactEquivalence}
The following are equivalent:
\begin{enumerate}
\item $\kappa$ is nearly $\lambda$-strongly compact.
\item Every $A\subseteq\lambda$ is an element of a weak $(\kappa,\lambda)$-model $M$ for which there is a good weak $M$-ultrafilter $U$.
\item Every basic weak $(\kappa,\lambda)$-model $M$ has a good weak $M$-ultrafilter $U$.
\end{enumerate}
\end{proposition}
\begin{proof}
Clearly, (1) implies (2) and (3) implies (1). Suppose that (2) is true.  Using that $P_\kappa(\kappa)\in M$ combined with the existence of an elementary embedding with critical point $\kappa$ easily implies that $\kappa$ is inaccessible.  Next, observe that it suffices to verify $(3)$ for weak $(\kappa,\lambda)$-models because if $M$ is a basic weak $(\kappa,\lambda)$-model and $M^*$ is the collapse of $M$, then $U$ is a good weak $M$-ultrafilter if and only if it is a good weak $M^*$-ultrafilter. So fix a weak $(\kappa,\lambda)$-model $M$. By (2), there is a weak $(\kappa,\lambda)$-model $\bar M$ with $M\in \bar M$ for which there is a good weak $\bar M$-ultrafilter $U$. Let $j:\bar M\to \bar N$ be the ultrapower embedding. Then there is a set $s\in \bar N$ such that $j\image\lambda\subseteq s$ in $\bar N$ with $|s|^{\bar N}<j(\kappa)$ (namely $s=[id]_U$). Consider the restriction $j\restrict M: M\to j( M)$ and observe that $s$ along with any bijection witnessing that $|s|^{\bar N}<j(\kappa)$ are elements of $V_{j(\kappa)}^{\bar N}$ and hence of $j(M)$ since $V_{j(\kappa)}^{\bar N}\subseteq j(M)$  (recall that $V_\kappa$ is contained in every basic weak $(\kappa,\lambda)$-model if $\kappa$ is inaccessible). Now we can use the restriction $j\restrict M$ to get a good weak $M$-ultrafilter. Thus, (2) implies (3). 
\end{proof}
Observe that the same proposition holds when we replace nearly $\lambda$-strongly compact with nearly $\lambda$-supercompact and weak $M$-ultrafilters with $M$-ultrafilters by carrying out the same proof with $s$ replaced by $j\image\lambda$.
\begin{proposition}\label{prop:nearlyLambdaSuperCompactEquivalence}
The following are equivalent:
\begin{enumerate}
\item $\kappa$ is nearly $\lambda$-supercompact.
\item Every $A\subseteq\lambda$ is an element of a  weak $(\kappa,\lambda)$-model $M$ for which there is a good $M$-ultrafilter $U$.
\item Every basic weak $(\kappa,\lambda)$-model $M$ has a good $M$-ultrafilter $U$.
\end{enumerate}
\end{proposition}

Using Magidor's result that the least strongly compact cardinal can be the least measurable cardinal \cite{magidor:identityCrisis}, it is easy to consistently separate nearly $\lambda$-strongly compact and nearly $\lambda$-supercompact cardinals for sufficiently large $\lambda$. If $\kappa$ is strongly compact and the least measurable, then $\kappa$ is, in  particular, nearly $\lambda$-strongly compact for every $\lambda$, but it can't be nearly $\lambda$-supercompact for any $\lambda\geq 2^\kappa$ because this would imply that $\kappa$ is a limit of measurable cardinals. 

Although, we will not use it in this paper, we note that nearly $\lambda$-supercompact cardinals are precisely the $|H(\lambda)|$-$\Pi^1_1$-subcompact cardinals \cite[Lemma 8]{HayutMagidor:subcompact}, introduced by Neeman and Steel in \cite{NeemanSteelSubcompact} and Hayut \cite{YairSquare} and used extensively in recent years (e.g. \cite{HAYUT_UNGER_2020,TomYairMoti,HayutMagidor:subcompact,HayutPoveda}).

Next, we show that, analogously to weakly compact cardinals, nearly $\lambda$-strongly compact cardinals satisfy a filter extension property for filters of size $\lambda$. 
\begin{definition}
 Suppose that $\kappa$ is nearly $\lambda$-strongly compact. The \emph{weak $P_\kappa(\lambda)$-filter extension property} holds if for every basic $(\kappa,\lambda)$-model $M$, weak $M$-ultrafilter $U$, and basic $(\kappa,\lambda)$-model $N\supseteq M$, there is a weak $N$-ultrafilter $W\supseteq U$. Suppose  that $\kappa$ is nearly $\lambda$-supercompact. The \emph{$P_\kappa(\lambda)$-filter extension property} holds if we replace weak $M$-ultrafilter and weak $N$-ultrafilter with $M$-ultrafilter and $N$-ultrafilter respectively.
\end{definition}
\begin{proposition}[\cite {BuhagiarDjamonja:squareCompactness}]\label{prop:wFEP}
If $\kappa$ is nearly $\lambda$-strongly compact, then the weak $P_\kappa(\lambda)$-filter extension property holds.
\end{proposition}
\begin{proof}
Fix a basic $(\kappa,\lambda)$-model $M$, a weak $M$-ultrafilter $U$, and a basic $(\kappa,\lambda)$-model $N\supseteq M$. Choose a basic $(\kappa,\lambda)$-model $\bar M$ such that $M,U,N\in \bar M$, all are contained in $\bar M$, $\bar M$ knows that $M$  has size $\lambda$, and there is an elementary embedding $j:\bar M\to \bar N$ with $\crit(j)=\kappa$, $j(\kappa)>\lambda$ and $j\image\lambda\subseteq s$ for some $s\in \bar N$ with $|s|<j(\kappa)$ in $\bar N$. By elementarity, $\bar N$ satisfies that $j(U)$ is a weak $j(M)$-ultrafilter. Since $\bar M$ has a bijection between $M$ and $\lambda$ and $j\image \lambda$ is contained in a set of size less than $j(\kappa)$ in $\bar N$, there is a set $s'\subseteq j(U)$ of size less than $j(\kappa)$ in $\bar N$ containing $j(U)\restrict j\image M$. Thus, there is $c\in\bigcap s'$. Let $$W=\{A\in N\mid c \in j(A)\}.$$ Then clearly $W$ is a weak $N$-ultrafilter extending $U$.
\end{proof}
\begin{theorem}\label{th:fepInconsistent}
The $P_\kappa(\lambda)$-filter extension property is inconsistent.
\end{theorem}
\begin{proof}
Suppose towards a contradiction that the $P_\kappa(\lambda)$-filter extension property holds for some nearly $\lambda$-supercompact $\kappa$. We can assume that $\kappa$ is least for which there is $\lambda\geq\kappa$ such that $\kappa$ is nearly $\lambda$-supercompact and the $P_\kappa(\lambda)$-filter extension property holds. Let \hbox{$M_0\prec H_{\lambda^+}$} be any $(\kappa,\lambda)$-model and $U_0$ be any $M_0$-ultrafilter. Given a $(\kappa,\lambda)$-model $M_n\prec H_{\lambda^+}$ and $M_n$-ultrafilter $U_n$, let $\la M_n,U_n\ra\in M_{n+1}\prec H_{\lambda^+}$ and $U_{n+1}$ be any $M_{n+1}$-ultrafilter extending $U_n$. Let $M_\omega=\bigcup_{n<\omega}M_n$ and let $U_\omega=\bigcup_{n<\omega}U_n$. Note that \hbox{$M_\omega\prec H_{\lambda^+}$} is a weak $(\kappa,\lambda)$-model and $U_\omega$ is a weakly amenable $M_\omega$-ultrafilter. Let $j:M_\omega\to N_{\omega}$ be the ultrapower embedding of $M_\omega$ by $U_\omega$, with $N_\omega$ not necessarily well-founded. By Proposition~\ref{prop:wa},  \hbox{$M_{\omega}=H_{\lambda^+}^{N_\omega}$}. First, let's argue that $\kappa$ is nearly $\lambda$-supercompact in $N_\omega$. Fix some $M$, which $N_\omega$ believes is a $(\kappa,\lambda)$-model. Since\linebreak $M\in H_{\lambda^+}^{N_\omega}=M_\omega$ and $M_\omega$ satisfies that $M$ is a $(\kappa,\lambda)$-model, $M$ is actually a $(\kappa,\lambda)$-model. It follows that $H_{\lambda^+}$ has an $M$-ultrafilter, and, by elementarity, so does $M_\omega$. Thus, $N_\omega$ has an $M$-ultrafilter. Next, let's argue that the $P_\kappa(\lambda)$-filter extension property holds in $N_\omega$. Working in $N_\omega$, fix $M, U, N\in H_{\lambda^+}^{N_\omega}=M_\omega$ such that $N_\omega$ (and, hence, $M_\omega$) satisfy that $M\subseteq N$ are $(\kappa,\lambda)$-models and $U$ is an  $M$-ultrafilter. Thus, $M,N$ are actually $(\kappa,\lambda)$-models, and so $H_{\lambda^+}$ has a $N$-ultrafilter extending $U$. By elementarity, $M_\omega$ has some $N$-ultrafilter $W\supseteq U$. Since $N_\omega$ satisfies that $\kappa$ is nearly $\lambda$-supercompact, with $\lambda<j(\kappa)$, and the $P_\kappa(\lambda)$-filter extension property holds, by elementarity, $M_\omega$ satisfies that there is a cardinal $\bar\kappa<\kappa$ and some $\bar\kappa\leq\bar\lambda<\kappa$  such that $\bar\kappa$ is nearly $\bar\lambda$-supercompact and the $P_{\bar\kappa}(\bar\lambda)$-filter extension property holds. Since $\kappa$ is inaccessible, $V_\kappa\subseteq M_0\subseteq M_\omega$, which means that $\bar\kappa$ is actually $\bar\lambda$-nearly supercompact and the $P_{\bar\kappa}(\bar\lambda)$-filter extension property actually holds. But this contradicts the leastness of $\kappa$.
\end{proof}
Theorem~\ref{th:fepInconsistent} shows that if we extend a $(\kappa,\lambda)$-model $M$, for which we have an $M$-ultrafilter $U$, to another $(\kappa,\lambda)$-model $N$, we cannot always correspondingly extend $U$ to an $N$-ultrafilter. But if given $M$, we \textit{choose} the $M$-ultrafilter $U$ strategically, then given any $N$ extending $M$, we will be able to extend $U$ to an $N$-ultrafilter. This  motivates the filter extension games we introduce in Section~\ref{sec:games}. The existence of winning strategies in these games will be intimately related to versions of generic supercompactness.

We end this section by giving a characterization of $\lambda$-ineffable cardinals in terms of the existence of certain $M$-ultrafilters for $(\kappa,\lambda)$-models $M$. Recall that a subset $A\subseteq P_\kappa(\lambda)$ is \emph{unbounded} if for every $x\in P_\kappa(\lambda)$, there is $y\in A$ such that $x\subseteq y$, \emph{closed} if for every increasing sequence $\{x_\xi\mid\xi<\alpha\}$, with $\alpha<\kappa$ and all $x_\xi\in A$, $\bigcup_{\xi<\alpha}x_\xi\in A$, \emph{stationary} if it has a non-empty intersection with every closed unbounded set. Jech \cite{jech:lambdaIneffable} defined that a cardinal $\kappa$ is \emph{$\lambda$-ineffable} if for any function $f:P_\kappa(\lambda)\to P_\kappa(\lambda)$ such that $f(x)\subseteq x$ for all $x\in P_\kappa(\lambda)$, there is $A\subseteq \lambda$ such that $$\{x\in P_\kappa(\lambda)\mid A\cap x=f(x)\}$$ is stationary.

The proof of the next proposition is based on the proof of Corollary 1.3.1 in \cite{AbramsonHarringtonKleinbergZwicker:FlippingProperties}, which shows that a cardinal $\kappa$ is ineffable if and only if for every collection $\vec A=\{A_\xi\mid\xi<\kappa\}$ of subsets of $\kappa$, there is a corresponding collection $\{\bar A_\xi\mid\xi<\kappa\}$, a \emph{flip} of $\vec A$, where $\bar A_\xi=A_\xi$ or $\bar A_\xi=\kappa\setminus A_\xi$, such that $\Delta_{\xi<\kappa}\bar A_\xi$ is stationary.

Recall that the stationarity of a diagonal intersection of any $\lambda$-sized collection of sets is independent of the enumeration of those sets as a $\lambda$-sequence. This is because for any collection of sets $\{A_\xi\mid \xi<\lambda\}$ with $A_\xi\subseteq P_\kappa(\lambda)$ for all $\xi$, the equivalence class of $\Delta_{\xi<\lambda}A_\xi$ is the greatest lower bound of $\{[A_\xi]_{I_{\rm NS}}\mid \xi<\lambda\}$ in the Boolean algebra $P(P_\kappa(\lambda))/I_{\rm NS}$, where $I_{\rm NS}$ is the non-stationary ideal.
\begin{proposition}\label{prop:lambdaIneffableEqui}
    A cardinal $\kappa$ is $\lambda$-ineffable if and only if every $(\kappa,\lambda)$-model $M$ has an $M$-ultrafilter $U=\{A_\xi\mid\xi<\lambda\}$ such that $\Delta_{\xi<\lambda} A_\xi$ is stationary. In particular, $\kappa$ is nearly $\lambda$-supercompact.
\end{proposition}
\begin{proof}
 Suppose that $\kappa$ is $\lambda$-ineffable. Let $M$ be any $(\kappa,\lambda)$-model. Fix some enumeration $\{A_\xi\mid\xi<\lambda\}$ of $P(P_\kappa(\lambda))^M$.  Define $f:P_\kappa(\lambda)\to P_\kappa(\lambda)$ by $$f(x)=\{\xi\in x\mid x\in A_\xi\},$$ and observe that by definition $f(x)\subseteq x$. Thus, there is a stationary set $B$ and $A\subseteq \lambda$ such that for every $x\in B$, $A\cap x=f(x)$. For $\xi<\lambda$, define $\bar A_\xi$ to be $A_\xi$ if $\xi\in A$ and $P_\kappa(\lambda)\setminus A_\xi$ otherwise. We claim that $B\subseteq\Delta_{\xi<\lambda}\bar A_\xi$. Suppose that $x\in B$. Then, we have $$A\cap x=\{\xi\in x\mid x\in A_\xi\}.$$ We need to show that $x\in \bigcap_{\xi\in x}\bar A_\xi$. Fix $\xi\in x$. First, suppose that $\xi\in A$. Then $\bar A_\xi=A_\xi$ and since $\xi\in x\cap A$, $x\in A_\xi$. Otherwise, $\xi\notin A$. Then $\bar A_\xi=P_\kappa(\lambda)\setminus A_\xi$ and $x\notin A_\xi$. Thus, $x\in \bar A_\xi$. This completes the argument that $B\subseteq \Delta_{\xi<\lambda}\bar A_\xi$. In particular, $\Delta_{\xi<\lambda}\bar A_\xi$ contains a stationary set. 

Next, we claim that $$U=\{\bar A_\xi\mid\xi<\kappa\}$$ is an $M$-ultrafilter. Let $A_\xi=P_\kappa(\lambda)$. Since $$\{x\in P_\kappa(\lambda)\mid \xi\in x\}$$ is a club, there is some $x\in B$ with $\xi \in x$. But then $x\in \bar A_\xi$, which means $\bar A_\xi=A_\xi=P_\kappa(\lambda)\in U$. Next, suppose that $A_\delta\supseteq \bar A_\xi$. Since $$\{x\in P_\kappa(\lambda)\mid \xi,\delta\in x\}$$ is a club, there is some $x\in B$ with $\xi,\delta\in x$, so that $x\in \bar A_\xi,\bar A_\delta$. It follows that $\bar A_\delta=A_\delta\in U$. Next, suppose that for $\eta<\beta<\kappa$, $\bar A_{\xi_\eta}\in U$ and let $A_\delta=\bigcap_{\eta<\beta}\bar A_{\xi_\eta}$. Let $a=\{\xi_\eta\mid\eta<\beta\}\cup\{\delta\}$. Since $$\{x\in P_\kappa(\lambda)\mid a\subseteq x\}$$ is a club, there is some $x\in B$ with $a\subseteq x$. Thus, $x\in \bar A_{\xi_\eta}$ for all $\eta<\beta$ and $x\in \bar A_\delta$. It follows that $\bar A_\delta=A_\delta\in U$. Thus, $U$ is a weak $M$-ultrafilter. It remains to show that $U$ is closed under diagonal intersections of sequences from $M$. Fix a collection $\{A_{\xi_\eta}\mid\eta<\lambda\}\in M$ with each $A_{\xi_\eta}\in U$. We want to argue that $A_\delta=\Delta_{\eta<\lambda}A_{\xi_\eta}\in U$. By our observation above that the equivalence class of the diagonal intersection in $P(\kappa)/I_{\rm NS}$, of a collection of sets, is the greatest lower bound of the sets, it follows that  there is a non-stationary set $X$ such that $B\subseteq \Delta_{\xi<\lambda}\bar A_\xi\subseteq A_\delta\cup X.$ So suppose towards a contradiction that $\bar A_\delta$ is the complement of $A_\delta$. Then, $B\subseteq \bar A_\delta\cup Y$, where $Y=\{x\in P_\kappa(\lambda)\mid \delta\notin x\}$ is non stationary. Thus, $B\subseteq (\bar A_\delta\cup Y)\cap (A_\delta\cup X)$, but this is impossible because $B$ is stationary and its alleged superset is not. Thus, we have reached a contradiction showing that $A_\delta\in U$.

For the other direction, fix a function $f:P_\kappa(\lambda)\to P_\kappa(\lambda)$ such that $f(x)\subseteq x$ for every $x\in P_\kappa(\lambda)$. Fix a $(\kappa,\lambda)$-model $M$ with $f\in M$ and an $M$-ultrafilter $U=\{A_\xi\mid\xi<\lambda\}$ such that $\Delta_{\xi<\lambda}A_\xi$ is stationary. For every $\xi<\lambda$, let $$B_\xi=\{x\in P_\kappa(\lambda)\mid \xi\in f(x)\}.$$ Define $\bar B_\xi$ to be whichever of the $B_\xi$ or $P_\kappa(\lambda)\setminus B_\xi$ is chosen by $U$, and observe that 
$B=\Delta_{\xi<\lambda}\bar B_\xi$ must be stationary. Let $A=\bigcup_{x\in B}f(x)$. We will argue that for every $x\in B$, $A\cap x=f(x)$. So fix $x\in B$. First, suppose that $\xi\in A\cap x$. Then $\xi\in x$ and there is $y\in B$ such that $\xi\in f(y)\subseteq y$. By our definition of the $B_\xi$, it follows that $y\in B_\xi$, and so since $\xi\in y$, it must be that $\bar B_\xi=B_\xi$. Now since $\xi\in x$, it must be that $x\in \bar B_\xi=B_\xi$, and hence $\xi\in f(x)$, by definition of $B_\xi$. Next, suppose that $\xi\in f(x)$. Then $\xi\in x$, since $f(x)\subseteq x$, and $\xi\in A$, since $x\in B$, meaning that $\xi\in A\cap x$.

\end{proof}

\section{Generic supercompactness}\label{sec:genericSupercompactness}
Observe that if $\kappa\leq\lambda$ are cardinals, then $V$ or $H_\theta$, for large enough $\theta$, are $(\kappa,\lambda)$-acceptable models of any forcing extension of $V$. As usual, we suppose that $\kappa\leq\lambda$ are regular cardinals and $\lambda^{{<}\kappa}=\lambda$.

First, we observe that nearly $\lambda$-strongly compact cardinals can be characterized by the existence of certain generic embeddings of $V$.
\begin{definition}
A cardinal $\kappa$ is \emph{almost generically $\lambda$-strongly compact with weak amenability (wa)} if in a set-forcing extension $V[G]$, there is a weakly amenable weak $V$-ultrafilter.
\end{definition}
The analogue of the next proposition for weakly compact cardinals appears in \cite{AbramsonHarringtonKleinbergZwicker:FlippingProperties} (Theorem 2.1.2). 
\begin{proposition}\label{prop:nearlyStrongCompactEquivGenerallyAlmostStrongCompact}
A cardinal $\kappa$ is nearly $\lambda$-strongly compact if and only if it is almost generically $\lambda$-strongly compact with wa.
\end{proposition}
\begin{proof}
Suppose that $V[G]$ is a set-forcing extension in which there is a weakly amenable weak $V$-ultrafilter $U$. Fix a $(\kappa,\lambda)$-model $M$ in $V$. By weak amenability, $u=U\restrict M\in V$ and it is clearly a weak $M$-ultrafilter. 

Next, suppose that $\kappa$ is nearly $\lambda$-strongly compact. Consider the poset $\p$ whose conditions are pairs $(M,u)$, where $M$ is a $(\kappa,\lambda)$-model and $u$ is a weak $M$-ultrafilter ordered by extension in both coordinates. Let $G\subseteq \p$ be $V$-generic. In $V[G]$, let $U$ be the union of the filters $u$ coming from the second coordinates of conditions in $G$. Let's argue that $U$ is a weakly amenable weak $V$-ultrafilter. Fix $A\subseteq P_\kappa(\lambda)$. Given any $(M,u)\in \p$, by Proposition~\ref{prop:wFEP}, there is $(N,w)\in \p$ extending $(M,u)$ such that $A\in N$. Thus, $U$ measures $A$ by density. Next, suppose that $\vec A=\{A_\xi\mid \xi<\alpha\}\in V$, with $\alpha<\kappa$, such that all $A_\xi\in U$. Given any $(M,u)\in \p$, there is $(N,w)\in\p$ extending $(M,u)$ such that $\vec A\in N$. By density, some such condition $(N,w)$ is in $G$, and so, in particular, $w\subseteq U$. Thus, all $A_\xi\in w$, and hence $\bigcap_{\xi<\alpha}A_\xi\in w\subseteq U$. Finally, suppose that $\vec A=\{A_\xi\mid\xi<\lambda\}\in V$ with each $A_\xi\subseteq P_\kappa(\lambda)$. Given any $(M,u)\in \p$, there is $(N,w)\in\p$ extending $(M,u)$ such that $\vec A\in N$. By density, some such condition $(N,w)$ is in $G$, and since $w\subseteq U$, $\{\xi<\lambda\mid A_\xi\in U\}=\{\xi<\lambda\mid A_\xi\in w\}\in V$. 
\end{proof}
We will see shortly that the analogously defined almost generic $\lambda$-super compactness with wa is not equivalent to nearly $\lambda$-supercompactness.

The following hierarchy of generic versions of supercompactness generalizes the generic embeddings characterization of nearly $\lambda$-strongly compact cardinals. 
\begin{definition} A cardinal $\kappa$ is:
\begin{enumerate}
\item \emph{generically $\lambda$-supercompact with weak amenability (wa) by $\delta$-closed forcing}, for a regular cardinal $\delta$, if in a set-forcing extension $V[G]$ by a $\delta$-closed forcing, there is a weakly amenable $V$-ultrafilter.  For $\delta=\omega$ we simply say that $\kappa$ is \emph{almost generically} $\lambda$-\emph{supercompact with weak amenability (wa)}.
\item \emph{generically $\lambda$-supercompact with weak amenability (wa) for sets} if for each regular $\theta>\kappa$, in a set-forcing extension $V[G]$, there is a good weakly amenable $H_\theta$-ultrafilter.
\item \emph{generically $\lambda$-supercompact with weak amenability (wa)} if in a set-forcing extension $V[G]$, there is a good weakly amenable $V$-ultrafilter.
\end{enumerate}
\end{definition}
Clearly, if $\kappa$ is generically $\lambda$-supercompact with weak amenability (wa), then it is generically $\lambda$-supercompact with weak amenability (wa) for sets. Also, note that if $\delta>\omega$ and $\kappa$ is generically $\lambda$-supercompact with weak amenability (wa) by $\delta$-closed forcing, then $\kappa$ is generically $\lambda$-supercompact with weak amenability (wa) because $V$ is closed under sequences of length ${<}\delta$ in any $\delta$-closed forcing extension and so any $V$-ultrafilter in such an extension has the countable intersection property.

Since the restriction of a weakly amenable $V$-ultrafilter (from a forcing extension $V[G]$) to any $(\kappa,\lambda)$-model $M$ is an $M$-ultrafilter in $V$, every almost generically $\lambda$-supercompact with wa cardinal is nearly $\lambda$-supercompact. However, in the following proposition we prove that the least $\kappa$ for which there is $\lambda\geq\kappa$ such that $\kappa$ is nearly $\lambda$-supercompact is not almost generically $\lambda$-supercompact with wa. 
\begin{proposition}
Suppose that $\kappa$ is least for which there is $\lambda\geq\kappa$ such that $\kappa$ is nearly $\lambda$-supercompact. Then $\kappa$ is not almost generically $\lambda$-supercompact with wa for any $\lambda\geq\kappa$. 
\end{proposition}
\begin{proof}
Suppose towards a contradiction that $\kappa$ is almost generically $\lambda$-supercompact with wa for some $\lambda\geq\kappa$. Let $V[G]$ be a set-forcing extension in which there is a weakly amenable $V$-ultrafilter $U$. Let $j:V\to M$ be the ultrapower embedding of $V$ by $U$. Although $M$ maybe not be well-founded it has the same $H_{\lambda^+}$ as $V$ by Proposition~\ref{prop:wa}. Fix $A\subseteq\lambda$ in $M$. Since $A\in V$ and $\kappa$ is nearly $\lambda$-supercompact in $V$, $V$ has a $(\kappa,\lambda)$-model $N$, with $A\in N$, and an $N$-ultrafilter $U$. Since $H_{\lambda^+}\subseteq M$, $N, U\in M$. Since $A$ was arbitrary, we have verified that $\kappa$ is nearly $\lambda$-supercompact in $M$. By elementarity, $V$ satisfies that there is $\bar\kappa<\kappa$ and $\bar\lambda$ such that $\bar\kappa$ is nearly $\bar\lambda$-supercompact, contradicting our minimality assumption.
\end{proof}
\begin{question}
Can nearly $\lambda$-supercompact cardinals be characterized by the existence of some form of generic elementary embeddings of $V$?
\end{question}

Next, we show that almost generically $\lambda$-supercompact with wa cardinals are analogues of completely ineffable cardinals.
\begin{definition}
A cardinal $\kappa$ is \emph{completely $\lambda$-ineffable} if there is a non-empty upward closed collection $\mathcal S$ of stationary subsets of $P_\kappa(\lambda)$ such that for every $B\in\mathcal S$ and function \hbox{$f:P_\kappa(\lambda)\to P_\kappa(\lambda)$} with $f(x)\subseteq x$ for all $x\in P_\kappa(\lambda)$, there is $A\subseteq \lambda$ and $B'\subseteq B$, with $B'\in \mathcal S$, such that for all $x\in B'$, $A\cap x=f(x)$. 
\end{definition}

The next theorem is the analogue of  \cite[Thm 2.2]{AbramsonHarringtonKleinbergZwicker:FlippingProperties}
\begin{theorem}\label{th:almostGenericallySupercompactCompletelyIneffable}
A cardinal $\kappa$ is almost generically $\lambda$-supercompact with wa if and only if it is completely $\lambda$-ineffable.
\end{theorem}
\begin{proof}
Suppose that $\kappa$ is almost generically $\lambda$-supercompact with wa. Let $V[G]$ be a set-forcing extension in which there is a weakly amenable $V$-ultrafilter $U$. Let $j:V\to M$ be the ultrapower embedding of $V$ by $U$. Let $\mathcal S_0$ be the collection of all stationary subsets of $P_\kappa(\lambda)$. Given that we have defined $\mathcal S_\xi$, we let $\mathcal S_{\xi+1}$ consist of all $B\in\mathcal S_\xi$ such that for every $f:P_\kappa(\lambda)\to P_\kappa(\lambda)$, with $f(x)\subseteq x$ for every $x\in P_\kappa(\lambda)$, there is $A\subseteq \lambda$ and $B'\subseteq B$ in $\mathcal S_\xi$ such that for all $x\in B'$, $A\cap x=f(x)$. If $\alpha$ is a limit ordinal, then $\mathcal S_\alpha=\bigcap_{\xi<\alpha} \mathcal S_\xi$. Since the sequence of the $\mathcal S_\xi$ is weakly decreasing, there must be some $\theta$ such that $\mathcal S_\theta=\mathcal S_{\theta+1}$. Clearly, either $\mathcal S_\theta=\emptyset$ or $\mathcal S_\theta$ witnesses that $\kappa$ is completely $\lambda$-ineffable. We will show that $\mathcal S_\theta\neq\emptyset$ by showing that $U\subseteq \mathcal S_\xi$ for every $\xi$. Clearly, $U\subseteq \mathcal S_0$ because it consists of stationary sets. Also, clearly, if $\alpha$ is a limit ordinal and $U\subseteq \mathcal S_\xi$ for all $\xi<\alpha$, then $U\subseteq \mathcal S_\alpha$. So suppose that $U\subseteq\mathcal S_\xi$. Fix $B\in U$ and $f:P_\kappa(\lambda)\to P_\kappa(\lambda)$ with $f(x)\subseteq x$ for every $x\in P_\kappa(\lambda)$. Let $A'=j(f)(j\image\lambda)\subseteq j\image\lambda$ and let $A$ be the preimage of $A'$ under $j$. Since $V$ and $M$ have the same subsets of $\lambda$ by Proposition~\ref{prop:wa}, $A\in V$. Observe that $j(A)\cap j\image\lambda=j(f)(j\image\lambda)$. Thus, $X=\{x\in P_\kappa(\lambda)\mid A\cap x=f(x)\}\in U$ and we can let $B'=X\cap B\in U\subseteq S_\xi$, which witnesses that $B\in S_{\xi+1}$. This completes the proof that $\kappa$ is completely $\lambda$-ineffable.

For the other direction, suppose that $\kappa$ is completely $\lambda$-ineffable. Fix a collection $\mathcal S$ of stationary subsets of $P_\kappa(\lambda)$ witnessing the complete $\lambda$-ineffability of $\kappa$. Observe that given a $(\kappa,\lambda)$-model $M$, we can obtain a stationary set $B$ (from an enumeration of $P(P_\kappa(\lambda))^M$) as in the forward direction of the proof of Proposition~\ref{prop:lambdaIneffableEqui} (which we then use to obtain an $M$-ultafilter with a stationary diagonal intersection), but ensuring that $B\in\mathcal S$. 

Now suppose that we are given a $(\kappa,\lambda)$-model $M$ with a stationary set $B^M\in\mathcal S$ obtained as above. Let $U$ be the $M$-ultrafilter constructed from $B^M$ as in the proof of Proposition~\ref{prop:lambdaIneffableEqui}. Next, suppose we are given a $(\kappa,\lambda)$-model $N$ with $M\in N$. Again, using an enumeration of $P(P_\kappa(\lambda))^N$, we obtain a stationary set $B^N$ as in the proof of the forward direction of Proposition~\ref{prop:lambdaIneffableEqui}, but now using the properties of $\mathcal S$, we can ensure that  $B^N\in \mathcal S$ and $B^N\subseteq B^M$. Let $W$ be the ultrafilter obtained from $B^N$ as in the proof of Proposition~\ref{prop:lambdaIneffableEqui}.  Since $B^N\subseteq B^M$, it easily follows that $W$ must extend $U$. Thus, more generally, 
given a $(\kappa,\lambda)$-model $M$, there is an associated set $B^M\in\mathcal{S}$ from which we can construct an $M$-ultrafilter $U$ such that given any $(\kappa,\lambda)$-model $N$ with $M\in N$, there is an associated set $B^N\subseteq B^M$, $B^N\in\mathcal{S}$, from which we can construct an $N$-ultrafilter $W\supseteq U$.  Now consider a poset $\p$ whose conditions are triples $(M,u,B^M)$ such that $M$ is a $(\kappa,\lambda)$-model and $u$ is an $M$-ultrafilter constructed as above from its associated set $B^M\in \mathcal S$. The triples are ordered so that $(N,w,B^M)\leq (M,u,B^N)$ whenever $M\in N$, $w$ extends $u$ and $B^N\subseteq B^M$. It is now easy to see that if $G\subseteq \p$ is $V$-generic, then $U$, the union of all second coordinates of conditions in $G$, is a weakly amenable $V$-ultrafilter.

\end{proof}

The characterization of almost generically $\lambda$-supercompact with wa cardinals as completely $\lambda$-ineffable cardinals allows us to separate them from generically $\lambda$-supercompact with wa for sets cardinals.

\begin{proposition}
Suppose $\kappa$ is the least cardinal for which there is $\lambda\geq\kappa$ such that $\kappa$ is almost generically $\lambda$-supercompact with wa, then $\kappa$ is not generically $\lambda$-supercompact with wa for sets for any $\lambda\geq\kappa$.
\end{proposition}
\begin{proof}
Suppose that $\kappa$ is $\lambda$-generically supercompact with wa for sets. Then $\kappa$ is clearly almost generically $\lambda$-supercompact with wa, and hence completely $\lambda$-ineffable. Fix some very large $\theta$ and a forcing extension $V[G]$ having a good weakly amenable $H_\theta$-ultrafilter $U$. Let $j:H_\theta\to M$ be the ultrapower embedding by $U$. Let's argue that $\kappa$ is completely $\lambda$-ineffable in $M$. Now we work inside the transitive model $M$. We let $\mathcal S_0$ be the collection of all stationary subsets of $P(P_\kappa(\lambda))$, given $S_\xi$, we let $\mathcal S_{\xi+1}$ consist of all $B\in\mathcal S_\xi$ such that for every $f:P_\kappa(\lambda)\to P_\kappa(\lambda)$, with $f(x)\subseteq x$ for every $x\in P_\kappa(\lambda)$, there is $A\subseteq \lambda$ and $B'\in \mathcal S_\xi$ such that for all $x\in B'$, $A\cap x=f(x)$, and if $\alpha$ is a limit ordinal,  we let $\mathcal S_\alpha=\bigcap_{\xi<\alpha} \mathcal S_\xi$. Since there must be some $\theta$ such that $\mathcal S_\theta=\mathcal S_{\theta+1}$, it remains to show that $S_\theta\neq\emptyset$. But now we argue as in the proof of Theorem~\ref{th:almostGenericallySupercompactCompletelyIneffable} that $U\subseteq S_\xi$ for all $\xi$. Thus, by elementarity, there must be $\bar\kappa<\kappa$ and $\kappa\leq \bar \lambda<\kappa$ such that $\bar \kappa$ is completely $\bar\lambda$-ineffable in $H_\theta$, which means that $\bar\kappa$ is actually completely $\bar\lambda$-ineffable, contradicting the minimality assumption on $\kappa$.
\end{proof}
\begin{question}
Can we separate generically $\lambda$-supercompact with wa cardinals from generically $\lambda$-supercompact with wa for sets cardinals?
\end{question}

\section{Filter games}\label{sec:games}
As usual, we suppose that $\kappa\leq \lambda$ are regular cardinals and $\lambda^{{<}\kappa}=\lambda$. We define the following two-player games of perfect information.
\begin{definition}\label{Definition:Filter Games} Suppose that $\delta\leq\lambda^+$ and $\theta\geq\lambda^+$ are regular cardinals.
\begin{enumerate}
\item (weak games) Let $wG_\delta(\kappa,\lambda)$ be the following two player game of perfect information played by the challenger and judge. The challenger starts the game and plays a basic $(\kappa,\lambda)$-model $M_0$ and the judge responds by playing an $M_0$-ultrafilter $U_0$. At stage $\gamma>0$, the challenger plays a basic $(\kappa,\lambda)$-model $M_\gamma$, with $\{\la M_\xi,U_\xi\ra\mid \xi<\gamma\}\in M_\gamma$ elementarily extending his previous moves, and the judge plays an $M_\gamma$-ultrafilter $U_\gamma$ extending $\bigcup_{\xi<\gamma}U_\xi$. The judge wins if she can continue playing for $\delta$-many steps.
\item Let $G_\delta(\kappa,\lambda)$ be the analogous game, but where to win the judge must satisfy the additional requirement that  $U=\bigcup_{\gamma<\delta}U_\gamma$ is a good $M=\bigcup_{\gamma<\delta}M_\gamma$-ultrafilter.
\item (strong games) Let $sG_\delta(\kappa,\lambda,\theta)$ be the analogous game, but where to win the judge must satisfy the additional requirement that  $U=\bigcup_{\gamma<\delta}U_\gamma$ has the countable intersection property.
\item  (non-normal games) Let $(s/w)G^{*}_\delta(\kappa,\lambda)$ be the analogously defined games, but where the judge is only required to play weak $M_\xi$-ultrafilters.
\end{enumerate}
We add the parameter $\theta$ to obtain the analogous games $(w/s)G^{(*)}(\kappa,\lambda,\theta)$, where the challenger instead plays basic $(\kappa,\lambda)$-models elementary in $H_\theta$.
\end{definition}
For compactness reasons, the expression $(s/w)G^{(*)}_\delta(\kappa,\lambda,(\theta))$ will refer to all possible $12$ different versions of the games defined above. If we leave out one of the parenthesis, then we are referring to the obvious subset of these games.

\begin{proposition}
Suppose $\text{cof}(\delta)\neq\omega$ and $\theta\geq\lambda^+$ is regular. Then the games $wG^{(*)}_\delta(\kappa,\lambda,\theta)$, $G^{(*)}_\delta(\kappa,\lambda,\theta)$, and $sG^{(*)}_\delta(\kappa,\lambda,\theta)$ are all equivalent. 
\end{proposition}
\begin{proof}
In this case, in any of the games, the union $M$ of the challenger's moves is closed under countable sequences and so the (weak) $M$-ultrafilter $U$, the union of the judge's moves, has the countable intersection property.
\end{proof}
The following proposition holds because, as long as the union ultrafilter is not required to be good, whether the judge wins or loses depends on the $(\kappa,\lambda)$-models $H_{\lambda^+}^{M_\xi}$, where the $M_\xi$ are the challenger's moves. This argument does not work in the case whether the union ultrafilter $U$ of the judge's moves is required to be good because it can produce a well-founded ultrapower of $H_{\lambda^+}^M$, where $M$ is the union of the challenger's moves, but at the same time the full ultrapower of $M$ could be ill-founded because the witnessing function is very large.
\begin{proposition}
In the games $wG^{(*)}_\delta(\kappa,\lambda,\theta)$ and $sG^{(*)}_\delta(\kappa,\lambda,\theta)$, either player has a winning strategy for some $\theta$ if and only if the same player has a winning strategy for all $\theta$ if and only if the same player has a winning strategy in the game $wG^{(*)}_\delta(\kappa,\lambda)$ and $sG^{(*)}_\delta(\kappa,\lambda)$ respectively.
\end{proposition}
\begin{corollary}\label{cor:goodGameNotCofOmega}
Suppose $\text{cof}(\delta)\neq\omega$. In the game $G^{(*)}_\delta(\kappa,\lambda,\theta)$, either player has a winning strategy for some $\theta$ if and only if the same player has a winning strategy for all $\theta$ if and only if the same player has a winning strategy in the game $G^{(*)}_\delta(\kappa,\lambda)$.
\end{corollary}

\section{Winning strategies and large cardinals}\label{sec: winning strategies}
In this section, we establish several filter-extension-like properties characterizing large cardinals in the interval [nearly $\lambda$-strongly compact, $\lambda$-supercompact]. More precisely, we connect the existence of a winning strategy for the judge in some of the various games from the previous section with the well-known large cardinal notions considered in Section~\ref{sec:modelsFilters}.  For the other games, we would like to connect the property of the judge having a winning strategy with one of the generic supercompact large cardinal notions  tailored in Section \ref{sec:genericSupercompactness} to fit this purpose.

As usual, we suppose that $\kappa\leq\lambda$ are regular cardinals and $\lambda^{{<}\kappa}=\lambda$.
Assuming $2^\lambda=\lambda^+$, the existence of a winning strategy for the judge in the game $G^*_{\lambda^+}(\kappa,\lambda)$ ($G_{\lambda^+}(\kappa,\lambda)$) is equivalent to $\lambda$-strong compactness ($\lambda$-supercompactness).
\begin{proposition} Suppose $2^\lambda=\lambda^+$. Then $\kappa$ is $\lambda$-strongly compact $($$\lambda$-supercompact$)$ if and only if
the judge has a winning strategy in the game $G^*_{\lambda^+}(\kappa,\lambda)$ $(G_{\lambda^+}(\kappa,\lambda))$.
\end{proposition}
\begin{proof} We prove the equivalence for $\lambda$-strongly compact cardinals, and observe that the argument for $\lambda$-supercompact cardinals is analogous.
If $\kappa$ is $\lambda$-strongly compact, then the judge can fix a fine measure $\mu$ on $P_\kappa(\lambda)$ and play pieces of it for her winning strategy in any game $G^*_{\lambda^+}(\kappa,\lambda)$, namely if the challenger plays a $(\kappa,\lambda)$-model $M_\xi$, then the judge plays $M_\xi\cap \mu$. Suppose now that the judge has a winning strategy $\sigma$ in the game $G^*_{\lambda^+}(\kappa,\lambda)$. Consider a run of $G^*_{\lambda^+}(\kappa,\lambda)$ where the challenger plays an elementary sequence of $(\kappa,\lambda)$-models $M_\xi$ such that $P(P_\kappa(\lambda))\subseteq\bigcup_{\xi<\lambda^+}M_\xi$, which is possible since $2^\lambda=\lambda^+$. Then $U$, the union of the judge's moves according to $\sigma$ for this run of the game, is a fine measure on $P_\kappa(\lambda)$.
\end{proof}
Since given any $A\subseteq \lambda$, the challenger can play a $(\kappa,\lambda)$-model containing $A$, the judge having a winning strategy in the game $G^*_1(\kappa,\lambda)$ ($G_1(\kappa,\lambda))$) of length 1 is equivalent to nearly $\lambda$-strong compactness (nearly $\lambda$-supercompactness). Using that nearly $\lambda$-strongly compact cardinals satisfy the weak filter extension property (Proposition~\ref{prop:wFEP}),  and that for the judge to win the weak game $wG^*_\omega(\kappa,\lambda)$,  the union ultrafilter is not required to be good, we get the stronger result that nearly $\lambda$-strong compactness is equivalent to the judge having a winning strategy in the game $wG^*_\omega(\kappa,\lambda)$. 
\begin{proposition}\label{prop:nearlyStrongCompact}
The following are equivalent for a cardinal $\kappa$.
\begin{enumerate}
\item $\kappa$ is nearly $\lambda$-strongly compact.
\item $\kappa$ is almost generically $\lambda$-strongly compact with wa.
\item The judge has a winning strategy in the game $G^*_1(\kappa,\lambda)$.
\item The judge has a winning strategy in the game $wG^*_\omega(\kappa,\lambda)$.
\end{enumerate}
\end{proposition}
\noindent Recall that $(1)\Longleftrightarrow (2)$ follows from Proposition~\ref{prop:nearlyStrongCompactEquivGenerallyAlmostStrongCompact}.
\begin{proposition}\label{Prop: G_1 wining strategy}
A cardinal $\kappa$ is nearly $\lambda$-supercompact if and only if the judge has a winning strategy in the game $G_1(\kappa,\lambda)$.
\end{proposition}
The next result shows that if the judge has a winning strategy in the game $G_2(\kappa,\lambda)$, then $\kappa$ is at least $\lambda$-ineffable.

\begin{proposition}
If the judge has a winning strategy in the game $G_2(\kappa,\lambda)$, then $\kappa$ is $\lambda$-ineffable.
\end{proposition}
\begin{proof}
Suppose the judge has a winning strategy $\sigma$ in the game $G_2(\kappa,\lambda)$. Suppose that some $(\kappa,\lambda)$-model $M$ doesn't have an $M$-ultrafilter with a stationary diagonal intersection. Let $M_0=M$, and have the challenger play $M_0$ as his first move. Let $U_0$ be the judge's response according to $\sigma$. Fix some enumeration $U_0=\{A_\xi\mid\xi<\lambda\}$. By our assumption $A=\Delta_{\xi<\lambda} A_\xi$ is not stationary, and so we can fix a club $C$ such that $A\cap C=\emptyset$. Next, we have the challenger play a $(\kappa,\lambda)$-model $M_1$ with $\{A_\xi\mid\xi<\lambda\}$ and $C$ both in $M_1$. The judge then chooses an $M_1$-ultrafilter $U_1$ extending $U_0$ according to $\sigma$. But then each $A_\xi\in U_1$ and $C\in U_1$, and so $A$ and $C$ are both in $U_2$, which is impossible since $A\cap C=\emptyset$. Thus, we have derived a contradiction, showing that every $(\kappa,\lambda)$-model $M$ has an $M$-ultrafilter with a stationary diagonal intersection, and so by Proposition~\ref{prop:lambdaIneffableEqui} $\kappa$ is $\lambda$-ineffable. 
\end{proof}
In the next section, we will use indescribability properties to separate the judge having a winning strategy in the game $G_n(\kappa,\lambda)$ for various finite $n$.

\begin{theorem}\label{th:weakWinningNotQuiteGeneric}
Suppose $\omega\leq\delta\leq\lambda$. Then $\kappa$ is generically $\lambda$-supercompact with wa by $\delta$-closed forcing if and only if the judge has a winning strategy in the game $wG_\delta(\kappa,\lambda)$.
\end{theorem}

\begin{proof}
Suppose that $\kappa$ is generically $\lambda$-supercompact with wa by $\delta$-closed forcing. Fix a $\delta$-closed forcing notion $\p$ such that in a generic extension by $\p$, there is a weakly amenable $V$-ultrafilter. Fix a condition $p\in\p$ and a $\p$-name $\dot U$ such that $p$ forces that $\dot U$ is a weakly amenable $V$-ultrafilter. We define a winning strategy $\sigma$ for the judge in the game $wG_{\delta}(\kappa,\lambda)$ as follows. Suppose the challenger starts by playing a $(\kappa,\lambda)$-model $M_0$. Let $p_0\leq p$ decide $U_0=\dot U\restrict M_0$, and have the judge play $U_0$. Suppose inductively that at step $\gamma>0$, we have a descending sequence of conditions $\{p_\xi\mid\xi<\gamma\}$ such that $p_\xi$ decides $U_\xi=\dot U\cap M_\xi$ and $U_\xi$ is the judge's move. Suppose the challenger plays a $(\kappa,\lambda)$-model $M_\gamma$. By $\delta$-closure, there is a condition $p_\gamma$ below $\{p_\xi\mid\xi<\gamma\}$ deciding $U_\gamma=\dot U\restrict M_\gamma$. Since the sequence of the $p_\xi$ for $\xi<\gamma$ is descending, $U_\gamma$ must extend $\bigcup_{\xi<\gamma}U_\xi$. Thus, we can have the judge play $U_\gamma$. Clearly, the judge can continue playing this way for $\delta$-many steps.

Suppose the judge has a winning strategy $\sigma$ in $wG_\delta(\kappa,\lambda)$. Let $\p=\Coll(\delta,H_{\lambda^+})$ and let $G\subseteq\p$ be $V$-generic. We work in $V[G]$. Fix an enumeration $H_{\lambda^+}=\{a_\xi\mid \xi<\delta\}.$ Let $M_0\prec H_{\lambda^+}$ be a $(\kappa,\lambda)$-model from $V$ with $a_0\in M_0$ and let $U_0$ be the response of $\sigma$ to the challenger playing $M_0$. Suppose, inductively, that we have a play according to $\sigma$ of length some $\gamma<\delta$, $\la M_0,U_0\ldots,M_{\xi}, U_{\xi},\dots\ra$,  such that $a_\xi\in M_\xi$ for $\xi<\gamma$. Note that this play is in $V$ by the $\delta$-closure of the collapse. Hence, we can let $M_\gamma$ be a $(\kappa,\lambda)$-model from $V$ with $a_\gamma\in M_\gamma$ extending the previous moves, and let $U_\gamma$ be the response of the judge according to $\sigma$. In this manner, we obtain the sequence of $(\kappa,\lambda)$-models $\{ M_\xi\mid \xi<\delta\}$ and filters $\{U_\xi\mid \xi<\delta\}$. Clearly, $\bigcup_{\xi<\delta} M_\xi=H_{\lambda^+}$ (note that $\{M_\xi\mid \xi<\delta\}\notin V$ since it witnesses $|H_{\lambda^+}|$ is collapsed). Let $U=\bigcup_{\xi<\delta} U_\xi$. This is clearly a $V$-ultrafilter. To see that $U$ is weakly amenable, let $\vec A=\{A_\eta\mid\eta<\lambda\}\in V$ be a sequence of subsets of $P_\kappa(\lambda)$. Then $\vec A\in M_\xi$ for some $\xi<\delta$, hence $U_\xi$ decides $\vec A$, but since $U_\xi\in M_{\xi+1}$, it is in $V$.

Suppose the judge has a winning strategy $\sigma$ in the game $wG_\delta(\kappa,\lambda)$. Let $\p=\Coll(\omega,H_{\lambda^+})$ and let $G\subseteq\p$ be $V$-generic. We work in $V[G]$. Fix an enumeration $$H_{\lambda^+}=\{a_n\mid n<\omega\}.$$ Let $M_0\prec H_{\lambda^+}$ be a $(\kappa,\lambda)$-model from $V$ with $a_0\in M_0$ and let $U_0$ be the response of $\sigma$ to the challenger playing $M_0$. Given a play $\la M_0,U_0\ldots,M_{n-1}, U_{n-1}\ra$ according to $\sigma$ with $a_i\in M_i$ for $i<n$, let $M_n$ be a $(\kappa,\lambda)$-model from $V$ with $a_n\in M_n$ and let $U_n$ be the response of the judge according to $\sigma$. 
In this manner, we obtain the sequence of $(\kappa,\lambda)$-models $\{ M_n\mid n<\omega\}$ and filters $\{U_n\mid n<\omega\}$. Clearly, $\bigcup_{n<\omega} M_n=H_{\lambda^+}$. Let $U=\bigcup_{n<\omega} U_n$. Let's argue that $U$ is a weakly amenable $V$-ultrafilter. Suppose $A\subseteq P_\kappa(\lambda)$ in $V$. Then $A\in M_n$ for some $n<\omega$, and so $U_n\subseteq U$ decides $A$. Suppose that $\vec A=\{A_\xi\mid\xi<\beta\}\in V$, with $\beta<\kappa$, such that all $A_\xi\in U$. Then $\vec A\in M_n$ for some $n<\omega$, and so all $A_\xi\in U_n$, which means that $\bigcap_{\xi<\beta}A_\xi\in U_n\subseteq U$. The same argument works for diagonal intersections of sequences from $V$. Finally, suppose that  $\vec A=\{A_\xi\mid\xi<\lambda\}\in V$ is a sequence of subsets of $P_\kappa(\lambda)$. Then $\vec A\in M_n$ for some $n<\omega$, and hence $U_n$ decides $\vec A$, but since $U_n\in H_{\lambda^+}$, it is in $M$, verifying weak amenability.
\end{proof}
Since the games $wG_\delta(\kappa,\lambda)$ and $G_\delta(\kappa,\lambda)$ are equivalent if $\delta>\omega$ is regular, we obtain the following immediate corollary.
\begin{corollary}\label{Thm: Equivalence for closed forcing} Suppose $\delta>\omega$ is regular. A cardinal $\kappa$ is generically $\lambda$-supercompact with wa by $\delta$-closed forcing if and only if the judge has a winning strategy in the game $G_\delta(\kappa,\lambda)$.
\end{corollary}

For the rest of this section, let us focus on the games of length $\omega$. First, by the characterization of almost generically $\lambda$-supercompact cardinals with wa in Theorem~\ref{th:almostGenericallySupercompactCompletelyIneffable},
we conclude that the judge has a winning strategy in the game $wG_\omega(\kappa,\lambda)$ if and only if $\kappa$ is completely $\lambda$-ineffable:
\begin{corollary}\label{th:weakWinningOmegaNotQuiteGeneric}
The following are equivalent for $\kappa$.
\begin{enumerate}
\item $\kappa$ is almost generically $\lambda$-supercompact with wa.
\item $\kappa$ is completely $\lambda$-ineffable.
\item The judge has a winning strategy in the game $wG_\omega(\kappa,\lambda)$.
\end{enumerate}
\end{corollary}

Not surprisingly, if the judge has a winning strategy in the game $G_\omega(\kappa,\lambda,\theta)$ for some $\theta$, then there are well-founded generic $\lambda$-supercompactness like embeddings. The next two theorems are generalization of Theorem 4.4 in \cite{NielsenWelch:games_and_Ramsey-like_cardinals}.
\begin{theorem}
If the judge has a winning strategy in the game $G_\omega(\kappa,\lambda,\theta)$, then in some set-forcing extension, there is a weakly amenable good $H_{\theta}$-ultrafilter. Thus, if the judge has a winning stategy in the game $G_\omega(\kappa,\lambda,\theta)$ for every regular $\theta\geq\lambda^+$, then $\kappa$ is generically $\lambda$-supercompact for sets with wa.
\end{theorem}
\begin{proof} Let $\sigma$ be a winning strategy for the judge in the game $G_\omega(\kappa,\lambda,\theta)$. Let us define a weakly amenable $H_\theta$-ultrafilter $U$ analogously to how we defined the $V$-ultrafilter in the proof Theorem~\ref{th:weakWinningNotQuiteGeneric} (backward direction with $\delta=\omega$) but with $H_{\lambda^+}$ replaced by $H_\theta$ and the challenger playing basic $(\kappa,\lambda)$-models elementary in $H_\theta$. Namely, we have a play according to $\sigma$ $\la M_0,U_0,\ldots,M_n,U_n,\ldots\ra$ such that $H_\theta=\bigcup_{n<\omega}M_n$ and $U=\bigcup_{n<\omega}U_n$. 

It suffices to show that $U$ is good. Suppose the ultrapower of $H_{\theta}$ by $U$ is ill-founded. Then there is a sequence $\{ g_n\mid n<\omega\}$ in $V[G]$ such that each $g_n:P_\kappa(\lambda)\to H_{\theta}$ is an element of $H_{\theta}$ and each $$A_n=\{x\in P_\kappa(\lambda)\mid g_{n+1}(x)\in g_n(x)\}\in U.$$ Choose an increasing sequence $\{i_n\mid n<\omega\}$ such that each $g_n\in M_{i_n}$.

Now we define the following tree $T$ of height $\omega$ in $V$. Elements of $T$ are sequences $$\la (N_0,U_0,f_0),\ldots, (N_{n-1},U_{n-1},f_{n-1})\ra$$ such that $\la N_0,U_0,\ldots,N_{n-1},U_{n-1}\ra$ is a subsequence of a play according to $\sigma$, for all $i< n$, $f_i\in N_i$, and  $$\{x\in P_\kappa(\lambda)\mid f_{i+1}(x)\in f_i(x)\}\in U_{i+1}.$$  The sequences are ordered by end-extension.  The tree $T$ has a branch in $V[G]$ as witnessed by the sequence $\la M_{i_0},U_{i_0},\ldots,M_{i_n},U_{i_n},\ldots\ra$. Therefore, by the absoluteness of well-foundedness, $T$ also has a branch in $V$. But that is impossible since the elements of the tree produce a play according to $\sigma$ where the ultrapower of the union model by the union ultrafilter is ill-founded.
\end{proof}
We would have liked to argue that if $\kappa$ is generically $\lambda$-supercompact for sets with wa, then the judge has a winning strategy in the game $G_\omega(\kappa,\lambda,\theta)$ for every regular $\theta\geq\lambda^+$, but we don't quite get this converse. We can only make the argument with the extra assumption that the target model $M$ of the generic embedding is contained in $V$. 

\begin{theorem}
If in a set-forcing extension, there is an elementary embedding $j:H_\theta\to M$ with $\crit(j)=\kappa$, $j(\kappa)>\lambda$, $j\image\lambda\in M$, and $M\subseteq V$, then the judge has a winning strategy in the game $G_\omega(\kappa,\lambda,\theta)$.
\end{theorem}

\begin{proof}
Suppose that a set-forcing extension $V[G]$ has an embedding $j:H_\theta\to M$ with $\crit(j)=\kappa$, $j(\kappa)>\lambda$, $j\image\lambda\in M$ and $M\subseteq V$. In particular, since $M\subseteq V$, it follows for free that $H_\theta$ and $M$ have the same subsets of $\lambda$. Let $U$ be the $H_\theta$-ultrafilter derived from $j$ using $j\image\lambda$. Since $H_\theta$ and $M$ have the same subsets of $\lambda$, $U$ is weakly amenable. 

Fix $\p$-names $\dot U$, $\dot j$, $\dot M$, and a condition $p\in\p$ forcing that $\dot U$ is a weakly amenable $H_\theta$-ultrafilter derived from the embedding $\dot j:\check H_\theta\to \dot M$ (with the above properties) and \hbox{$\dot j\image \check\lambda=\check A$}. Suppose the challenger plays a $(\kappa,\lambda)$-model $M_0\prec H_\theta$. Observe that \hbox{$j\image M_0\in M\subseteq V$}. Thus, we can choose a condition $p_0\leq\ p$ deciding that $\dot U\restrict M_0=\check U_0$ and $\dot j\restrict \check M_0=\check j_0$. Next, suppose the challenger plays $M_1\prec H_\theta$ extending $M_0$. Again, we can choose a condition $p_1\leq p_0$ such that $p_1$ forces $\dot U\restrict \check M_1=\check U_1$ and $\dot j\restrict \check M_1=\check j_1$. Since $p_1\leq p_0$, it must be the case that $U_1$ extends $U_0$ and $j_1$ extends $j_0$. The judge can continue playing this strategy for $\omega$-many steps. We need to argue that $U^*=\bigcup_{n<\omega}U_n$ produces a well-founded ultrapower of $M^*=\bigcup_{n<\omega} M_n$. So suppose the ultrapower is ill-founded. Thus, there are functions $g_n:P_\kappa(\lambda)\to H_\theta$ in $H_\theta$, for $n<\omega$, such that each $$A_n=\{x\in P_\kappa(\lambda)\mid g_{n+1}(x)\in g_n(x)\}\in U.$$ Choose an increasing sequence $\{i_n\mid n<\omega\}$ such that each $g_n\in M_{i_n}$. Thus,  each $$A_n=\{x\in P_\kappa(\lambda)\mid g_{n+1}(x)\in g_n(x)\}\in U_{i_{n+1}}.$$ It follows that $$p_{i_{n+1}}\forces \dot j(\check g_{n+1})(\check A)\in \dot j(\check g_n)(\check A).$$ Let $B_n=j_n(g_n)(A)$ for $n<\omega$. Then $\{B_n\mid n<\omega\}$ is an $\in$-descending sequence in $V$, which is the desired contradiction.

\end{proof}

\begin{remark}
Let us say that $\kappa$ is \emph{generically measurable for sets with wa} if for every sufficiently large $\theta$ some forcing extension has a weakly amenable good $H_\theta$-ultrafilter. Nielsen and Schindler \cite[4.5\&4.6]{NielsenWelch:games_and_Ramsey-like_cardinals} showed that this property is equiconsistent with the existence of a winning strategy for the judge in the one-cardinal game analogs of $G_\omega(\kappa,\lambda,\theta)$ for every sufficiently large $\theta$. Their argument relies on the fact that the extra assumption of the above theorem, that $M\subseteq V$, is obtained for free in the case $V=L$. In particular, generic measurability for sets with wa is consistent with $L$, and hence much weaker than generic measurability. But, of course, such arguments are not available in our context because consistency-wise our large cardinals are beyond the scope of inner model theory.

\end{remark}
\begin{question}
If $\kappa$ is generically $\lambda$-supercompact for sets with wa, does the judge have a winning strategy in the game $G_\omega(\kappa,\lambda,\theta)$ for every regular $\theta\geq\lambda^+$?
\end{question}
\begin{question}
Are the following equiconsistent?
\begin{enumerate}
\item There is a generically $\lambda$-supercompact with wa cardinal $\kappa$.
\item  For some regular cardinal $\theta$, in a set-forcing extension, there is an elementary embedding $j:H_\theta\to M$ with $\crit(j)=\kappa$, $j\image\lambda\in M$, $j(\kappa)>\lambda$, and $M$ transitive.
\item For some regular cardinal $\theta$, in a set-forcing extension, there is an elementary embedding $j:H_\theta\to M$ with $\crit(j)=\kappa$, $j\image\lambda\in M$, $j(\kappa)>\lambda$, $M$ transitive, and $M\subseteq V$.
\end{enumerate}
\end{question}
The next theorem is a generalization of Theorem 4.21 in \cite{NielsenWelch:games_and_Ramsey-like_cardinals}.
\begin{theorem}\label{th:winStratsGImpliesGenericEmbedding}
If the judge has a winning strategy in the game $sG_\omega(\kappa,\lambda)$, then $\kappa$ is generically $\lambda$-supercompact with wa.
\end{theorem}
\begin{proof}
Let $\sigma$ be the winning strategy for the judge in the game $sG_\omega(\kappa,\lambda)$. Let $\p=\Coll(\omega,H_{\lambda^+})$ and let $G\subseteq\p$ be $V$-generic. Let $U$ be defined as in the proof of Theorem~\ref{th:weakWinningNotQuiteGeneric}, as the union of the plays $U_n$ of the judge according to $\sigma$ in response to plays $M_n\prec H_{\lambda^+}$ of the challenger so that $\bigcup_{n<\omega} M_n=H_{\lambda^+}$. We will argue that the ultrapower of $V$ by $U$ is well-founded. If not, then in $V[G]$, there is a sequence of functions $F_n:P_\kappa(\lambda)\to V$ for $n<\omega$, with each $F_n\in V$, such that each set $$A_n=\{x\in P_\kappa(\lambda)\mid F_{n+1}(x)\in F_n(x)\}\in U.$$ Let $\dot {\vec M}$ be a $\p$-name for the sequence $\vec M=\{M_n\mid n<\omega\}$, $\dot {\vec U}$ be a $\p$-name for the sequence $\vec U=\{U_n\mid n<\omega\}$, $\dot{\vec F}$ be a $\p$-name for the sequence $\{F_n\mid n<\omega\}$, and $\dot {\vec A}$ be a $\p$-name for $\vec A=\{A_n\mid n<\omega\}$. Fix a condition $p$ forcing that $\dot {\vec U}$ are the responses of the judge to $\dot {\vec M}$ according to $\sigma$ and $\dot {\vec F}$ along with $\dot {\vec A}$ witness that the ultrapower is ill-founded.

Choose a regular $\rho$ large enough that $H_{\lambda^+}$, $\sigma$, $\p$, $\dot {\vec M}$, $\dot {\vec U}$, $\dot {\vec F}$, and $\dot {\vec A}$ are in $H_\rho$ and let $N\prec H_{\rho}$ be a countable elementary submodel containing all these sets and the condition $p$. Let $g\subseteq \p$ be $N$-generic with $p\in g$, and consider $N[g]$. Since $N\prec H_\rho$, $N$ knows what is forced by $p$, and so this must be satisfied by $N[g]$. Thus, we have that for every $n<\omega$, the sets $m_n=\dot {\vec M}(n)_g$, $u_n=\dot{\vec U}(n)_g$, $f_n=\dot{\vec F}(n)_g$, and $a_n=\dot {\vec A}(n)_g$ are  all in $N$. The model $N$ must satisfy that each finite initial segment of the sequences $\{m_n\mid n<\omega\}$ and $\{u_n\mid n<\omega\}$ forms a play according to $\sigma$, each $f_n$ is a  function on $P_\kappa(\lambda)$ and, each $a_n=\{x\in P_\kappa(\lambda)\mid f_{n+1}(x)\in f_n(x)\}\in u_k$ for some $k<\omega$. Since $N\prec H_\rho$, it is actually correct about all this. Let $m=\bigcup_{n<\omega}m_n$ and $u=\bigcup_{n<\omega}u_n$. Since $\sigma$ is a winning strategy for the game $sG(\kappa,\lambda)$, $u$ must have the countable intersection property. Thus, the ultrapower of $N$ by $u$ must be well-founded, but this is contradicted by the functions $f_n$ and sets $a_n\in u$. So we have reached a contradiction showing that the ultrapower of $V$ by $U$ must have been well-founded.
\end{proof}
\begin{question}
If $\kappa$ is generically $\lambda$-supercompact with wa, does the judge have a winning strategy in the game $sG_\omega(\kappa,\lambda,\theta)$ for every $\theta$?
\end{question}
We can get a potentially stronger generic embedding property from the assumption that the judge has a winning strategy in the game $sG_\omega(\kappa,\lambda)$. Let us say that $\kappa$ is \emph{generically $\lambda$-supercompact with wa and $\alpha$-iterability} ($0\leq\alpha\leq\omega_1$) if in a set-forcing extension of $V$ there is a weakly amenable $V$-ultrafilter with $\alpha$-many well-founded iterated ultrapowers. Thus, in particular, almost generically $\lambda$-supercompact with wa cardinals are generically $\lambda$-supercompact  with wa and 0-iterability and generically $\lambda$-supercompact cardinals with wa have 1-iterability. Where the ultrapower is taken by an external ultrafilter as in the case of $M$-ultrafilters and $(\kappa,\lambda)$-acceptable models $M$, weak amenability is required for defining successor step ultrafilters for the iteration. For such external ultrafilters, the countable intersection property suffices to ensure that all the iterated ultrapowers are well-founded. \cite{kanamori:higher} (Chapter 19) For ultrafilters that exist in the model, countable completeness is equivalent to the ultrapower being well-founded and indeed is equivalent to all the iterated ultrapowers being well-founded. For external ultrafilters, the countable intersection property is not necessary for all the iterated ultrapowers to be well-founded and there are ultrafilters with exactly $\alpha$-many iterated ultrapowers for any $\alpha<\omega_1$ \cite{gitman:welch} (for both internal and external ultrafilters, if the first $\omega_1$-many iterated ultrapowers are well-founded, then so are all iterated ultrapowers, see  \cite{kanamori:higher} (Chapter 19)).
\begin{theorem}
If the judge has a winning strategy in the game $sG_\omega(\kappa,\lambda)$, then $\kappa$ is generically $\lambda$-supercompact with wa and $\omega_1$-iterability.
\end{theorem}
\begin{proof}
Let $\sigma$ be a winning strategy for the judge in the game $sG_\omega(\kappa,\lambda)$.  Let $\p=\Coll(\omega,H_{\lambda^+})$ and let $G\subseteq\p$ be $V$-generic. Let $U$ be defined as in the proof of Theorem~\ref{th:weakWinningNotQuiteGeneric}, as the union of the plays $U_n$ of the judge according to $\sigma$ in response to plays $M_n\prec H_{\lambda^+}$ of the challenger so that $\bigcup_{n<\omega} M_n=H_{\lambda^+}$.  Suppose towards a contradiction that there is $\alpha<\omega_1$ such that the $\alpha$-th iterated ultrapower of $V$ by $U$ is ill-founded. By taking $\alpha$ to be the least such, we can assume that all the previous ultrapowers are well-founded. Let $M^{(\xi)}$ for $\xi\leq\alpha$ be the $\xi$-th iterated ultrapower of $V$ by $U$ ($V=M^{(0)}$), let $i_{\xi\eta}:M^{(\xi)}\to M^{(\eta)}$ be the corresponding embedding maps, and let $a_\xi=i_{0\xi}\image\lambda$. Following arguments from \cite{kanamori:higher} (Chapter 19), it is not difficult to see that every element of $M^{(\alpha)}$ has the form $$i_{0\alpha}(f)(a_{\xi_0},\ldots,a_{\xi_{n-1}}),$$ where $n<\omega$, $f:P_\kappa(\lambda)^n\to V$ and $\xi_0<\cdots <\xi_{n-1}<\alpha$. Thus, to witness the ill-foundedness of $M^{(\alpha)}$, there are functions $f_n:P_\kappa(\lambda)^{k_n}\to V$, with $k_n<\omega$, and sequences $(a_{\gamma^n_0},\ldots,a_{\gamma^{n}_{k_n}})$, with $\gamma^n_0<\cdots<\gamma^n_{k_n}<\alpha$,
such that the elements $$i_{0\alpha}(f_n)(a_{\gamma^n_0},\ldots,a_{\gamma^n_{k_n}})$$ form a descending membership chain. For $n<\omega$, let $k^*_n=k_n+k_{n+1}$, and let $$A_n=\{(x_0,\ldots,x_{k_n^*})\in P_\kappa(\lambda)^{k_n^*}\mid f_{n+1}(x_{j_0},\ldots,x_{j_{k_{n+1}}})\in f_{n}(x_{j'_0},\ldots,x_{j'_{k_n}})\},$$ where the indices $j_i$ and $j'_i$ indicate how the sequences $(\gamma^n_0,\ldots,\gamma^n_{k_n})$ and $(\gamma^{n+1}_0,\ldots,\gamma^{n+1}_{k_{n+1}})$ intertwine. 

As in the proof of Theorem~\ref{th:winStratsGImpliesGenericEmbedding}, we choose a countable elementary substructure $N$ of a large enough $H_\rho$ containing $\p$-names for the $M_n$, $U$, $f_n$, and $A_n$ and a condition $p$ forcing all the relevant information about them. Let $g\subseteq \p$ be $N$-generic. Then there are models $m_n\prec H_{\lambda^+}$ in $N$ whose union is $H_{\lambda^+}^N$ together with $m_n$-ultrafilters $u_n$, so that the models $m_n$ together with the filters $u_n$ form a play according to $\sigma$. It follows that $u=\bigcup_{n<\omega}u_n$ has the countable intersection property. Thus, all the iterated ultrapowers of $N$ by $u$ must be well-founded. But the interpretations of the names for the functions $f_n$ and sets $A_n$ are going to contradict this once we reach stage $\alpha$ in the iteration.
\end{proof}

\section{Games of finite length and indescribability}\label{sec:finitelength}
In this section, we separate the existence of winning strategies for the judge in the games of finite length $n$ using a generalization of indescribable cardinals to the two-cardinal context introduced by Baumgartner. As usual, we suppose that $\kappa\leq\lambda$ are regular cardinals and $\lambda^{{<}\kappa}=\lambda$.

Let's recall Baumgartner's definition of two-cardinal indescribability. Suppose $\kappa$ is a regular cardinal and  $A$ is a set of ordinals of size at least $\kappa$. First, we define a cumulative hierarchy up to $\kappa$, where we start with the set $A$ and take all subsets of size less than $\kappa$ (instead of the complete powerset) together with whatever we obtained so far at successor stages. Let $$V_0(\kappa,A)=A,$$ $$V_{\alpha+1}=P_\kappa(V_\alpha(\kappa,A))\cup V_\alpha(\kappa,A),$$ and for $\alpha$ limit, $$V_\alpha(\kappa,A)=\bigcup_{\beta<\alpha}V_\beta(\kappa,A).$$ If $\kappa$ is inaccessible, then it is easy to see that $V_\kappa\subseteq V_\kappa(\kappa,A)$ and if $A$ is transitive, then so is every $V_\alpha(\kappa,A)$. We don't expect the structure $V_\kappa(\kappa,A)$ to satisfy any significant fragment of $\ZFC$, but this won't be required in any of the arguments. Given a set $x$, we define $\kappa_x=|x\cap\kappa|$. We will make use of the following easy proposition.
\begin{proposition}\label{prop:GenCumHiUnion} $($\cite[Lemma 1.3]{abe}$)$
If $\kappa$ is inaccessible, then $V_\kappa(\kappa,\lambda)=\bigcup_{x\in P_\kappa(\lambda)}V_{\kappa_x}(\kappa_x,x)$.
\end{proposition}

 Using the generalized cumulative hierarchy, we can now define a version of $\Pi^1_n$-indescribable cardinals for the two-cardinal context. We will say that $x\in P_\kappa(\lambda)$ is \emph{good} if $\kappa_x=x\cap \kappa$. Note that the collection of all good $x$ is a club in $P_\kappa(\lambda)$.
\begin{definition} (Baumgartner)$\,$
\begin{enumerate}
\item A cardinal $\kappa$ is $\Pi^1_n$-$\lambda$-\emph{indescribable} if for every $\Pi^1_n$-formula $\varphi$ and $R\subseteq V_\kappa(\kappa,\lambda)$, if $(V_\kappa(\kappa,\lambda),\in, R)\models \varphi$, then there is a good  $x\in P_\kappa(\lambda)$ such that $$(V_{\kappa_x}(\kappa_x,x),\in, R\cap V_{\kappa_x}(\kappa_x,x))\models\varphi.$$
\item A cardinal $\kappa$ is \emph{totally $\lambda$-indescribable} if it is $\lambda$-$\Pi^1_n$-indescribable for every $n<\omega$.
\end{enumerate}
\end{definition}
Since $\Pi^1_n$-$\lambda$-indescribability is expressible by a $\Pi^1_{n+1}$-formula, it is not difficult to see that any $\Pi^1_{n+1}$-$\lambda$-indescribable cardinal is a limit of $\Pi^1_n$-$\lambda$-indescribable cardinals.
\begin{theorem}\label{th:CodyNearlySupercompactIndescribable} $($\cite[Cody]{Cody:weaklyCompactIdeal}$)$
A cardinal $\kappa$ nearly $\lambda$-supercompact cardinal if and only if it is $\Pi^1_1$-$\lambda$-indescribable. 
\end{theorem}
Note that every $(\kappa,\lambda)$-model has $V_\kappa(\kappa,\lambda)$ as an element.
\begin{lemma}\label{lem:VlambdaVjImageLambda} Suppose that $M$ is a  $(\kappa,\lambda)$-model, $U$ is an $M$-ultrafilter, and $j:M\to N$ the ultrapower embedding. Then
    \begin{enumerate}
        \item $(V_\kappa(\kappa,\lambda),\in)\cong (V_\kappa(\kappa,j\image\lambda),\in)$ as witnessed by $j\restriction V_\kappa(\kappa,\lambda)$.
        \item If $A_1,\ldots,A_n\subseteq V_\kappa(\kappa,\lambda)$ and $\varphi(X_1,\ldots,X_n)$ is a second-order formula then 
        $$(V_\kappa(\kappa,\lambda),\in)\models \varphi(A_1,\ldots,A_n)\Longleftrightarrow(V_\kappa(\kappa,j\image\lambda),\in)\models \varphi(j\image A_1,\ldots,j\image A_n).$$
    \end{enumerate}
\end{lemma}
\begin{proof}
First, let's prove (1). It suffices to argue that $j$ maps $V_\kappa(\kappa,\lambda)$ onto $V_\kappa(\kappa,j\image\lambda)$. Clearly, $j$ maps $V_0(\kappa,\lambda)=\lambda$ onto $V_0(\kappa,j\image\lambda)=j\image\lambda$. So suppose inductively that $j$ maps $V_\alpha(\kappa,\lambda)$ onto $V_\alpha(\kappa,j\image\lambda)$ for all $\alpha<\beta$. If $\beta$ is a limit ordinal, then it is immediate that $j$ maps $V_\beta(\kappa,\lambda)$ onto $V_\beta(\kappa,j\image\lambda)$. So suppose that $\beta=\alpha+1$. Fix $a\in V_{\alpha+1}(\kappa,\lambda)$. Then either $a\in V_\alpha(\kappa,\lambda)$, in which case $j(a)\in V_\alpha(\kappa,j\image\lambda)\subseteq V_{\alpha+1}(\kappa,j\image\lambda)$, or $a$ is a subset of $V_\alpha(\kappa,\lambda)$ of size less than $\kappa$, and so $j(a)=j\image a$. Since  $j\image a$ is a subset of $V_\alpha(\kappa,j\image\lambda)$ by our inductive assumption, $j\image a$ is a subset of $V_\alpha(\kappa,j\image \lambda)$ of size less than $\kappa$, and so $j(a)=j\image a\in V_{\alpha+1}(\kappa,j\image\lambda)$. Next, suppose that $b\in V_{\alpha+1}(\kappa,j\image\lambda)$. Then either $b\in V_\alpha(\kappa,j\image\lambda)$, in which case $b=j(a)$ for some $a\in V_\alpha(\kappa,\lambda)\subseteq V_{\alpha+1}(\kappa,\lambda)$, or $b$ is a subset of $V_\alpha(\kappa,j\image\lambda)$ of size less than $\kappa$. By our inductive assumption, every $x\in b$ has the form $j(y)$ for some $y\in V_\alpha(\kappa,\lambda)$. Let $a$ be the collection of these $y$. Then $a$ is a subset of $V_\alpha(\kappa,\lambda)$ of size less than $\kappa$, and so $a\in V_{\alpha+1}(\kappa,\lambda)$. Now clearly, $j(a)=j\image a=b$.

Item $(2)$ is a more general fact: if $f:A\to B$ is a $\in$-isomorphism, then for any $X_1,\ldots,X_n\subseteq A$, $f$ is also an isomorphism of the structures $(A,\in, X_1,\ldots,X_n)$ and\break $(B,\in, f\image X_1,f\image X_2,\ldots,f\image X_n)$. 
\end{proof}


\begin{theorem}\label{th:almostGenericSCtotallyIndescribable}
Every almost generically $\lambda$-supercompact with wa cardinal $\kappa$ is totally $\lambda$-indescribable.
\end{theorem}
\begin{proof}
Fix a forcing extension $V[G]$ with a weakly amenable $V$-ultrafilter $U$ and let \hbox{$j:V\to M$} be the ultrapower embedding by $U$, with $M$ not necessarily well-founded. Also, fix $n<\omega$. Let $\varphi$ be a $\Pi^1_n$-formula and let $R\subseteq V_\kappa(\kappa,\lambda)$ be such that $$(V_\kappa(\kappa,\lambda),\in, R)\models\varphi.$$  Since, by weak amenability, $H_{\lambda^+}^V=H_{\lambda^+}^M$, $M$ agrees that $(V_\kappa(\kappa,\lambda),\in, R)\models\varphi$.

Note that since $V_\kappa(\kappa,\lambda)$ has size $\lambda$ in $V$, $j\image V_\kappa(\kappa,\lambda)$ is in $M$. To see this, take any bijection $h:\lambda\rightarrow V_\kappa(\kappa,\lambda)$ in $V$, then $j\image V_\kappa(\kappa,\lambda)=j(h)\image (j\image\lambda)\in M$. Similarly, $j\image R\in M$. 

By Lemma~\ref{lem:VlambdaVjImageLambda}, $j\restrict V_\kappa(\kappa,\lambda)$ maps $(V_\kappa(\kappa,\lambda),\in,R)$ onto $(V_\kappa(\kappa,j\image\lambda),\in,j\image R)$ witnessing that the two structures are isomorphic in $M$. Hence 
$$M\models (V_\kappa(\kappa,j\image\lambda),\in, j\image R)\models\varphi.$$
To reflect that, we note that in $M$, 
$$[x\mapsto \kappa_x]_U=\kappa, \ [x\mapsto V_{\kappa_x}(\kappa_x,x)]_U=V_\kappa(\kappa,j\image\lambda), \ [x\mapsto R]_U\cap [x\mapsto V_{\kappa_x}(\kappa_x,x)]_U=j\image R.$$
The last equality holds since  $$j\image R=j(R)\cap j\image V_\kappa(\kappa,\lambda)=V_\kappa(\kappa,j\image\lambda)=V_{[x\mapsto\kappa_x]_U}([x\mapsto \kappa_x]_U,[id]_U).$$ 
By Lo\'{s} theorem,
$$\{x\in P_\kappa(\lambda)\mid (V_{\kappa_x}(\kappa_x,x),\in, R\cap V_{\kappa_x}(\kappa_x,x))\models\varphi\}\in U,$$ verifying that $\kappa$ is $\Pi^1_n$-$\lambda$-indescribable. Since $n$ was arbitrary, we have verified that $\kappa$ is totally $\lambda$-indescribable.
\end{proof}
\begin{corollary}
Every completely $\lambda$-ineffable cardinal is totally $\lambda$-indescribable.
\end{corollary}

\begin{proposition}
The property that the judge has a winning strategy in the game $G_n(\kappa,\lambda)$ for some $1\leq n<\omega$ is expressible by a $\Pi^1_{2n}$-$\lambda$-formula.
\end{proposition}
\begin{proof}
For example, the judge has a winning strategy in the game $G_2(\kappa,\lambda)$ if and only if for every basic $(\kappa,\lambda)$-model $M$, there is an $M$-ultrafilter $U$ such that for every basic $(\kappa,\lambda)$-model $N$ with $(M,U)\in N$, there is an $N$-ultrafilter $W$ extending $U$. Via coding the $(\kappa,\lambda)$-models and their ultrafilters as subsets of $\lambda$, this is clearly a $\Pi^1_4$-$\lambda$-formula over the structure $V_\kappa(\kappa,\lambda)$. The general statement is obtained similarly.
\end{proof}

\begin{theorem}
Suppose the judge has a winning strategy in the game $G_n(\kappa,\lambda)$ for some $1\leq n<\omega$. Then  $\kappa$ is $\Pi^{1}_{2n-1}$-$\lambda$-indescribable.
\end{theorem}
\begin{proof}
The proof is similar to the proof \cite[Thm. 3.4]{NielsenWelch:games_and_Ramsey-like_cardinals}, but requires some adjustments for the two-cardinal setting. Let us start with $n=1$, although it already follows from Cody's result (Theorem~\ref{th:CodyNearlySupercompactIndescribable}) and Proposition \ref{Prop: G_1 wining strategy}, we include it as we believe it includes most of the ideas from the general case. Suppose towards a contradiction that $\kappa$ is not $\Pi^1_1$-$\lambda$-indescribable.  Let $\varphi=\forall X \psi(X)$ be a $\Pi^1_1$-$\lambda$-formula such that
$$(V_\kappa(\kappa,\lambda),\in,R)\models \varphi,$$ but there is no good $x\in P_\kappa(\lambda)$ such that $(V_{\kappa_x}(\kappa_x,x),\in, V_{\kappa_x}(\kappa_x,x)\cap R)\models\varphi$.
So for every good $x\in P_\kappa(\lambda)$ there is $A_x\subseteq V_{\kappa_x}(\kappa_x,x)$, such that $$(V_{\kappa_x}(\kappa_x,x),\in,R\cap V_{\kappa_x}(\kappa_x,x))\models \neg \psi(A_x)$$ (we can let $A_x$ be anything if $x$ is not good).  Let $M$ be a $(\kappa,\lambda)$-model with $x\mapsto A_x,R\in M$. 
By the assumption that the judge has a winning strategy for the game $G_1(\kappa,\lambda)$, we can find an $M$-ultrafilter $U$ and let $j:M\to N$ be the ultrapower embedding. Let $$A=[x\mapsto A_x]_U\subseteq V_\kappa(\kappa,j\image\lambda)$$ in $N$. Then by \Los\ theorem,
$$N\models (V_\kappa(\kappa,j\image\lambda), \in , j\image R)\models \neg\psi(A).$$ 
Since $\neg\psi$ is a $\Delta_0$-formula, it is absolute and therefore in $V$, $$(V_\kappa(\kappa,j\image\lambda), \in , j\image R)\models \neg\psi(A).$$ By Lemma~\ref{lem:VlambdaVjImageLambda},  $(V_\kappa(\kappa,j\image\lambda), \in , j\image R)\cong (V_\kappa(\kappa,R),\in,R)$, contradicting our assumption that $(V_\kappa(\kappa,\lambda),\in,R)\models \varphi$.

For the general case, suppose again towards a contradiction that $\varphi$ is a $\Pi^1_{2n-1}$-$\lambda$-formula  $$\varphi=\forall X_1\exists Y_1\forall X_2\exists Y_2\ldots\forall X_{n-1}\exists Y_{n-1}\forall X_{n}\psi(X_1,Y_1,
X_2,Y_2,\ldots,X_{n-1},Y_{n-1},X_{n})$$ where $\psi$ is a $\Delta_0$-formula. Suppose that
$$(V_\kappa(\kappa,\lambda),\in,R)\models \varphi$$
and for every good $x\in P_\kappa(\lambda)$
$$(V_{\kappa_x}(\kappa_x,x),\in,R\cap V_{\kappa_x}(\kappa_x,x))\models \neg \varphi.$$
Then, for every good $x$, there is $A^1_x\subseteq V_{\kappa_x}(\kappa_x,x)$, such that 
$$(V_{\kappa_x}(\kappa_x,x),\in,R\cap V_{\kappa_x}(\kappa_x,x))\models$$ $$ \forall Y_1\exists X_2\forall Y_2\ldots\exists X_{n-1}\forall Y_{n-1}\exists X_{n}\neg\psi(A^1_x,Y_1,
X_2,Y_2,\ldots,X_{n-1},Y_{n-1},X_{n}).$$
Let $M_1$ be a $(\kappa,\lambda)$-model with $x\mapsto A^1_x,R\in M$.
Let $U_1$ be an $M_1$-ultrafilter given by the winning strategy against the challenger's move $\langle M_1 \rangle$ and let $j_1:M_1\to N_1$ be the ultrapower embedding. Let $$A^1=[x\mapsto A_x^1]_{U_1}\subseteq V_\kappa(\kappa,j_1\image\lambda)$$ in $N_1$.   Let $$B^1=j_1^{-1}[A^1]\subseteq V_\kappa(\kappa, \lambda).$$ Although, $B^1$ may not be an element of $M_1$, for every good $x\in P_\kappa(\lambda)$, $$B^1\cap V_{\kappa_x}(\kappa_x,x)\in M_1$$ because it is a set of size less than $\kappa$. Given a good $x\in P_\kappa(\lambda)$, consider the set  $$S^1_x=\{y\in P_\kappa(\lambda) \mid B^1\cap V_{\kappa_x}(\kappa_x,x)= A^1_y\cap V_{\kappa_x}(\kappa_x,x)\},$$ which is in $M_1$. In $N_1$, we have $$j_1(B^1\cap V_{\kappa_x}(\kappa_x,x))= j_1\image (B^1\cap V_{\kappa_x}(\kappa_x,x))=j_1\image B^1\cap j_1\image V_{\kappa_x}(\kappa_x,x)$$  since it is a set of size less than $\kappa$. Thus, 
\begin{equation} \label{eq1}
\begin{split}
j_1(B^1\cap V_{\kappa_x}(\kappa_x,x)) & = j_1\image B^1\cap V_{\kappa_x}(\kappa_x,j_1\image x)\\
 & = A^1\cap V_{\kappa_x}(\kappa_x,j_1\image x)\\
 &= [y\mapsto A^1_y]_{U_1}\cap V_{\kappa_x}(\kappa_x,j_1\image x).
\end{split}
\end{equation}
Hence, $S^1_x\in U_1$.

We continue with the run of the game.  By our assumption 
there is $C^1$ such that 
$$(V_\kappa(\kappa,\lambda),\in,R)\models \forall X_2\exists Y_2\ldots\forall X_{n-1}\exists Y_{n-1}\forall X_{n}\psi(B^1,C^1,
X_2,Y_2,\ldots,X_{n-1},Y_{n-1},X_{n})$$
Also, by our choice of $A^1_x$, $$(V_{\kappa_x}(\kappa_x,x),\in,R\cap V_{\kappa_x}(\kappa_x,x))\models$$
$$ \exists X_2\forall Y_2\dots\exists X_{n-1}\forall Y_{n-1}\exists X_{n}\neg\psi(A^1_x,C^1\cap V_{\kappa_x}(\kappa_x,x),
X_2,Y_2,\ldots,X_{n-1},Y_{n-1},X_{n}).$$
So let $A^2_x\subseteq V_{\kappa_x}(\kappa_x,x)$ be such that $$(V_{\kappa_x}(\kappa_x,x),\in,R\cap V_{\kappa_x}(\kappa_x,x))\models$$
$$ \forall Y_2\ldots\exists X_{n-1}\forall Y_{n-1}\exists X_{n}\neg\psi(A^1_x,C^1\cap V_{\kappa_x}(\kappa_x,x),
A^2_x,Y_2,\ldots,Y_{n-1},X_{n}).$$
Next, we find a model $M_1\subseteq M_2$, with $x\mapsto  A^2_x\in M_2$ and $C^1\in M_2$, and let $U_2$ be an $M_2$-ultrafilter given by the strategy applied to the run of the game $\langle M_1,U_1,M_2\rangle$. Let $$A^2=[x\mapsto A^2_x]_{U_2}\text{ and }B^2=j_{U_2}^{-1}[A^2].$$ Now we would like to argue that it is still the case that $j_2\image B^1=[y\mapsto A^1_y]_{U_2}$. Since $U_2\supseteq U_1$, for every good $x\in P_\kappa(\lambda)$, $S^1_x\in U_2$. It follows, by reversing the argument in the previous paragraph, that $$j_2\image B^1\cap V_{\kappa_x}(\kappa_x,j_2\image x)=[y\mapsto A^1_y]_{U_2}\cap V_{\kappa_x}(\kappa_x,j_2\image x).$$ But now since, by Proposition~\ref{prop:GenCumHiUnion}, $$V_\kappa(\kappa,j_2\image\lambda)=\bigcup_{x\in P_\kappa(j_2\image\lambda)}V_{\kappa_x}(\kappa_x,x)=\bigcup_{x\in P_\kappa(\lambda)}V_\kappa(\kappa_x,j_2\image x)$$ and $[y\mapsto A^1_y]_{U_2}\subseteq V_{\kappa}(\kappa,j_2\image\lambda)$, it follows that $$j_2\image B^1=[y\mapsto A^1_y]_{U_2}.$$

Since the judge plays according to a winning strategy in the game $G_n(\kappa,\lambda)$, we can keep going this way for $n$-many stages and obtain a winning run $\langle M_1,U_1,M_2,U_2,\ldots,M_n,U_n\rangle$ along with sets $C^1,\ldots,C^{n-1}\subseteq V_\kappa(\kappa,\lambda)$  such that $C^1,\ldots,C^{n-1}\in M_n$ and letting $$A^i=[x\mapsto A^i_x]_{U_i}\text{ and }B^i=j_{U_i}^{-1}[A^i]$$ we have:
    \begin{enumerate}
        \item     
    $(V_\kappa(\kappa,\lambda),\in,R)\models \psi(B^1,C^1,
B^2,C^2,\ldots,B^{n-1},C^{n-1},B^{n})$.
    \item For every good $x\in P_\kappa(\lambda)$, $$(V_{\kappa_x}(\kappa_x,x),\in,R\cap V_{\kappa_x}(\kappa_x,x))\models$$
    $$ \neg\psi(A^1_x,C_1\cap V_{\kappa_x}(\kappa_x,x),A^2_x,
C_2\cap V_{\kappa_x}(\kappa_x,x),\dots,A^{n-1}_x,C_{n-1}\cap V_{\kappa_x}(\kappa_x,x),A^{n}_x).$$
    \end{enumerate}
Considering the elements $[x\mapsto A^1_x]_{U_n},\ldots,[x\mapsto A^n_x]_{U_n}$ of the ultrapower $N_n$, we can argue as before that we have that $j_n\image B^i=[x\mapsto A^i_x]_{U_n}$ and $$j_n\image C^i=j(C^i)\cap V_\kappa(\kappa,j_n\image\lambda)=[x\mapsto C^i\cap V_{\kappa_x}(\kappa_x,x)]_{U_n}$$ are both elements of $N_n$.

Thus, by (2) and elementarity, we conclude that in $N_n$,
$$ (V_\kappa(\kappa,j_n\image\lambda),\in, j_n\image R)\models\neg\psi(j_n\image B^1,j_n\image 
C^1,j_n\image B^2,j_n\image
C^2,\ldots,j_n\image B^{n-1},j_n\image C^{n-1},j_n\image B^n)$$
Since $\neg\psi$ is a $\Delta_0$-formula, in $V$ we have $$ (V_\kappa(\kappa,j_n\image \lambda),\in, j_n\image R)\models\neg\psi(j_n\image B^1,j_n\image 
C^1,j_n\image B^2,j_n\image 
C^2,\ldots,j_n\image B^{n-1},j_n\image C^{n-1},j_n\image B^n)$$
Thus, by Lemma~\ref{lem:VlambdaVjImageLambda},
$$(V_\kappa(\kappa,\lambda),\in,R)\models \neg \psi(B^1,C^1,B^2,\ldots,C^{n-1},B^n)$$
contradicting $(1)$ above.
\end{proof}

\section{Precipitous ideals}\label{sec:precipitousideals} 
As usual, we suppose that $\kappa\leq\lambda$ are regular cardinals and $\lambda^{{<}\kappa}=\lambda$. We also additionally assume that $2^\lambda=\lambda^+$. In this section, we show that if the judge has a winning strategy in the game $sG^*_{\omega}(\kappa,\lambda)$ (where the judge plays weak $M$-ultrafilters), then there is a precipitous ideal on $P_\kappa(\lambda)$, and if the judge has a winning strategy in the game $sG_{\omega}(\kappa,\lambda)$, then the precipitous ideal is normal. We start with some relevant definitions.

Suppose that $\mathcal I$ is an ideal on $P_\kappa(\lambda)$. We denote by $\mathcal I^+$, the complement of $\mathcal I$ in $P(P_\kappa(\lambda))$. We shall say that $\mathcal I$ is:
\begin{enumerate}
\item \emph{fine} if it contains the complement of every set $X_\alpha=\{x\in P_\kappa(\lambda)\mid\alpha\in x\}$ for $\alpha<\lambda$,
\item $\kappa$-\emph{complete} if it is closed under unions of size less than $\kappa$,
\item \emph{normal} if it is closed under diagonal unions of length $\lambda$,
\item $\gamma$-\emph{saturated} if $P(P_\kappa(\lambda))/{\mathcal I}$ is $\gamma$-cc.
\item $\lambda$-\emph{measuring} if for every set $B\in \mathcal I^+$ and sequence $\vec A=\{A_\xi\subseteq P_\kappa(\lambda)\mid\xi<\lambda\}$, there is $\bar B\subseteq B $ in $\mathcal I^+$ such that for every $\xi<\lambda$, either $\bar B\subseteq_{\mathcal I} A_\xi$ or $\bar B\subseteq_{\mathcal I}\kappa\setminus A_\xi$.
\end{enumerate}
Let $\p_{\mathcal I}$ be the poset $(\mathcal I^+,\supseteq)$ and suppose that $G\subseteq \p_{\mathcal I}$ is $V$-generic. Combining standard arguments (see e.g. \cite{jech:settheory}[Lemma 22.13]) and our arguments in Section~\ref{sec:modelsFilters}, it is easy to see that if $\mathcal I$ is fine and $\kappa$-complete, then $G$ is a weak $V$-ultrafilter in $V[G]$ and if $\mathcal I$ is additionally normal, then $G$ is a $V$-ultrafilter in $V[G]$. We shall say that $\mathcal I$ is \emph{precipitous} if it is fine and $\kappa$-complete and $G$ is good. Given a sequence $\vec A=\{A_\xi\subseteq P_\kappa(\lambda)\mid\xi<\lambda\}$, we will say that a condition $B\in \p_{\mathcal I}$ \emph{measures} $\vec A$ if for every $\xi<\lambda$, $B\forces \check A\in \dot G$  or $B\forces\check\kappa\setminus \check A\in \dot G$. 
\begin{proposition}\label{prop:decisiveWeaklyAmenable}
Suppose that $\mathcal I$ is a fine $\kappa$-complete ideal on $P_\kappa(\lambda)$. Then the following are equivalent.
\begin{enumerate}
    \item $\mathcal I$ is $\lambda$-measuring.
    \item For every sequence $\vec A=\{A_\xi\subseteq P_\kappa(\lambda)\mid\xi<\lambda\}$, the set of conditions in $\p_{\mathcal I}$ measuring $\vec A$ is dense. 
    \item Every $V$-generic $G\subseteq \p_{\mathcal I}$ is weakly amenable.
\end{enumerate}
\end{proposition}
\begin{proof} Suppose that $(1)$ holds. Fix a sequence $\vec A=\{A_\xi\subseteq P_\kappa(\lambda)\mid\xi<\lambda\}$ and $B\in\p_{\mathcal I}$. Then there is $\bar B\subseteq B$ in $\p_{\mathcal I}$ such that for every $\xi<\lambda$, either $\bar B\subseteq_{\mathcal I} A_\xi$ or $\bar B\subseteq_{\mathcal I} \kappa\setminus A_\xi$. Clearly, $\bar B$ then measures $\vec A$ because if $\bar B$ is in a $V$-generic $G$ for $\p_{\mathcal I}$ and $\bar B\subseteq_{\mathcal I}A$, then $A\in G$ as well. Thus, $(1)$ implies $(2)$. 

Next, suppose that $(2)$ holds. Let $G\subseteq \p_{\mathcal I}$ be $V$-generic. Fix a sequence $$\vec A=\{A_\xi\subseteq P_\kappa(\lambda)\mid\xi<\lambda\}$$ in $V$. Then, by density, there is a condition $B\in G$ measuring $\vec A$. Thus, $$\{\xi<\lambda\mid A_\xi\in G\}=\{\xi<\lambda\mid B\forces \check{A}_\xi\in \dot G\},$$ and the later set exists in $V$. Thus, $(2)$ implies $(3)$.

Finally, suppose that $(3)$ holds. Again, fix $\vec A=\{A_\xi\subseteq P_\kappa(\lambda)\mid\xi<\lambda\}$. Fix a condition $\bar B\in \p_{\mathcal I}$ and let $G\subseteq \p_{\mathcal I}$ be $V$-generic with $\bar B\in G$. Since $G$ is weakly amenable, the set $D=\{\xi<\lambda\mid A_\xi\in G\}\in V$. Thus, there is a condition $B\subseteq \bar B$ in $G$ such that $$B\forces \check D=\{\xi<\check\lambda\mid \check A_\xi\in \dot G\}.$$ So clearly $B$ measures $\vec A$. But, now observe that if $B\forces\check A\in \dot G$, then $B\subseteq_{\mathcal I}A$. Thus, $(3)$ implies $(1)$.
\end{proof}
This gives the following immediate corollary.
\begin{corollary} Suppose that $\mathcal I$ is a fine $\kappa$-complete $\lambda$-measuring ideal on $P_\kappa(\lambda)$.
\begin{enumerate}
    \item Then $\kappa$ is almost generically $\lambda$-strongly compact with wa.
    \item If $\mathcal I$ is additionally precipitous, then $\kappa$ is generically $\lambda$-strongly compact with wa.
    \item If $\mathcal I$ is additionally normal, then $\kappa$ is almost generically $\lambda$-supercompact with wa.
    \item If $\mathcal I$ is additionally normal and precipitous, then $\kappa$ is generically $\lambda$-supercompact with wa.
\end{enumerate}
\end{corollary}
By Proposition~\ref{prop:nearlyStrongCompact} and Corollary~\ref{th:weakWinningOmegaNotQuiteGeneric}, we immediately get the following corollary.
\begin{corollary}
Suppose that $\mathcal I$ is a fine $\kappa$-complete $\lambda$-measuring ideal on $P_\kappa(\lambda)$.
\begin{enumerate}
    \item Then the judge has a winning strategy in the game $wG_\omega^*(\kappa,\lambda)$.
    \item If $\mathcal I$ is additionally normal, then the judge has a winning strategy in the game $wG_\omega(\kappa,\lambda)$.
\end{enumerate}
\end{corollary}
Given an ideal $\mathcal{I}$, the \emph{ideal game $G_{\mathcal I}$} is a game of perfect information played by two players Player I and Player II for $\omega$-many steps, both taking turns playing to produce a $\subseteq$-decreasing sequence of $\mathcal I^+$-sets. Player I wins if the intersection of the sequence is empty and otherwise, Player II wins.
\begin{theorem}[Jech, Lemma 22.21 \cite{jech:settheory}]
An ideal $\mathcal I$ is precipitous if and only if Player I does not have a winning strategy in the game $G_{\mathcal I}$.
\end{theorem}
For the arguments in this section, we need to introduce the following auxiliary game. A sequence $\{N_\alpha\mid\alpha<\lambda^+\}$ of $(\kappa,\lambda)$-models is said to be \emph{internally approachable} if it is elementary, continuous and for every $\alpha'<\alpha<\lambda^+$, $\{N_\eta\mid\eta<\alpha'\}\in N_\alpha$.
\begin{definition} Fix an internally approachable sequence $$\vec N=\{N_\alpha\mid\alpha<\lambda^+\}$$ of $(\kappa,\lambda)$-models such that every subset of $P(P_\kappa(\lambda))$ of size $\lambda$ is contained in some $N_\alpha$.
Let $(s/w)G^{\vec N}_\delta(\kappa,\lambda)$ be an analogous game to $(s/w)G_\delta(\kappa,\lambda)$, but where the moves of the challenger are limited to the models $N_\alpha$. Let $(s/w)G^{\vec N*}(\kappa,\lambda)$ be the non-normal version of this game.
\end{definition}
Given a winning strategy $\sigma$ for the judge in one of the filter extension games of length $\delta$, we can define the corresponding \emph{hopeless} ideal $\mathcal I(\sigma)$ to consist of $A\subseteq P_\kappa(\lambda)$ such that there is no run according to $\sigma$ whose union ultrafilter $U$ contains $A$. Given a partial run $R_\gamma$ of length $\gamma<\delta$ according to $\sigma$, we can also consider the \textit{conditional hopeless ideal}  $\mathcal I(R_\gamma,\sigma)$ consisting of $A\subseteq P_\kappa(\lambda)$ such that there is no run according to $\sigma$ extending $R_\gamma$ whose union ultrafilter $U$ contains $A$. The next lemma applies to all of the games we introduced.

\begin{lemma}\label{lem:propertiesHopelessIdeal}
Suppose that $\sigma$ is a winning strategy for the judge in a filter game of length some limit ordinal $\delta$. For the non-normal games, $\mathcal I(\sigma)$ and each $\mathcal I(R_\gamma,\sigma)$ is a fine $\kappa$-complete ideal. For the normal games, the ideals are additionally normal and $\lambda$-measuring.
\end{lemma}
\begin{proof} We will prove the result for the ideal $\mathcal I(\sigma)$ and note that the proof for the ideals $\mathcal I(R_\gamma,\sigma)$ is completely analogous. First, we argue that $\mathcal I(\sigma)$ is a $\kappa$-complete ideal. Suppose $A\in \mathcal I(\sigma)$ and $B\subseteq A$. Suppose that there is a run according to $\sigma$ with a union filter containing $B$. Then there is $\gamma<\delta$ such that $B\in U_\gamma$, the filter from the judge's move at stage $\gamma$. But then the challenger can play a model containing $A$ at the next stage $\gamma+1$. Let $U_{\gamma+1}$ be the filter chosen by the judge according to $\sigma$. Then $A\in U_{\gamma+1}\supseteq U_\gamma$, but this contradicts that $A\in \mathcal I(\sigma)$.  Suppose that $A_\xi$ for $\xi<\alpha$ are elements of $\mathcal I(\sigma)$, but $A=\bigcup_{\xi<\alpha}A_\xi\in \mathcal I(\sigma)^+$. Then there is a partial run according to $\sigma$ such that $A\in U_\gamma$, but then the challenger can put $\{A_\xi\mid\xi<\alpha\}$ into their next move, forcing the judge to put some $A_\xi$ into $U_{\gamma+1}$, her move according to $\sigma$, which is impossible since every $A_\xi\in \mathcal I(\sigma)$. Clearly, $\mathcal I(\sigma)$ is fine. Now assume that we have a normal game. Normality is shown similarly to $\kappa$-compeleteness using that the judge now has to play $M$-ultrafilters. So it remains to show  $\lambda$-measuring. Fix a sequence $\vec A=\{A_\xi\subseteq P_\kappa(\lambda)\mid\xi<\lambda\}$ and a set $B\in \mathcal I(\sigma)^+$. Thus, there is a run according to $\sigma$ with a union filter containing $B$. Then there is $\gamma<\delta$ such that $B\in U_\gamma$, the filter from the judge's move at stage $\gamma$. But then the challenger can play a model containing $\vec A$ at the next stage $\gamma+1$. Let $U_{\gamma+1}$ be the filter chosen by the judge according to $\sigma$ and for every $\xi<\lambda$, let $\bar A_\xi$ be either $A_\xi$ or its complement depending on which is in $U_{\gamma+1}$. Let $M_{\gamma+2}$ be the next move of the challenger and let $U_{\gamma+2}$ be the response of the judge according to $\sigma$. Then $\{\bar A_\xi\mid\xi<\lambda\}\in M_{\gamma+2}$ since $U_{\gamma+1}\in M_{\gamma+2}$. Thus, $\Delta_{\xi<\lambda}\bar A_\xi\in U_{\gamma+2}$ and $B\in U_{\gamma+2}$, but then $C=B\cap \Delta_{\xi<\lambda}\bar A_\xi\in U_{\gamma+2}$. It follows that $C\in \mathcal I(\sigma)^+$, $C\subseteq B$ and for every $\xi<\lambda$, $C\subseteq_{\mathcal I}A_\xi$ or $C\subseteq_{\mathcal I}\kappa\setminus A_\xi$. Since the set $B$ was arbitrary, we have verified that $\mathcal I(\sigma)$ is $\lambda$-measuring.
\end{proof}
\begin{theorem} \label{th:precipitousIdeal}
Suppose there is no fine $\kappa$-complete $\lambda^+$-saturated ideal on $P_\kappa(\lambda)$.
\begin{enumerate}
\item If the judge has a winning strategy in the game $sG^*_\omega(\kappa,\lambda)$, then there is a precipitous ideal on $P_\kappa(\lambda)$. 
\item If the judge has a winning strategy in the game $sG_\omega(\kappa,\lambda)$, then there is a $\lambda$-measuring normal precipitous ideal on $P_\kappa(\lambda)$.
\end{enumerate}
\end{theorem}
\begin{proof}
Suppose the judge has a winning strategy $\sigma'$ in the game  $sG^*_\omega(\kappa,\lambda)$. Fix an internally approachable sequence $\{N_\alpha\mid\alpha<\lambda^+\}$ of $(\kappa,\lambda)$-models and restrict the strategy $\sigma'$ to a winning strategy $\sigma$ for the judge in the game $sG^{\vec N*}_\omega(\kappa,\lambda)$. We will build the following tree $T(\sigma)$ of height $\omega$ corresponding to the strategy $\sigma$.

Since the ideal $\mathcal I(\sigma)$ is not $\lambda^+$-saturated, there is a collection $$\{A_{\la \xi\ra}\subseteq P_\kappa(\lambda)\mid \xi<\lambda^+\}$$ of sets in $\mathcal I(\sigma)^+$ such that for $\xi\neq\eta$, $A_{\la \xi\ra}\cap A_{\la \eta\ra}\in \mathcal I(\sigma)$. Choose a partial winning run $R_{\la 0\ra}$ according to $\sigma$ having the model $N_{\gamma_{\la 0\ra}}$ as the last move of the challenger and $A_{\la 0\ra}\in U_{\gamma_{\la 0\ra}}$, the response of the judge, so that we minimized $\gamma_{\la 0\ra}$. Given that we have chosen partial runs $R_{\la \xi\ra}$ for all $\xi<\xi'$, we choose a run $R_{\la \xi'\ra}$ according to $\sigma$ having $N_{\gamma_{\la \xi'\ra}}$ as the last move of the challenger and $A_{\la \xi'\ra}\in U_{\gamma_{\la \xi'\ra}}$, the response of the judge, so that we minimized so that $\gamma_{\la \xi'\ra}>\gamma_{\la \xi\ra}$ for all $\xi<\xi'$. Level 1 of our tree will now consist of the filters $U^{\la \xi\ra}=U_{\gamma_{\la \xi\ra}}$ for $\xi<\lambda^+$. 

Next, let's fix a node $U^{\la \xi\ra}$ on level 1 and show how to construct its successors. Consider now the ideal $\mathcal I(R_{\la \xi\ra},\sigma)$, which also cannot be $\lambda^+$-saturated by our assumption. Thus, there is a collection $$\{A_{\la \xi\eta\ra}\subseteq\kappa\mid \eta<\lambda^+\}$$ of sets in $\mathcal I(R_{\la \xi\ra},\sigma)^+$ such that for $\eta_1\neq\eta_2$, $A_{\la\xi\eta_1\ra}\cap A_{\la\xi\eta_2\ra}\in \mathcal I(R_{\la \xi\ra},\sigma)$. Choose a partial run $R_{\la \xi0\ra }$ extending $R_{\la \xi\ra}$ according to $\sigma$ having  $N_{\gamma_{\la \xi0\ra}}$ as the last move of the challenger and $A_{\la \xi 0\ra}\in U_{\gamma_{\la \xi 0\ra}}$, the response of the judge,  so that we have minimized $\gamma_{\la \xi 0\ra}$. Given that we have chosen partial runs $R_{\la\xi\eta\ra}$ for all $\eta<\eta'$, we choose $R_{\la \xi\eta'\ra}$ to witness that $A_{\la \xi\eta'\ra}\in \mathcal I(R_{\la \xi\ra},\sigma)^+$ with $N_{\la \gamma_{\xi\eta'}\ra}$ minimized so that $\gamma_{\la \xi\eta'\ra}>\gamma_{\la\xi\eta\ra}$ for all $\eta<\eta'$. Successors of $U^{\la \xi\ra}$ will now consist of filters $U^{\la \xi\eta\ra}=U_{\gamma_{\la \xi\eta\ra}}$ for $\eta<\lambda^+$. This also shows how we are going to define successor levels in general, thus completing the construction of the tree $T(\sigma)$.

Observe that if $M$ and $N$ are $(\kappa,\lambda)$-models such that $P(P_\kappa(\lambda))^M\subseteq N$ and $W$ is a weak $N$-ultrafilter, then clearly $U=W\cap M$ is a weak $M$-ultrafilter. It follows that given any $(\kappa,\lambda)$-model $M$ and $n<\omega$, there is some weak $N_\gamma$-ultrafilter on level $n$ of $T(\sigma)$ such that $U\restrict M$ is a weak $M$-ultrafilter. 

Now we are going to use the tree $T(\sigma)$ to define a new winning strategy $\tau$ for the judge in the game $sG^
*_\omega(\kappa,\lambda)$. Suppose the challenger plays a $(\kappa,\lambda)$-model $M_0$ as their first move. We choose the least $\gamma_{\la \xi_0\ra}$ such that $P(P_\kappa(\lambda))^{M_0}\subseteq N_{\gamma_{\la \xi_0\ra}}$. It follows that $U_0=U_{\gamma_{\la \xi_0\ra}}\cap M_0$ is a weak $M_0$-ultrafilter, and we have the judge respond with $U_0$. Next, suppose the challenger responds with $M_1$.  Let $\gamma_{\la \xi_0\xi_1\ra}$ be least such that $P(P_\kappa(\lambda))^{M_1}\subseteq N_{\gamma_{\la \xi_0\xi_1\ra}}$ and have the judge play $U_1=U_{\gamma_{\la \xi_0\xi_1\ra}}\cap M_1$. The tree will always provide a next move for the judge, so it remains to check that the union ultrafilter has the countable intersection property. But this is true because the ultrafilters $U_{\vec\xi}$ whose restrictions are played by the judge, formed a run of the game according to $\sigma$. Consider the hopeless ideal $\mathcal I(\tau)$ and crucially observe that $\mathcal I(\tau)^+=\bigcup_{\vec \xi\in (\lambda^+)^{{<}\omega}}U^{\vec\xi}$. 

By Lemma~\ref{lem:propertiesHopelessIdeal}, the ideal $\mathcal I(\tau)$ is fine and $\kappa$-complete. Finally, we argue that the ideal $\mathcal I(\tau)$ is precipitous by verifying that Player II has a winning strategy in the ideal game $G_{\mathcal I(\tau)}$. Let $X_0\in \mathcal I(\tau)^+$ be the first move of Player I. Choose $U^{\la \xi_0\ldots\xi_n\ra}$ with $X_0\in U^{\la \xi_0\ldots\xi_n\ra}$ and have Player II play $$Y_0=X_0\cap A_{\la \xi_0\ra}\cap\cdots\cap A_{\la\xi_0\ldots\xi_n\ra}.$$ Next, Player I plays $X_1\subseteq Y_0$. Choose $U^{\la \eta_0\ldots\eta_m\ra}$ with $X_1\in U^{\la \eta_0\ldots\eta_m\ra}$ and $m\geq n$ (which we can do without loss of generality). Let's argue that $\la \eta_0\ldots\eta_m\ra$ end-extends $\la \xi_0\ldots\xi_n\ra$. Suppose to the contrary that there is $i\leq n$ such that $\eta_i\neq \xi_i$ and it is least such. We have $X_1\subseteq A_{\la \xi_0\ldots\xi_{i-1}\xi_i\ra}$ by construction. Let $X=X_1\cap A_{\la \xi_0\ldots \xi_{i-1}\eta_i\ra}$. Since $X_1\in U^{\la \eta_0\ldots\eta_m\ra}$ and $ A_{\la \xi_0\ldots \xi_{i-1}\eta_i\ra}\in U^{\la \eta_0\ldots\eta_m\ra}$, it follows that $X\in U^{\la \eta_0\ldots\eta_m\ra}$. Thus, $X\in \mathcal I(\tau)^+\subseteq \mathcal I(\sigma)^+$. But then also we have $X\subseteq A_{\la \xi_0\ldots \xi_{i-1}\xi_i\ra}$ and $X\subseteq A_{\la \xi_0\ldots \xi_{i-1}\eta_i\ra}$ contradicting our assumption that $A_{\la \xi_0\ldots \xi_{i-1}\xi_i\ra}$ and $A_{\la \xi_0\ldots \xi_{i-1}\eta_i\ra}$ are incompatible modulo $\mathcal I(\sigma)$. If Player II continues to play in this fashion, we obtain at the end a branch through the tree $T(\sigma)$, which is a play according to the strategy $\sigma$. Thus, $U$, the union of all the filters on the branch, has the countable intersection property, meaning that $$\bigcap_{n<\omega}Y_n\neq\emptyset.$$

We can carry out an analogous argument starting with a winning strategy for the judge in the game $sG(\omega,\kappa)$ and constructing a tree $T$ of $N_\gamma$-ultrafilters (instead of weak $N_\gamma$-ultrafilters). We only need to observe that
if $M$ and $N$ are $(\kappa,\lambda)$-models such that $P(P_\kappa(\lambda))^M\subseteq N$ and $W$ is an $N$-ultrafilter, then $U=W\cap M$ is an $M$-ultrafilter. The key issue here is verifying normality for sequences from $M$, but it suffices to note that a sequence $\{A_\xi\mid\xi<\lambda\}$ of subsets of $P_\kappa(\lambda)$ can be coded by a subset of $P_\kappa(\lambda)$ in such a way that a basic $(\kappa,\lambda)$-model can decode it. Thus, we can verify as above that the hopeless ideal $\mathcal I(\tau)$ is precipitous. But by Lemma~\ref{lem:propertiesHopelessIdeal}, any hopeless ideal resulting from a strategy in the game $sG_\omega(\kappa,\lambda)$ is automatically normal and $\lambda$-measuring.
\end{proof}
Suppose that $\mathcal I$ is an ideal on $P_\kappa(\lambda)$ and $T$ is a tree whose elements are weak $N$-ultrafilters for some basic $(\kappa,\lambda)$-model $N$ ordered by inclusion.  We say that $T$ is \emph{dense} in $\mathcal I^+$ if $\bigcup T=\mathcal I^+$ and that $T$ is \emph{$\lambda$-measuring} if for every $\mathcal A\subseteq P(P_\kappa(\lambda))$ of size $\lambda$ and every filter $W\in T$ there is some filter $U\in T$ extending $W$ deciding all the sets in $\mathcal A$. 

\begin{theorem}\label{Thm:Ultrafilter Method}
Suppose there is no fine $\kappa$-complete $\lambda^+$-saturated ideal on $P_\kappa(\lambda)$. If the judge has a winning strategy in the game $G^*_\delta(\kappa,\lambda)$ for a regular cardinal $\delta\leq\kappa$, then there is a  tree $T$ of height $\delta$ with the following properties:
\begin{enumerate}
\item The ordering of $T$ corresponds to inclusion.
\item Successor level elements of $T$ are weak $N$-ultrafilters for some basic $(\kappa,\lambda)$-model $N$.
\item Unions are taken at limit levels (and, in particular, $T$ is $\delta$-closed).
\item $T$ is $\lambda$-measuring.
\end{enumerate}
If the judge has a winning strategy in the game $G_\delta(\kappa,\lambda)$, then there is an analogous tree consisting of $N$-ultrafilters. 
\end{theorem}
The tree $T$ is constructed as in the proof of Theorem~\ref{th:precipitousIdeal} with unions taken at limit levels\footnote{Recall that the construction of the ultrafilters in the tree $T$ follow runs of the game played according to a winning strategy and taking unions at limit steps aligns with the rules of the game.}.  As in the proof of Theorem~\ref{th:precipitousIdeal}, the tree $T$ can then be used to define a new winning strategy $\tau$ for the judge in the game $G_\delta^{*}(\kappa,\lambda)$ ($G_\delta(\kappa,\lambda)$). This is described in the next proposition, which is almost a converse to Theorem \ref{th:precipitousIdeal}.
\begin{proposition}
Suppose $\delta\leq\kappa$ is regular cardinal and there is a $\delta$-closed tree $T$ of height $\delta$ with the following properties: 
\begin{enumerate}
\item The ordering on $T$ corresponds to inclusion.
\item Successor level elements of $T$ are weak $N$-ultrafilters for some basic $(\kappa,\lambda)$-model $N$.
\item Unions are taken at limit levels.
\item $T$ is $\lambda$-measuring. 
\end{enumerate}
Then there is a winning strategy $\tau$ for the judge in the game $wG^*_\delta(\kappa,\lambda)$. If the tree $T$ consists of $N$-ultrafilters, then there is such a winning strategy for the judge in the game $wG_\delta(\kappa,\lambda)$. 
\end{proposition}
\begin{proof}
Suppose the challenger plays a $(\kappa,\lambda)$-model $M_0$. Since $T$ is $\lambda$-measuring, there is some weak $N_0$-ultrafilter $W_0\in T$ deciding $P(P_\kappa(\lambda))^{M_0}$. We have the judge play $U_0=W_0\cap P(P_\kappa(\lambda))^{M_0}$. Suppose inductively that we have a play $\la M_0,U_0,\ldots,M_\xi,U_\xi,\ldots\ra$ for $\xi<\gamma<\delta$ and a corresponding branch of $N_\xi$-ultrafilters $W_\xi$ through the tree $T$ such that $U_\xi=P(P_\kappa(\lambda))^{M_\xi}\cap W_\xi$. Suppose that the challenger plays a $(\kappa,\lambda)$-model $M_\gamma$. Since $T$ is $\delta$-closed, there is $W'_\gamma\in T$ above the branch of the $W_\xi$, and therefore we can find some $N_\gamma$-ultrafilter $W_\gamma\in T$ above $W'_\gamma$ deciding $P(P_\kappa(\lambda))^{M_\gamma}$ and have the judge play $U_\gamma=W_\gamma\cap P(P_\kappa(\lambda))^{M_\gamma}$. Clearly, this is a winning strategy for the judge. 
\end{proof}
\section{More on the measurable case}\label{sec:moreon themeasurable}
Theorem~\ref{th:precipitousIdeal} is the analogue of one of the main results of \cite{MagForZem} for games where the challenger plays certain $\kappa$-algebras ($\kappa$-complete sub-algebras of $P(\kappa)$ of size $\kappa$) and the judge responds with certain $\kappa$-complete ultrafilters for these algebras.\footnote{These games can be reformulated into equivalent games with $\kappa$-models and their ultrafilters studied by Holy and Schlicht in \cite{HolySchlicht:HierarchyRamseyLikeCardinals}.} A crucial component in the argument from \cite{MagForZem} is the ability to pass to a game where the challenger plays \emph{normal} $\kappa$-algebras \footnote{If a set $A$ is in the $\kappa$-algebra, then all sets coded on the slices of $A$ (using a fixed bijection between $\kappa$ and $\kappa\times\kappa$) are in the $\kappa$-algebra and if $\{A_\xi\mid\xi<\kappa\}$ is coded by a set $A$ in the $\kappa$-algebra, then $\Delta_{\xi<\kappa}A_\xi$ is in the $\kappa$-algebra.} and the judge plays uniform \emph{normal} ultrafilters, i.e. ultrafilters normal for sequences coded in the $\kappa$-algebra. This allows the transfer of strategies from the game where the judge plays filters to \emph{set} games where the judge plays sets determining the ultrafilters.{ In turn, the set game was used to construct a special tree of sets from which the authors of \cite{MagForZem} constructed the precipitous ideal.} This modification is unclear when considering the two-cardinal games. The way we overcame this obstacle in the previous section was to construct a tree of filters rather than a tree of sets. These filters form, similar to Dedekind cuts, exact cuts of the ultrafilter cunstructed along a run of the game. As proven in the previous section, this suffices to construct a precipitous ideal from a winning strategy of the judge. It remains open whether there is an analog of the set games which is equivalent to the filter games in the two-cardinal setting (see Question \ref{question: analog set game}).

Let us recall the games from \cite{MagForZem}, discussed in the introduction. In order to unify the notation for all the relevant variations of the games considered in \cite{MagForZem}, we will we follow the template: $$G^{\text{target cardinal}}_{\text{length of game}}(\text{challenger object type},\text{winning condition},\text{judge object type})$$
For the \textit{challenger object types} we have either $W$ (Welch) where the challenger plays a $\kappa$-algebra (see $(1)$ below) or $E^{\vec N}$ (extension) where he plays an index of a model from a fixed internally approachable sequence $\vec N$ (see $(4)$ below). For the \textit{winning condition} we have either $w$ (weak) where the judge only has to survive all the stages of the game, and $s$ (strong) which puts an extra requirement (see $(2)$ below). Finally, the \textit{judge object types} can either be a $uf$ which stands for a $\kappa$-complete ultrafilter on $\kappa$, or $nuf$, corresponding to a normal ultrafilter on $\kappa$, $s$ (set) a set which diagonalizes\footnote{We say that $Y$ diagonalizes $\mathcal{A}$, if for every $A\in\mathcal{A}$, $Y\subseteq^* A$.} an ultrafilter or $ns$ (normal set) a set which  diagonalizes a normal ultrafilter (see $(4)$ below). 
\begin{definition}\label{definition: measurable games} The games $G^\kappa_\delta(W/E^{\vec N},w/s,uf/nuf/s/ns)$ are defined as follows:
    \begin{enumerate}
        \item (weak Welch game) The game $G^\kappa_\delta(W,w,uf)$ proceeds for $\delta$-many rounds, with the challenger playing a $\subseteq$-increasing sequence of $\kappa$-algebras, $\{\mathcal{A}_i\mid \xi<\delta\}$ and the judge playing an increasing sequence of $\kappa$-complete $\mathcal{A}_\xi$-ultrafilters $U_\xi$. The judge wins the game if she can survive all stages $\xi<\delta$.\footnote{This game is equivalent to the game $wG^*_\delta(\kappa)$ from the introduction.}
        \item (strong Welch game) The game $G^\kappa_\delta(W,s,uf)$ is played analogously, but the winning condition is that $\bigcup_{\xi<\delta} U_\xi$ is a $\kappa$-complete ultrafilter on the $\kappa$-algebra generated by $\bigcup_{\xi<\delta}\mathcal{A}_\xi$.\footnote{This game is sandwiched between the game $sG^*_\delta(\kappa)$, where the union ultrafilter of the judge's moves is required to have the countable intersection property, and the game $G_{\delta+1}(\kappa)$.}
        \item (normal Welch game) The game $G^\kappa_\delta(W,w/s,nuf)$  is played analogously, but with the challenger playing normal $\kappa$-algebras and the judge responding with normal ultrafilters. The winning conditions are the same for the $w/s$ games, adding the requirement that  $\bigcup_{\xi<\delta} U_\xi$ is a normal ultrafilter in the strong version.\footnote{This game is equivalent to the game $wG(\kappa)$ from the introduction.}
        \item (set/normal set game) The game $G^\kappa_\delta(W,w/s,s/ns)$ is played analogously, but the judge plays a $\subseteq^*$-decreasing sequence $\{Y_\xi\mid \xi<\kappa\}$ of sets such that $Y_\xi$ determines an ultrafilter (with appropriate restrictions) on $\mathcal{A}_\xi$. The winning conditions are identical to the games $G^\kappa_\delta(W,w/s,uf/nuf)$.
        \item (extension game) We fix an internally approachable elementary sequence\break $\vec N=\{ N_\xi\mid \xi<\kappa^+\}$ of $\kappa$-models. The game $G^\kappa_\delta(E^{\vec N},w/s,uf/nuf/s/ns)$ is played analogously, but the challenger now plays an index of a model $\alpha_\xi<\kappa^+$, and the judge responds with an appropriate object (according to $uf/nuf/s/ns$) determining an ultrafilter on $N_{\alpha_\xi+1}$. The winning condition is again determined by $w/s$.
        
    \end{enumerate}
\end{definition}
Translating the definition of the games $G_0,G_1,G_2$ of length $\gamma$ from \cite{MagForZem}, we have:
$$G_0=G^\kappa_\gamma(W,w,uf), \ G_1=G^\kappa_\gamma(E^{\vec N},w,nuf),\ G_2=G^\kappa_\gamma(E^{\vec N},w,ns).$$
Also the games $G_0(Q_\gamma),G_1(Q_\gamma),G_2(Q_\gamma)$ from \cite{MagForZem}, correspond to the games $G^\kappa_\gamma(W,s,uf)$, $G^\kappa_\gamma(E^{\vec N},s,nuf)$, $G^\kappa_\gamma(E^{\vec N},s,ns)$ respectively. In terms of winning strategies, if $2^\kappa=\kappa^+$, the Welch games and the extension games are equivalent. 

Let us introduce another type of ultrafilters which is useful to consider. Given a\break  $\kappa$-model $M$, we shall say, that an ultrafilter $U$ on the subsets of $\kappa$ of $M$ is a \emph{weak $M$-ultrafilter} if it is $M$-$\kappa$-complete (but not necessarily $M$-normal). 
\begin{definition}
    Let $M$ be a $\kappa$-model. A weak $M$-ultrafilter $U$ is called an \emph{$M$-$p$-point} if for every sequence $\{ X_\alpha\mid \alpha<\kappa\}\in M$ with all $X_\alpha\in U$, there is a set $X\in U$ such that $X\subseteq^* X_\alpha$ for all $\alpha<\kappa$.
\end{definition}
\noindent It is easy to see, using normality, that every $M$-ultrafilter is an $M$-p-point.

We need to develop a small portion of ultrafilter theory relative to a $\kappa$-model $M$. 
\begin{definition}
    Let $U,W$ be weak $M$-ultrafilters. We say that $U$ is \emph{Rudin-Keisler below} $W$, and denote it by $U\leq_{RK}W$, if there is a function $f:\kappa\to\kappa$ in $M$ such that
    $$f_*(W)=\{X\in P^M(\kappa)\mid f^{-1}[X]\in W\}=U.$$
\end{definition}
The following lemma is an adaptation of a folklore result:
\begin{lemma}
    Let $M$ be a $\kappa$-model and $U,W$ be weak $M$-ultrafilters. Then $U\leq_{RK} W$ if and only if  there is an elementary embedding $k:M_U\to M_W$ such that $j_W=k\circ j_U$.
\end{lemma}
The Rudin-Keisler order is really a special case of the Katetov order which is defined for any two filters $F,G$ on $X,Y$ (resp.) by $F\leq_K G$ if and only if there is a function $f:Y\to X$ such that $f_*(G)=F$. The difference is that in the Rudin-Keisler ordering, we require that the witnessing function $f$ belongs to $M$.

\begin{proposition}\label{prop:p-point}
 Let $M$ be a $\kappa$-model and $U,W$ be weak $M$-ultrafilters.  If $U\leq_{RK}W$ and $W$ is an $M$-$p$-point, then $U$ is an $M$-$p$-point.
\end{proposition}
\begin{proof}
Let $f\in M$ be a function witnessing that $U\leq_{RK}W$ and $\{ X_\alpha\mid \alpha<\kappa\}\in M$ be a sequence of sets in $U$. Then $\{ f^{-1}[X_\alpha]\mid \alpha<\kappa\}\in M$ is a sequence of sets in $W$. Since $W$ is an $M$-$p$-point, there is $X\in W$ such that for every $\alpha<\kappa$, $X\subseteq^* f^{-1}[X_\alpha]$. It is now easy to see that $f[X]$ is a set in $U$ such that $f[X]\subseteq^* X_\alpha$ for every $\alpha<\kappa$.
\end{proof}
\begin{theorem}\label{th:p-point}
Suppose that  $M\subseteq N$ are $\kappa$-models and $W\in N$ is a weak $M$-ultrafilter. If there exists some $N$-$p$-point, then there is an $N$-$p$-point extending $W$. In particular, if $\kappa$ is weakly compact, then whenever $W$ is a weak $M$-ultrafilter, $M\subseteq N$ and $W\in N$, then $W$ can be extended to an $N$-$p$-point.
\end{theorem}
\begin{proof}
The second part of the theorem follows from the first part and the fact that weak compactness ensures the existence of an $N$-ultrafilter for every $\kappa$-model $N$. Let $W$ be a weak $M$-ultrafilter, $M\subseteq N$, $W\in N$ and $U$ be an $N$-p-point. Let $j_U:N\rightarrow N_U$ be the ultrapower of $N$ by $U$. First, we note that in $N$, by the closure of $M$ under ${<}\kappa$-sequences, $W$ is a collection of sets with the strong $\kappa$-intersection property i.e. the intersection of fewer than $\kappa$-many sets in $W$ is unbounded in $\kappa$. By elementarity, in $N_U$, $j_U(W)$ has the strong $j_U(\kappa)$-intersection property.  Next, let's observe that for any set $A\in N$ of size $\kappa$ in $N$, $j_U\image A\in N_U$.  In particular, $j_U\image W\in N_U$ and $N_U$ satisfies that it has size $\kappa<j_U(\kappa)$. By the  $j_U(\kappa)$-intersection property, $\bigcap j_U\image W$ is unbounded in $j_U(\kappa)$. Take any $\alpha\in \bigcap j_U\image W$ such that $\alpha\geq\kappa$. Then the weak $N$-ultrafilter $W^*$ on $\kappa$ derived from $j_U$ and $\alpha$ extends $W$. To see that $W^*$ is an $N$-$p$-point, we will show that  $W^*\leq_{RK} U$, then since $U$ is an $N$-$p$-point, Proposition~\ref{prop:p-point} can be used to conclude that $W^*$ is an $N$-$p$-point. Let $g:\kappa\to\kappa$ be a function in $N$ such that $[g]_U=\alpha$. We claim that $g_*(U)=W^*$. By maximality of ultrafilters, it suffices to prove that $g_*(U)\subseteq W^*$. Let $X\subseteq \kappa$ in $N$, then 
    $$X\in g_*(U)\Rightarrow g^{-1}[X]\in U\Rightarrow \{\nu<\kappa\mid g(\nu)\in X\}\in U\Rightarrow [g]_U\in j_U(X)\Rightarrow X\in W^*.$$

\end{proof}
\begin{remark}\label{remark: ramsey extension}$\,$
   \begin{enumerate}
       \item By a result of the second author (Proposition 2.13 in \cite{HolySchlicht:HierarchyRamseyLikeCardinals}), there is no hope of strengthening the above theorem to replace $M$-$p$-points by $M$-ultrafilters. 
       \item It is possible to generalize the above result to a suitable definition of $M$-Ramsey ultrafilters.
   \end{enumerate} 
\end{remark}
Let's now consider the games $G^\kappa_{\delta}(W/E,w/s,p)$ which are defined analogously to the games in Definition~\ref{definition: measurable games}, but where the challenger plays $\kappa$-models $M$ and the judge plays $M$-$p$-points.

In \cite{MagForZem}, the authors used a clever argument to translate a winning strategy for the judge in the game $G^\kappa_\delta(W,s,uf)$ to a winning strategy in the game $G^\kappa_\delta(W,s,nuf)$, for regular cardinals $\delta\geq\omega$. Hence, also the $p$-point game $G^\kappa_\delta(W,s,p)$ is equivalent to those games. 

However, for a general $\delta$, a winning strategy for the judge in $G^\kappa_\delta(W,s,uf)$ does not induce a winning strategy in the game $G^\kappa_\delta(W,s,nuf)$. For example, for $\delta=3$, weakly compact cardinals ensure the existence of a winning strategy in the game $G^\kappa_3(W,s,uf)$, but a winning strategy for $G^\kappa_3(W,s,nuf)$ implies that $\kappa$
 is $\Pi^1_3$-indescribable by Neilsen's theorem from \cite{NielsenWelch:games_and_Ramsey-like_cardinals}. With $M$-$p$-points we have a slightly stronger relation.
 \begin{proposition}
   Suppose that $\kappa$ is inaccessible and the judge has a winning strategy in the game $G^\kappa_\delta(W,w/s,p)$. Then she also has a winning strategy in the games $G^\kappa_{\delta+n}(W,w/s,p)$ for every $n<\omega$.
 \end{proposition}
 \begin{proof}
Suppose the challenger and the judge have already played $\delta$-many grounds according a winning strategy for the judge, with $\{M_\xi\mid\xi<\delta\}$ being the moves of the challenger and $\{U_\xi\mid\xi<\delta\}$ being the moves of the judge . Let $M_\delta$ be the $\delta$-th move of the challenger. Since $\kappa$ must be weakly compact, there is some weak $M_\delta$-ultrafilter extending $\bigcup_{\xi<\delta}U_\xi$, and so by Theorem~\ref{th:p-point}, there is an $M_\delta$-$p$-point $U_\delta$ extending $W$. We have the judge play $U_\delta$. The judge can clearly repeat this strategy for finitely many steps.
 \end{proof}
 \begin{question}
Is it true for every $\delta$, that if the judge has a  winning strategy in the game $G^\kappa_\delta(W,s,uf)$, then the judge has a winning strategy in the game $G^\kappa_\delta(W,s,p)$?
 \end{question}
Adapting our approach of a tree of ultrafilters in the one-cardinal settings recovers \cite[Thm. 1.1]{MagForZem} without having to pass through the set games.  Namely, Theorem \ref{th:precipitousIdeal} and Theorem \ref{Thm:Ultrafilter Method} in the case $\kappa=\lambda$ yield:
\begin{theorem}
    Assume $\kappa$ is inaccessible and $2^\kappa=\kappa^+$. If the judge has a winning strategy in the game $G^\kappa_\omega(W,s,uf)$, then there is a precipitous ideal on $\kappa$. 
\end{theorem}
\noindent (Note that $\kappa$-completeness is included in the standard definition of precipitousness.)

\begin{theorem}\label{Thm: measurable therem with trees}
    Assume $\kappa$ is inaccessible, $2^\kappa=\kappa^+$, and that $\kappa$ does not carry a $\kappa$-complete $\kappa^+$-saturated ideal. For any regular cardinal $\delta\leq\kappa$, if the judge has a winning strategy in the game $G^\kappa_\delta(E^{\vec N},w,uf)$, then there is a $\kappa$-complete ideal $\mathcal{I}$ on $\kappa$ with a tree $T$ of height $\delta$ such that:
    \begin{enumerate}
        \item The ordering of $T$ corresponds to inclusion.
        \item Successor level elements of $T$ are weak $N_\xi$-ultrafilters.
\item Unions are taken at limit levels (and, in particular, $T$ is $\delta$-closed).
        \item $T$ is dense, i.e. $\mathcal{I}^+=\bigcup T$.
        \item $T$ is $\kappa$-\emph{measuring}, i.e. for every filter $U\in T$ and every $\xi<\kappa^+$, there is $W$ on a successor level of $D$ such that $W\cap N_\xi$ is a weak $N_\xi$-ultafilter.
    \end{enumerate}
\end{theorem}

The theorem above is a variant of Theorem~\ref{thm: dense tree measurable intro} from \cite{MagForZem}, but weaker since it does not yield an actual dense subset of $\mathcal I^+$ that is a tree because in our construction the elements of the tree are ultrafilters and not $\mathcal I$-positive sets.

Let us call an ideal $\mathcal I$ on $\kappa$ is $\kappa$-\emph{measuring} if for every $B\in\mathcal I^+$ and sequence\break $\vec A=\{A_\xi\subseteq\kappa\mid\xi<\kappa\}$, there exists $\bar B\subseteq B$ in $\mathcal I^+$ such that for each $\xi<\kappa$, $\bar B\subseteq_{\mathcal I}A_\xi$ or $\bar B\subseteq_{\mathcal I}\kappa\setminus A_\xi$. We then say $\bar B$ \emph{measures} $\vec A$. Analogously to Proposition~\ref{prop:decisiveWeaklyAmenable}, we have that a $\kappa$-complete ideal $\mathcal I$ is $\kappa$-measuring if and only if the weak $V$-ultrafilter added by forcing with $\p_{\mathcal I}$ is weakly amenable.

Recall that in Theorem~\ref{thm: dense tree measurable intro} from \cite{MagForZem} the precipitous ideal $\mathcal{I}$ is defined by constructing the tree $D=T(\sigma)$ (of subsets of $\kappa$) from a winning strategy $\sigma$ for the judge in $G^\kappa_\delta(E^{\vec N},w,s)$, then defining a new strategy $\tau$ with $\mathcal{I}=\mathcal{I}(\tau)$. The assumption that $\kappa$ carries no $\kappa$-complete $\kappa^+$-saturated ideal is used to construct successor levels of the tree $D$ as in Theorem~\ref{th:precipitousIdeal}. Recall that given a partial run $R$ according to $\sigma$, we define the associated conditional hopeless ideal $\mathcal{I}(R,\sigma)$ as the collection of sets $X$ for which no run extending $R$ places $X$ in the ultrafilter built along the run. Given a set $X_0\in D$ and the associated fixed partial run $R_{X_0}$ of the game played according to  $\sigma$ ending with $X_0$, we use that $\mathcal{I}(R_{X_0},\sigma)^+$ contains an antichain $\mathcal{A}$ of size $\kappa^+$ to define $Y_X\subseteq^* X$ for every $X\in\mathcal{A}$, forming the successors of $X_0$ in $T(\sigma)$.

\begin{proposition}
$\mathcal{I}=\mathcal{I}(\tau)$ above is $\kappa$-measuring. Moreover, the same property holds replacing $\subseteq_{\mathcal{I}}$ by $\subseteq^*$.
\end{proposition}
\begin{proof}
    Fix some $B\in D$ and a sequence $\vec A=\{A_\xi\subseteq\kappa\mid\xi<\kappa\}$ . Find a large enough index $\gamma<\kappa^+$ so that $\vec{A}\in N_\gamma$ and find $\bar B$ above $B$ in $D$ such that $\bar B$ was obtained by diagonalizing an $N_\alpha$-ultrafilter with $\alpha>\gamma$. Then $\bar B\subseteq^* B$ and $\bar B$ measures $\vec{A}$. Then $\bar B\cap B$ is as desired.
\end{proof}

With the condition of $\kappa$-measuring, it is easy to prove the converse of Theorem \ref{thm: dense tree measurable intro} (which improves \cite[Thm. 1.4]{MagForZem}):
\begin{theorem}\label{th:decisiveIdealWinningStrategyMeasurableGame}
Suppose that there is a $\kappa$-complete (uniform normal) $\kappa$-measuring ideal $\mathcal I$ on $\kappa$ with a $\delta$-closed dense tree $D$. Then the judge has a winning strategy in the game $G^\kappa_\delta(W,w,uf)$ ($G^\kappa_\delta(W,w,nuf)$). 
\end{theorem}
\begin{proof}
Suppose the challenger starts out by playing $M_0$. By $\kappa$-measuring and density, there is $B_0\in D$ measuring $P(\kappa)^{M_0}$. Let $U_0=\{A\in P(\kappa)^{M_0}\mid B_0\subseteq_{\mathcal I} A\}$. Since $U_0$ will be the restriction of any $V$-generic filter for $\p_{\mathcal I}$ containing $B_0$, it follows that $U$ is a weak $M_0$-ultrafilter. We have the judge play $U_0$. Now suppose inductively that we have chosen a path $\{B_\xi\mid\xi<\gamma\}$ through the tree $D$ such that $B_\xi$ decides $P(\kappa)^{M_\xi}$ for the $\xi$-th move $M_\xi$ of the challenger. Suppose the challenger plays $M_\gamma$. By $\kappa$-measuring, density, and $\delta$-closure, there is $B_\gamma$ above $\bigcup_{\xi<\gamma}B_\xi$ measuring $P(\kappa)^{M_\gamma}$ and we can let $U_\gamma$, consisting of $A\in P(\kappa)^{M_\gamma}$ such that $B_\gamma\subseteq_{\mathcal I}A$, be the response of the judge. Clearly, this is a winning strategy for the judge.
\end{proof}
The above strategy can also be used to give an answer to the following question:
\begin{question}[{\cite[Q.1]{MagForZem}}]\label{Question1}
Theorem \ref{thm: dense tree measurable intro} requires
the non-existence of saturated ideals on $\kappa$. Is this hypotheses necessary?
\end{question}
Recall that successor nodes of a node $X_0$ in the tree $D=T(\sigma)$ are constructed using the fact that the hopeless ideal $\mathcal{I}^+(R_{X_0},\sigma)$ is not $\kappa^+$-saturated, where $R_{X_0}$ is a partial run according to $\sigma$ ending in the judge's move $X_0$. Although $T(\sigma)$ is not dense in $\mathcal{I}^+(R_{X_0},\sigma)$, every antichain in $T(\sigma)$ remains an antichain in $\mathcal{I}^+(R_{X_0},\sigma)$. Thus, to run the construction, we don't actually need the assumption there are no $\kappa$-complete $\kappa^+$-saturated ideals at all, only that there is a strategy $\sigma$ for the judge in $G^\kappa_\delta(E^{\vec N},w,uf)$ such that the hopeless ideals associated with $\sigma$ are not $\kappa^+$-saturated. But we will show that this statement reverses because if there exists a $\kappa$-measuring $\kappa$-complete ideal with a $\delta$-closed dense tree, then the conditional hopeless ideals from the induced strategy are likewise not $\kappa^+$-saturated. Hence, in this reformation of the hypothesis, it becomes necessary.

\begin{theorem}\label{Thm: answer to question 1}
Let $\delta\leq \kappa$ be a regular cardinal. Suppose that $\mathcal{I}$ is a $\kappa$-measuring $\kappa$-complete  ideal on $\kappa$ and there is $D\subseteq \mathcal{I}^+$ such that:
\begin{enumerate}
\item $(D,\subseteq_\mathcal{I})$ is a tree of height $\delta$.
 \item $D$ is $\delta$-closed.
    \item $D$ is dense in $\mathcal{I}^+$.
   
\end{enumerate}
Then the judge has a winning strategy $\sigma$ in the game $G^\kappa_\delta(E^{\vec N},w,uf)$ such that for every partial run of the game $R$ played according to $\sigma$, the conditional hopeless ideal $\mathcal{I}(R,\sigma)$ is not $\kappa^+$-saturated.
\end{theorem}
\begin{proof}
Let us start with some simple general lemmas.

\begin{lemma}\label{Lemma: generation of positive sets}
  If $\mathcal I$ is a proper ideal on $\kappa$ containing all the bounded sets, then $\mathcal I^+$ is not generated by less than $\kappa^+$-many sets. 
\end{lemma}
\begin{proof}
Suppose to the contrary that $\mathcal I$ is a proper ideal on $\kappa$ containing all the bounded sets and $\mathcal I^+$ is generated by the sets $\{ X_\alpha\mid \alpha<\kappa\}$. We can find (by a standard recursive diagonalization argument and using that $|\kappa\times\kappa|=\kappa$ to ensure that we go over each set $X_\alpha$ cofinally many times) $X$ such that for every $\alpha<\kappa$, $X_\alpha\not\subseteq^* X$ and $X_\alpha\not\subseteq^* \kappa\setminus X$. So neither $X$ nor $\kappa\setminus X$ are in $\mathcal I^+$, but then $\kappa\in \mathcal I$, contradicting that $\mathcal I$ is proper.
\end{proof}
Given a tree $T$, we will denote by $\mathcal L_\alpha(T)$ the $\alpha$-th level of $T$. Given a set $A\in D$, we let $D_A=\{B\in D\mid B\subseteq_{\mathcal I} A\}$, and note that $D_A$ is a tree. 
\begin{proposition}
    Suppose that there is an ideal $\mathcal I$ and a tree $D\subseteq \mathcal I^+$ satisfying $(1)$-$(3)$ as in the hypothesis of the theorem, then for every set $A\in D$ there is a level $\alpha<\delta$ such that $\mathcal{L}_\alpha(D_A)$ has size $\kappa^+$.
\end{proposition}
\begin{proof}
   Fix any set $A\in D$. Let $\mathcal I(\kappa\setminus A)$ be the ideal generated by $\mathcal I$ and $\kappa\setminus A$. Note that the tree $D_A$ is dense in $\mathcal I(\kappa\setminus A)^+$. Indeed, if $C\in\mathcal I(\kappa\setminus A)^+$, then $C\in \mathcal I^+$ and $C\subseteq_{\mathcal I} A$.  By density, there must be some $B\in D$ such that $B\subseteq_{\mathcal I} C$, and clearly  $B\in D_A$. By Lemma \ref{Lemma: generation of positive sets} applied to $\mathcal I(\kappa\setminus A)$, $|D_A|=\kappa^+$ and since the height of the tree is $\delta\leq \kappa$, there is a level of the tree $\alpha<\kappa$, such that $|\mathcal{L}_\alpha(D_A)|=\kappa^+$. \end{proof}

 Let us prove now Theorem \ref{Thm: answer to question 1}. Let $\delta$, $D$ and $\mathcal{I}$ be as in the theorem. 
\begin{corollary}
    There is a dense subtree $D'\subseteq D$ of height $\delta$ such that $D'$ satisfies $(1)$-$(3)$ from the hypothesis of the theorem and:
    \begin{enumerate}
        \item [(a)] each $A\in D'$ has $\kappa^+$-many immediate successors.
        \item [(b)] for each $\alpha<\kappa^+$, each $A\in D'$ has an immediate successor which measures $P(\kappa)^{N_\alpha}$.
    \end{enumerate} 
\end{corollary}
\begin{proof}
    We construct the tree inductively, at limit steps exploiting the closure of the original tree $D$. So given  $A\in D'$, let us define its immediate successors. By the previous lemma, there is a level $\alpha<\delta$ such that $\mathcal{L}_\alpha(D_A)$ has size $\kappa^+$. This clearly persists for levels $\beta\geq\alpha$. Since $\mathcal I$ is $\kappa$-measuring and $D$ is dense in $\mathcal I^+$, for every $\rho<\kappa^+$, there is $B_\rho\in D_A$ such that $B_\rho$ measures $P(\kappa)^{N_\rho}$. Let $\alpha(\rho)<\delta$ be the level of $B_\rho$. By the pigeonhole principle, there is $\alpha^*<\delta$ such that $\mathcal{L}_{\alpha^*}(D_A)$ contains measuring sets for cofinally many $N_\rho$ (and thus every $N_\rho$). Let us define the successors of $A$ in $D'$ to be $\mathcal{L}_{\alpha^*}(D_A)$. Note that such a tree will remain dense in $\mathcal{I}^+$ since every element in $D_A$ of level $<\alpha^*$ has an extension to $\mathcal{L}_{\alpha^*}(D_A)$. 
\end{proof}

Without loss of generality, let us  assume that already $D$ satisfies properties $(1)$-$(3)$ and $(a)$-$(b)$ above. 
We can then define a winning strategy $\sigma^D=\sigma$ for the judge in the game $G^\kappa_\delta(E^{\vec N},w,uf)$ as in the proof of Theorem~\ref{th:decisiveIdealWinningStrategyMeasurableGame}. Assume that $\mathcal{L}_0(D)=\{\kappa\}$.  We define the strategy $\sigma$ inductively. Along the way we also define auxiliary sets $A_R\in D$, for a run of the game $R$ played according to $\sigma$. At the beginning, $A_{\langle\ra}=\emptyset$. Suppose that the challenger plays $\beta_0<\kappa^+$. The strategy of the judge is to find a successor  $A_0$ of $A_{\langle\rangle}$ (minimal in a fixed well ordering of $H_\theta$), which measures $P(\kappa)^{N_{\beta_0+1}}$ (such a set $A$ exists by assumption $(b)$) and play the ultrafilter $U_0$ derived from $A_0$. Suppose that we have defined the strategy up to runs of stage $\alpha<\gamma$ and let $R=\langle \beta_0,U_0,\ldots,\beta_\xi,U_\xi,\ldots,\rangle$ be a run of the game of length $\gamma$. Associated with the run $R$ is a path  $\{A_\alpha\mid \alpha<\gamma\}$ through the tree of the sets deciding ultrafilters $U_\xi$.  Since the tree $D$ is $\delta$-closed, the sequence $\{A_\alpha\mid \alpha<\gamma\}$ has a $\subseteq_{\mathcal I}$-greatest lower bound in $D$, call it $A_R$ (note that if $\gamma=\alpha+1$, then $A_R=A_\alpha$). 
    Let $\beta_\gamma\geq \bigcup_{\alpha<\gamma}\beta_\alpha$ be the move of the challenger. The strategy for the judge is to choose  
 a successor $A_\gamma$ of $A_R$ (minimal in a fixed well ordering of $H_\theta$), which measures $P(\kappa)^{N_{\beta_\gamma+1}}$. This defines the strategy $\sigma$, which is clearly winning in the game $G^\kappa_\gamma(E^{\vec N},w,uf)$. Also note that $\mathcal{I}\subseteq \mathcal{I}(\sigma)$. This is because any ultrafilter appearing in a run of the game is generated by a set $A\in D\subseteq\mathcal{I}^+$. Hence $\mathcal{I}(\sigma)^+\subseteq \mathcal{I}^+$.
   To finish the proof of the theorem, let us prove that the conditional hopeless ideals $\mathcal{I}(R,\sigma)$ are non-$\kappa^+$-saturated for every run of the game $R$. 
\begin{proposition}
    Suppose that $R$ is a partial run in the game played according to $\sigma$. Then $\mathcal{I}(R,\sigma)$ is not $\kappa^+$-saturated. 
\end{proposition}
\begin{proof}
Fix a run $R$ of the game. The run $R$ corresponds to a path through the tree $D$ and so we can choose an element $X\in D$ above this path. Each of the $\kappa^+$-many successors of $X$ is an element of $\mathcal I(R,\sigma)$ and they are incompatible modulo $\mathcal I$ and hence also modulo $\mathcal I(R,\sigma)$.

\end{proof}
This completes the proof of the theorem.
\end{proof}
As noted in the proof of Theorem~\ref{Thm: answer to question 1}, for any tree $D$ dense in $\mathcal{I}^+$ as in the theorem, we have $\mathcal{I}\subseteq\mathcal{I}(\sigma)$.
Q.~2 from \cite{MagForZem} asks whether more can be said about the relation between $\mathcal I$ and $\mathcal I(\sigma)$—specifically, whether they are equal.
  We will show that equality can be arranged, but depends on the choice of the dense tree $D$. 
To see this, consider the subtree $D^*\subseteq D$ consisting of the sets $A_R$ where $R$ is a partial run of the game $G^\kappa_\delta(E^{\vec N},w,uf)$  played according to $\sigma$. 
\begin{proposition}
    $\,$
\begin{enumerate}
\item $D^*$ is dense in $\mathcal{I}(\sigma)$.
    \item $D^*$ satisfies $(1)$-$(3)$ and $(a)$-$(b)$.
    \item $D^*$ is dense in $\mathcal{I}$ if and only if $D=D^*$. \item $\sigma^{D^*}=\sigma$.
    \end{enumerate}
\end{proposition}
\begin{proof}
To see $(1)$, Given $X\in \mathcal{I}(\sigma)^+$, there is a run $R$ played according to $\sigma$ such that $X$ is in the ultrafilter determined by $R$. Hence one of the sets appearing in $R$ will be $\subseteq_I$-below $X$. By the definition of $A_R\in D^*$, $A_R\subseteq_I X$. $(2)$ is a routine verification.    
To see $(3)$, if $B\in D\setminus D^*$, suppose towards contradiction that some $C\in D^*$ would be a subset of $B$, then $C=A_R$ for some run $R$. Since $B\in D$, $C$ is below $B$ in the tree order and therefore $B$ has to be of the form $A_{R\restriction\xi}$.

To see $(4)$, at stage  $0$ of the game, suppose that the challenger played $\beta_0<\kappa^+$. Then the judge finds $A$ on the first level of the tree which decides $P(\kappa)^{N_{\beta_0+1}}$ and minimal with that property in the fixed well-ordering of $H_\theta$. Then $A=A_{R}$ where $R=\langle \beta_0\rangle$. Hence $A\in D^*$ and also has to be minimal. It follows that $\sigma^{D^*}(R)=A=\sigma^D(R)$.  Suppose that we have proven that $\sigma^{D^*}$ and $\sigma^{D}$ agree on up to runs of stage $\alpha<\gamma$ and let $R=\langle \beta_0,A_0,\ldots,\beta_\xi,A_\xi,\ldots,\rangle$ be a run of the game of length $\gamma$, which by our inductive assumption is the same for $\sigma^{D}$ and $\sigma^{D^*}$. Since the tree is $\gamma$-closed, the sequence $\{A_\alpha\mid \alpha<\gamma\}$ has a $\subseteq_I$-greatest lower bound in $D$, which by definition of $D^*$ also belongs to $D^*$, call it $A_R$. 
    Let $\beta_\gamma$ be the move of the challenger. Let $A_\gamma$ be a successor of $A_R$ minimal in the fixed well order of $H_\theta$ which decides $P(\kappa)^{N_{\beta_{\gamma+1}}}$ and therefore this is the same choice that $\sigma^{D^*}$ made. This shows that the strategy $\sigma^{D}(R)=\sigma^{D^*}(R)$. 
Finally, we have that $\mathcal{I}(\sigma^D)=\mathcal{I}(\sigma^{D^*})$. 

\end{proof}
\begin{corollary}
    Let $D$ be a dense subtree of $\mathcal{I}^+$ satisfying $(1)$-$(3)$ and $(a)$-$(b)$. Then $\mathcal{I}(\sigma)=\mathcal{I}$ if and only if $D=D^*$.
\end{corollary}
\section{Open problems}\label{sec:Problems}
We collect here the questions raised along the paper and several others.
\begin{question}
Do nearly $\lambda$-supercompact cardinals have a characterization in terms of the existence of generic elementary embeddings of $V$?
\end{question}
\begin{question}
Can we separate generically $\lambda$-supercompact with wa cardinals from generically $\lambda$-supercompact with wa for sets cardinals?
\end{question}
\begin{question}
    Are $\lambda$-ineffable cardinal $\Pi^1_2$-$\lambda$-indescribable?
\end{question}
In \cite{abe}, Abe proves this for a slightly stronger version of $\lambda$-ineffability, which does not seem to be equivalent to Jech's definition. We do not know whether the two definitions coincide. We also conjecture that the judge having a winning strategy in the game $G_2(\kappa,\lambda)$ implies that $\kappa$ is $\lambda$-ineffable in Abe's sense.  
\begin{question}
If $\kappa$ is generically $\lambda$-supercompact for sets with wa, does the judge have a winning strategy in the game $G_\omega(\kappa,\lambda,\theta)$ for every regular $\theta\geq\lambda^+$?
\end{question}
\begin{question}
Are the following equiconsistent?
\begin{enumerate}
\item There is a generically $\lambda$-supercompact with wa cardinal $\kappa$.
\item  For some regular cardinal $\theta$, in a set-forcing extension, there is an elementary embedding $j:H_\theta\to M$ with $\crit(j)=\kappa$, $j\image\lambda\in M$, $j(\kappa)>\lambda$, and $M$ transitive.
\item For some regular cardinal $\theta$, in a set-forcing extension, there is an elementary embedding $j:H_\theta\to M$ with $\crit(j)=\kappa$, $j\image\lambda\in M$, $j(\kappa)>\lambda$, $M$ transitive, and $M\subseteq V$.
\end{enumerate}
\end{question}
\begin{question}
If $\kappa$ is generically $\lambda$-supercompact with wa, does the judge have a winning strategy in the game $sG_\omega(\kappa,\lambda,\theta)$ for every $\theta$?
\end{question}
 \begin{question}
Is it true for every $\delta$, that if the judge has a  winning strategy in the game $G^\kappa_\delta(W,s,uf)$, then the judge has a winning strategy in the game $G^\kappa_\delta(W,s,p)$?
 \end{question}
\begin{question}\label{question: analog set game}
    Is there an equivalent version of the two-cardinal games where we play sets instead of ultrafilters? 
\end{question}
A natural approach would be to replace the $\subseteq^*$-order by inclusion modulo the filter generated by the sets $\{x\in P_\kappa(\lambda)\mid x_0\prec x\}$. It is unclear if the translation between the filter games and the games where we play set from \cite{MagForZem} generalizes to the two cardinal settings.
\newpage
\begin{center}
\begin{figure}[H]
\centering
    
    \adjustbox{scale=0.8,center}{
\begin{tikzcd}
	&& {\lambda\text{-supercompact}} & {\text{w.s in } G_{\lambda^+}(\kappa,\lambda)} \\
	{\text{gen. }\lambda\text{-s.c. with wa by }\omega_1\text{-closed}} &&& {\text{w.s in } G_{\omega_1}(\kappa,\lambda,)} \\
	{\text{gen. }\lambda\text{-s.c. with wa and }\omega_1-\text{iterability}} &&& {\text{w.s in } s G_{\omega}(\kappa,\lambda)} \\
	{\text{gen. }\lambda\text{-s.c. with wa for sets}} &&& {\text{w.s in } G_\omega(\kappa,\lambda,\theta)} \\
	{\text{al. gen. }\lambda\text{-s.c. with wa}} && {\lambda-\text{completely ineffable}} & {\text{w.s in } wG_\omega(\kappa,\lambda)} \\
	& {\lambda-\text{totally ind.}} \\
	& {\lambda-\Pi^1_{2n}-ind.} \\
	& {\lambda-\Pi^1_{2n-1}-ind.} && {\text{w.s in }G_n(\kappa,\lambda) } \\
	& {\lambda-\Pi^1_4-ind.} \\
	& {\lambda-\Pi^1_3-ind.} && {\text{w.s in } G_2(\kappa,\lambda)} \\
	& {\lambda-\Pi^1_2-ind.} & {\lambda\text{-ineffable}} \\
	& {\lambda-\Pi^1_1-ind.} & {\text{nearly }\lambda\text{-supercompact}} & {\text{w.s in } G_1(\kappa,\lambda)} \\
	{\text{al. gen. }\lambda\text{-st. com. with wa}} && {\text{nearly }\lambda\text{-strongly compact}} & {\text{w.s. in } G^*_1(\kappa,\lambda)}
	\arrow[from=1-3, to=2-1]
	\arrow[dashed, no head, from=1-4, to=1-3]
	\arrow[dashed, no head, from=2-1, to=2-4]
	\arrow[from=3-4, to=3-1]
	\arrow[from=4-4, to=4-1]
	\arrow[dashed, no head, from=5-1, to=5-3]
	\arrow[dashed, no head, from=5-3, to=5-4]
	\arrow[from=5-3, to=6-2]
	\arrow[from=7-2, to=8-4]
	\arrow[from=8-4, to=8-2]
	\arrow[from=9-2, to=10-4]
	\arrow[from=10-2, to=11-3]
	\arrow[from=10-4, to=10-2]
	\arrow[from=11-3, to=12-3]
	\arrow[dashed, no head, from=12-2, to=12-3]
	\arrow[dashed, no head, from=12-4, to=12-3]
	\arrow[dashed, no head, from=13-1, to=13-3]
	\arrow[dashed, no head, from=13-3, to=13-4]
\end{tikzcd}}  \caption{Implication table}
\end{figure}
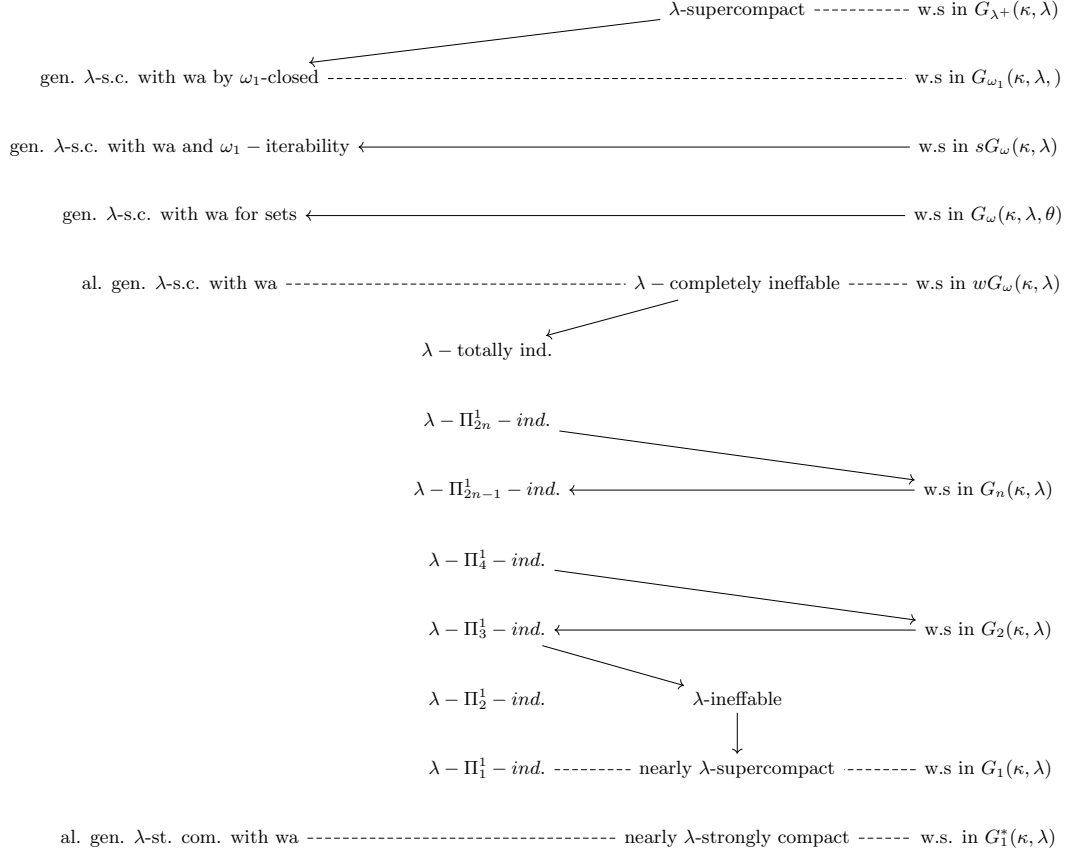
\end{center}

\bibliographystyle{amsplain}
\bibliography{games}
\end{document}